\documentclass[a4paper, bibliography=totoc]{scrartcl}
\usepackage{amsmath}
\usepackage{amssymb}
\usepackage{enumerate}
\usepackage{amsthm}
\usepackage{todonotes}
\usepackage{mathtools} 
\usepackage[UKenglish]{babel}
\usepackage[T1]{fontenc}
\usepackage[utf8]{inputenc}
\usepackage{tikz}
\usepackage{thmtools} 
\usepackage{thm-restate} 

\usepackage[mathlines]{lineno}

\usepackage{color}
\definecolor{grey}{rgb}{.7,.7,.7}
\definecolor{blue}{rgb}{0,0,.8}
\definecolor{red}{rgb}{.8,0,0}
\definecolor{green}{rgb}{0,.4,0}
\definecolor{gold}{rgb}{0.8,0.6,0.1}
\definecolor{brown}{rgb}{0.8,0.4,0.1}

\usepackage{hyperref} 
\definecolor{blue3}{rgb}{.1,.0,.4}
\hypersetup{colorlinks=true, linkcolor=red, urlcolor=blue3, citecolor=blue3, pdfpagemode=UseNone, pdfstartview=FitH, bookmarksopen=true}
\hypersetup{pdftitle=Structural properties of the set of nonrealizable densities}
\usepackage[open]{bookmark}

\declaretheorem[name=Theorem,numberwithin=section]{thm}

\newtheorem*{define*}{Definition}
\newtheorem{define}[thm]{Definition}

\newtheorem*{lemma*}{Lemma}
\newtheorem{lemma}[define]{Lemma}

\newtheorem*{algorithm*}{Algorithm}
\newtheorem{algorithm}[define]{Algorithm}

\newtheorem*{construction*}{Construction}

\newtheorem*{prop*}{Proposition}
\newtheorem{prop}[define]{Proposition}

\newtheorem*{obs*}{Observation}
\newtheorem{obs}[define]{Observation}

\newtheorem*{fact*}{Fact}

\newtheorem*{remark*}{Remark}
\newtheorem{remark}[define]{Remark}

\newtheorem*{quest*}{Question}
\newtheorem{quest}[define]{Question}

\newtheorem*{cor*}{Corollary}
\newtheorem{cor}[define]{Corollary}

\newtheorem*{conjecture*}{Conjecture}

\newtheorem*{question*}{Question}
\newtheorem{question}[define]{Question}

\newtheorem*{example*}{Example}
\newtheorem{example}[define]{Example}

\newcounter{claimcounter}[define]
\numberwithin{claimcounter}{define}
\newtheorem*{claim*}{Claim}
\newtheorem{claim}[claimcounter]{Claim}
\numberwithin{equation}{section}

\newcounter{dummycounterA} 
\newcounter{dummycounterB}

\newcommand{\R}{\mathbb{R}}

\newcommand{\Z}{\mathbb{Z}}
\newcommand{\lip}{\operatorname{Lip}}
\newcommand{\reg}{\operatorname{Reg}}
\newcommand{\dist}{\operatorname{dist}}
\newcommand{\N}{\mathbb{N}}

\newcommand{\id}{\operatorname{id}}

\newcommand{\leb}{\mathcal{L}}
\newcommand{\haus}{\mathcal{H}}
\DeclareMathOperator{\diam}{diam}
\DeclareMathOperator{\supp}{supp}
\DeclareMathOperator{\rank}{rank}
\DeclareMathOperator{\intr}{int}
\DeclareMathOperator*{\argmin}{arg\,min}

\newcommand{\sign}{\operatorname{sign}}

\newcommand{\jac}{\operatorname{Jac}}
\newcommand{\mb}[1]{\mathbf{#1}}

\newcommand{\mc}[1]{\mathcal{#1}}

\newcommand{\niceint}[3]{\int_{#1}{#2\,\mathrm{d}#3}}
\newcommand{\lint}[2]{\niceint{#1}{#2}{\leb}}

\newcommand{\abs}[1]{\left|#1\right|}
\newcommand{\lnorm}[2]{\left\|#2\right\|_#1}

\newcommand{\opnorm}[1]{\lnorm{\text{op}}{#1}}
\newcommand{\Sq}{\mathcal{S}}
\newcommand{\Reg}{\reg} 

\newcommand{\E}{\mathcal{E}}
\newcommand{\F}{\mathcal{F}}
\newcommand{\G}{\mathcal{G}}
\newcommand{\set}[1]{\left\{#1\right\}}
\newcommand{\cl}[1]{\overline{#1}}
\newcommand{\rest}[2]{#1|_{#2}}
\DeclarePairedDelimiter{\ceil}{\lceil}{\rceil}
\DeclarePairedDelimiter{\floor}{\lfloor}{\rfloor}
\DeclarePairedDelimiter{\br}{(}{)}
\DeclarePairedDelimiter{\bs}{[}{]}
\def\Xint#1{\mathchoice
	{\XXint\displaystyle\textstyle{#1}}%
	{\XXint\textstyle\scriptstyle{#1}}%
	{\XXint\scriptstyle\scriptscriptstyle{#1}}%
	{\XXint\scriptscriptstyle\scriptscriptstyle{#1}}%
	\!\int}
\def\XXint#1#2#3{{\setbox0=\hbox{$#1{#2#3}{\int}$ }
		\vcenter{\hbox{$#2#3$ }}\kern-.6\wd0}}

\def\dashint{\Xint-}

\makeatletter
\newcommand{\mylabel}[2]{#2\def\@currentlabel{#2}\label{#1}}
\makeatother

\newcommand{\pushforwardeq}{\@empty}
\newcommand{\unitinterval}{[0,1]} 
\newcommand{\bdependsonreg}{$b$ depending only on a constant quantifying Lipschitz regularity of $f$ (this is made precise in Section~\ref{section:regular})}

\title{Mapping $n$ grid points onto a square forces an arbitrarily large Lipschitz~constant}
\author{Michael Dymond \and Vojt\v ech Kalu\v za\thanks{V.K.\ was partially supported by the project GAČR 16-01602Y. He was also supported by the project FWF P23628-N18 during his stays at Universität Innsbruck.} \and Eva Kopeck\'a\thanks{E.K.\ was partially supported by the project FWF P23628-N18.}}

\begin{document}
\maketitle
\begin{abstract}
	We prove that the regular $n\times n$~square grid of points in the integer lattice $\Z^{2}$ cannot be recovered from an arbitrary $n^{2}$-element subset of $\Z^{2}$ via a mapping with prescribed Lipschitz constant (independent of $n$). This answers negatively a question of Feige from 2002. Our resolution of Feige's question takes place largely in a continuous setting and is based on some new results for Lipschitz mappings falling into two broad areas of interest, which we study independently. Firstly the present work contains a detailed investigation of \emph{Lipschitz regular mappings} on Euclidean spaces, with emphasis on their bilipschitz decomposability in a sense comparable to that of the well known result of Jones. Secondly, we build on work of Burago and Kleiner and McMullen on \emph{non-realisable~densities}. We verify the existence, and further prevalence, of strongly non-realisable densities inside spaces of continuous functions. 
\end{abstract}
\section{Introduction}\label{section:introduction}
The main objective of this paper is to answer a question of Feige originating in the 1990s, which asks whether every subset of the integer lattice $\Z^{2}$ with cardinality $n^{2}$ for some $n\in\N$ can be mapped via a bijection $f$ onto the regular $n\times n$ grid $\left\{1,\ldots,n\right\}^{2}$ in such a way that the Lipschitz constant of $f$ may be bounded above independently of $n$:
\begin{quest}[`Feige's Question']
\label{q:Feige_orig}
Is there a constant $L>0$ such that for every $n\in\N$ and every set $S\subset \Z^{2}$ with cardinality $\left|S\right|=n^{2}$ there is a bijection $f\colon S\to \left\{1,\ldots,n\right\}^{2}$ with Lipschitz constant $\lip(f)\leq L$?
\end{quest}
Feige's question appears to have first come to wider attention at the workshop `Discrete Metric Spaces and their Algorithmic Applications' held in Haifa in March 2002~\cite{Haifa} and appeared in the technical report~\cite[Question~2.12]{Matousek_open}, which lists many open problems arising from this meeting and, in several cases, their subsequent solutions. However, to the best of the authors' knowledge, there has been almost no progress on Feige's question prior to the present work. The second author worked on Feige's question in his master thesis~\cite{VKM}; some ideas contained there were helpful in the development of the present work.

Feige's motivation to ask the question stemmed from his work on the so-called `graph bandwidth' problem. In this problem the goal is to find a bijection $l\colon V\to\set{1,\ldots,n}$ for a given $n$-vertex graph $(V,E)$ with the `bandwidth' as small as possible, that is, minimising the quantity $\max_{uv\in E}\abs{l(u)-l(v)}$. It is known that finding the optimal solution to this problem is \textbf{NP}-hard \cite{bandwidth_hard}. In~\cite{Feige_approx} Feige designated a randomised approximation algorithm that produces a solution with bandwidth larger than the optimum by a factor polylogarithmic in $n$. His algorithm can also be adapted to a generalisation of the bandwidth problem to two dimensions; however, in this case it does not produce a bijection between the set of vertices $V$ and the grid $\set{1,\ldots, \sqrt{n}}^2$, but rather an injection into a larger grid leaving some of the grid points unused~\cite{Feige_email}. Feige then asked whether one can map such a set bijectively onto the grid $\set{1,\ldots, \sqrt{n}}^2$ without increasing the bandwidth `too much'. 

We will prove that the answer to Feige's question is negative in all dimensions $d\geq 2$:
\begin{thm}\label{thm:mainresult}
Let $\F_{n}$ denote the collection of all subsets $S\subset \Z^{d}$ with $\left|S\right|=n^{d}$ and
\begin{linenomath} 
\begin{equation*}
 L_{S}:=\inf\left\{\lip(f)\colon \quad f\colon S\to\left\{1,\ldots,n\right\}^{d}\text{ is a bijection}\right\}
 \end{equation*}
\end{linenomath}
 for each $S\in\F_{n}$. Then the sequence
\begin{linenomath} 
\begin{equation*}
 C_{n}:=\sup\left\{L_{S}\colon S\in\F_{n}\right\},\qquad n\in\N
 \end{equation*}
\end{linenomath}
 is unbounded.	
\end{thm}
Moreover, we show how almost any positive continuous function on the unit cube $[0,1]^{d}$ can be used to construct a sequence of sets $S_{n}\in\F_{n}$ verifying Theorem~\ref{thm:mainresult}. It is then natural to ask how fast $C_{n}$ grows. Whilst it is straightforward to verify that $C_{n}\leq \sqrt{d}n$, the results of the present article do not permit finer estimates of $C_{n}$, either from above or below. The question of the rate of growth of the sequence $(C_{n})$ will be an interesting topic for future research.
 
In order to answer Feige's question, we adapt a technique developed independently by Burago and Kleiner~\cite{BK1} and McMullen~\cite{McM}, which translates questions about Lipschitz mappings on discrete sets into a continuous setting. The papers \cite{BK1} and \cite{McM} present a negative answer to the question of whether every two separated nets in the plane are bilipschitz equivalent, or put differently, whether every separated net $M\subseteq \Z^{2}$ admits a bilipschitz bijection $f\colon M\to\Z^{2}$. This `discrete' question is shown to be equivalent to the following `continuous' one: Does there exist for every measurable function $\rho\colon [0,1]^{2}\to(0,\infty)$ with $0<\inf\rho<\sup\rho<\infty$ a bilipschitz mapping $f\colon[0,1]^{2}\to\R^{2}$ with
\begin{equation}\label{eq:norlsble}
\rho=\jac(f)\qquad\text{ almost everywhere.}
\end{equation}
Burago and Kleiner and McMullen answer both questions negatively by constructing non-realisable density functions, i.e.\ functions $\rho$ for which the equation~\eqref{eq:norlsble} has no bilipschitz solutions $f\colon[0,1]^{2}\to\R^{2}$.
The idea to approach Feige's question by the study of the constructions of Burago and Kleiner~\cite{BK1} and McMullen~\cite{McM} comes from Jiří Matoušek, who shared it with the second author while supervising his master thesis~\cite{VKM}.

By encoding measurable functions $\rho\colon [0,1]^{d}\to(0,\infty)$ as sequences of discrete sets, we will show that a negative answer to Feige's question is implied by the existence of a special type of non-realisable density function. We require here a function non-realisable in a stronger sense than that provided by Burago and Kleiner and McMullen: Our density $\rho$ must exclude solutions to a natural generalisation of equation \eqref{eq:norlsble}, not only in the class of bilipschitz mappings, but in the wider class of \emph{Lipschitz regular} mappings. Lipschitz regular mappings can be thought of as `non-degenerate' Lipschitz mappings and were introduced by David in \cite{David88}; we give a more detailed introduction to this class in Section~\ref{section:regular}. Given a measurable function $\rho\colon [0,1]^{d}\to[0,\infty)$ and a mapping $f\colon [0,1]^{d}\to\R^{d}$ we associate measures $\rho\leb$ and $f_{\sharp}\rho\leb$ on $[0,1]^{d}$ and $f([0,1]^{d})$ respectively, defined by
\begin{linenomath}
\begin{align*}
\rho\leb(A):&=\lint{A}{\rho}\quad \text{for all measurable $A\subseteq[0,1]^{d}$, and}\qquad\\ f_{\sharp}\rho\leb(T):&=\rho\leb(f^{-1}(T)) \quad\text{for all measurable } T\subseteq f([0,1]^{d}).
\end{align*}
\end{linenomath}
We verify the existence of density functions $\rho$ on $[0,1]^{d}$ for which
\begin{equation}\label{eq:nonrlsble2}
f_{\sharp}\rho\leb=\leb|_{f([0,1]^{d})},
\end{equation}
\renewcommand{\pushforwardeq}{\eqref{eq:nonrlsble2}\ }
a generalisation of equation~\eqref{eq:norlsble} for non-bilipschitz mappings $f$, has no Lipschitz regular solutions $f\colon [0,1]^{d}\to\R^{d}$. In fact, we prove a stronger statement:
\begin{restatable*}{thm}{thmregrlzprs}\label{thm:regrlzprs}
	Let
	\begin{linenomath}
	\begin{equation*}
	\E:=\left\{\rho\in C(\unitinterval^{d},\R)\colon \text{\pushforwardeq admits a Lipschitz regular solution $f\colon \unitinterval^{d}\to\R^{d}$}\right\}.
	\end{equation*}
	\end{linenomath}
	Then $\E$ is a $\sigma$-porous subset of $C(\unitinterval^{d},\R)$.
\end{restatable*}
In the above, $C([0,1]^{d},\R)$ denotes the Banach space of continuous real-valued functions on the unit cube $[0,1]^{d}$. Further, the term $\sigma$-porous refers to a class of negligible subsets of complete metric spaces. This notion is discussed in greater detail at the beginning of Section~\ref{section:realizability}. 

The fact that Lipschitz regular mappings need not be bijective presents significant additional difficulties in constructing non-realisable density functions for this class. To produce bilipschitz non-realisable densities on $[0,1]^{d}$ it suffices to take any non-empty open subset $U\subseteq [0,1]^{d}$ and to define a `badly behaving' function $\rho$ there in the style of McMullen or Burago and Kleiner. This function can then be extended arbitrarily to the whole of $[0,1]^{d}$. However, for a Lipschitz regular mapping $f\colon [0,1]^{d}\to\R^{d}$ and $T\subseteq f([0,1]^{d})$ the pre-image $f^{-1}(T)$ may consist of several `pieces' whose images overlap in $T$. Therefore any local `bad behaviour' of $\rho$ on one piece of $f^{-1}(T)$ may be compensated for by its values on other pieces, allowing $f$ to satisfy equation~\eqref{eq:nonrlsble2}. So excluding Lipschitz regular solutions of \eqref{eq:nonrlsble2} crucially requires global control of the density $\rho$.

We overcome these problems, in part, by using a bilipschitz decomposition result for Lipschitz regular mappings, which derives from a result of Bonk and Kleiner~\cite[Theorem~3.4]{bonk_kleiner2002}:
\begin{restatable*}{thm}{thmbilipdecomp}\label{p:preq2}
	Let $U\subseteq\R^d$ be open and $f\colon \cl{U}\to\R^{d}$ be a Lipschitz regular mapping. Then there exist pairwise disjoint, open sets $(A_{n})_{n=1}^{\infty}$ in $\cl{U}$ such that $\bigcup_{n=1}^{\infty}A_{n}$ is dense in $\cl{U}$ and $f|_{A_{n}}$ is bilipschitz for each $n\in\N$ with lower bilipschitz constant \bdependsonreg.
\end{restatable*}
\renewcommand{\bdependsonreg}{$b=b(\reg(f))$}
This result turns out to be very useful. Given a Lipschitz regular mapping $f\colon [0,1]^{d}\to\R^{d}$, it allows us to carefully choose the set $T\subseteq f([0,1]^{d})$, referred to in the discussion above so that its pre-image $f^{-1}(T)$ decomposes precisely as a finite union of open sets on which $f$ is bilipschitz; see Proposition~\ref{p:regular_bilip_decomp}. We then extend the existing techniques for constructing bilipschitz non-realisable densities to fit this more complex situation, where there are multiple bilipschitz mappings in play instead of just one.

\subsection*{Structure of the paper.}
\addcontentsline{toc}{subsection}{Structure of the paper.}
In Section~\ref{section:regular} we investigate Lipschitz regular mappings on Euclidean spaces, in particular proving Theorem~\ref{p:preq2}.  
Besides the material necessary for the proof of Theorem~\ref{p:preq2}, and subsequently for the proof of Proposition~\ref{p:regular_bilip_decomp}, which we use to answer Question~\ref{q:Feige_orig}, Section~\ref{section:regular} also contains a discussion of optimality and limits of Theorem~\ref{p:preq2}.

The groundwork for our verification of non-realisable densities is laid in Section~\ref{section:geometric}, where we develop a construction of Burago and Kleiner in order to derive certain powerful properties of bilipschitz Jacobians; see Lemma~\ref{lemma:geometric}. This section is the most technical part of the present paper; readers willing to accept Lemma~\ref{lemma:geometric} may safely skip this Section during the first pass-over, and perhaps, return to it only later.
Both Sections~\ref{section:regular} and \ref{section:geometric} can be read independently of the rest of the paper.

Section~\ref{section:realizability} is devoted to the proof of existence of non-realisable densities and Theorem~\ref{thm:regrlzprs}. The only material from the previous sections needed in Section~\ref{section:realizability} is Proposition~\ref{p:regular_bilip_decomp} and Lemma~\ref{lemma:geometric}.
Having gathered together all the necessary ingredients, we resolve Feige's question in Section~\ref{section:feige}.

In an effort to make the paper easier to read and to bring the more important ideas of the arguments to the forefront, we postpone the proofs of technical lemmas, or formal verifications of intuitive statements until an appendix. Many of the results in the appendix can be treated as exercises and a reader interested in the core argument can safely skip them. For the sake of completeness we include all proofs. The appendix is divided into four subsections, each corresponding to a section of the paper. Where specific notation is introduced within a given section, we adopt the same notation in its appendix.

\subsection*{Notation.}
\addcontentsline{toc}{subsection}{Notation.}
We conclude this introduction with a summary of the notation and key definitions common to all sections of the paper:

\paragraph{Sets and measures.} For $i\in\N$ we write $[i]$ for the set $\left\{1,2,\ldots,i\right\}$. Given a finite set $F$, we let $\left|F\right|$ denote the cardinality of $F$. For an infinite set $F$ we put $\abs{F}$ equal to $\infty$. Throughout the paper $I$ will denote the unit interval $I:=[0,1]$. The closure, interior and boundary of a set $A$ are written as $\overline{A}$, $\intr A$ and $\partial A$ respectively. We adopt the convention of Mattila~\cite{Mattila} and do not distinguish between outer measures and measures. The symbol $\leb$ will refer to the Lebesgue measure and we write a.e.\ instead of `almost everywhere' or `almost every' with respect to $\leb$. We let
\begin{linenomath}
\begin{equation*}
\dashint_{S}\rho:=\frac{1}{\leb(S)}\int_{S}\rho\, d\leb
\end{equation*}
\end{linenomath}
denote the average value of a measurable, real-valued function $\rho$ on a measurable set $S$ of positive Lebesgue measure. We will use the term \emph{density} to refer to a non-negative, measurable real valued function. Thus, each density $\rho\colon I^{d}\to[0,\infty)$ can be associated to a measure $\rho\leb$ on $I^{d}$ as defined earlier in this introduction.

\paragraph{Norms and balls.} We write $\lnorm{2}{-}$ for the Euclidean norm and $\lnorm{\infty}{-}$ for both the supremum norm and (briefly) the $L^{\infty}$ norm. In addition, the symbol $\opnorm{T}$ will represent the operator norm of a linear mapping $T$. An open ball with centre $a$ and radius $r$ will be denoted by $B(a,r)$. Most of the time balls will be in Euclidean spaces, but we sometimes consider balls in general metric spaces including spaces of functions. It will be clear from the context which norm or metric is relevant. Occasionally we will extend this notation to denote neighbourhoods of sets: Given a set $A\subseteq\R^{d}$ and $r>0$ we let $B(A,r):=\bigcup_{a\in A}B(a,r)$. To denote the closure of a ball or set neighbourhood we write $\overline{B}$ instead of $B$.

\paragraph{Mappings.} For $L\geq 1$, we call a mapping $f\colon\R^{d}\to\R^{n}$ \emph{$L$-Lipschitz} if 
	\begin{linenomath}
\begin{equation*}
\lnorm{2}{f(y)-f(x)}\leq L\lnorm{2}{y-x}, \qquad  y,x\in \R^{d}.
\end{equation*}
	\end{linenomath}
If, in addition, there is a constant $0<b\leq L$ such that
	\begin{linenomath}
\begin{equation*}
\lnorm{2}{f(y)-f(x)}\geq b\lnorm{2}{y-x}, \qquad  y,x\in\R^{d}
\end{equation*}
	\end{linenomath}
then we say that $f$ is \emph{$(b,L)$-bilipschitz}. Moreover, we use the term \emph{$L$-bilipschitz} to refer to the special case where $b=1/L$. Occasionally we will consider Lipschitz and bilipschitz mappings between more general metric spaces and generalise the above notions in the standard way. Given a mapping $f$, we write $Df(x)$ for the (Fr\' ech\' et) \emph{derivative} of $f$ at a point $x$ and $\jac(f):=\det(Df)$ for the \emph{Jacobian} of $f$ whenever these notions make sense.
When $f$ is Lipschitz and the domain of $f$ is open, then the classical theorem of Rademacher~\cite[Thm.\ 7.3]{Mattila} asserts that $Df(x)$ exists a.e.
\section{Lipschitz regular mappings.}\label{section:regular}
Lipschitz regular mappings constitute an intermediate class between Lipschitz and bilip\-schitz mappings. While bilipschitz mappings are sometimes too rigid, Lipschitz mappings can be very degenerate; they can map many points onto a single one or map sets of positive measure onto sets of measure zero. Various classes of mappings lying somewhere in between Lipschitz and bilipschitz have been studied, for instance, in \cite{BJLPS99}, \cite{johnson_lindenstrauss_preiss_schechtman_2000}, \cite{Maleva2003}, or in \cite[Chapter~11]{benyamini1998geometric}. Lipschitz regular mappings were introduced, for the first time, by David~\cite{David88} as a class of non-degenerate Lipschitz mappings having several nice properties; see the book \cite[Chapter 2]{DS97} for a further reference\footnote{David and Semmes call `Lipschitz regular' mappings just `regular' mappings. Since the authors find the word `regular' heavily overused in mathematics, they extended the name to `Lipschitz regular'.}.

\begin{define}\label{def:regular}
Let $M$ and $M'$ be metric spaces. We say that a Lipschitz mapping $f\colon M\to M'$ is \emph{Lipschitz regular} if there is a constant $C\in\N$ such that for every $r>0$ and every ball $B\subset M'$ of radius $r$ the set $f^{-1}(B)$ can be covered with at most $C$ balls of radius $Cr$.
The smallest such $C$ is referred to as the regularity constant of $f$ and denoted by $\reg(f)$.
\end{define}

All bilipschitz mappings between metric spaces are Lipschitz regular. A clasic example of a non-bilipschitz (in fact non-injective) Lipschitz regular mapping is given by a folding mapping of the plane $\R^{2}$, i.e. take the plane and fold it along the $y$-axis. This defines a mapping $\R^{2}\to\R^{2}$ which is Lipschitz regular with regularity constant $2$.

David and Semmes studied Lipschitz regular mappings in the context of general metric spaces and Euclidean spaces. In the Euclidean space setting, David proves that Lipschitz regular mappings behave somewhat like bilipschitz mappings. More precisely, that inside any ball $B$ in the domain of a Lipschitz regular mapping $f\colon \R^{d}\to\R^{m}$ with $m\geq d$, one can find a set $E$ of large measure so that the restriction of $f$ to the set $E$ is bilipschitz; see \cite[Proposition~1, p.95]{David1988} or \cite[Theorem~4.1, p.380]{DS_reg_btwn_dim}. Although the set $E$ is large in measure, we point out that it may have empty interior. There are two natural questions arising from this result:
\begin{quest*}
Can one find a non-empty, open ball inside $B$ on which $f$ is bilipschitz and, if yes, can one additionally demand that the set $E$ above is open?
\end{quest*}
We provide answers to these questions in the case that the dimension of the domain is equal to the dimension of the co-domain. 

The main result of the present section can be derived quickly from a result of Bonk and Kleiner~\cite[Theorem~3.4]{bonk_kleiner2002}. Indeed an answer to the first part of the above question is implicitly present in \cite{bonk_kleiner2002}. Readers primarily interested in the resolution of Feige's question may therefore choose to take a short cut, beginning after the introduction of necessary notation and background on Lipschitz regular mappings; see the paragraph under the heading `Short cut' on page~\pageref{par:shortcut}.

Our argument for the proof of Theorem~\ref{p:preq2} appears to be new. Moreover, since \cite[Theorem~3.4]{bonk_kleiner2002} is a statement for more general mappings than we consider here, many of the difficulties which occur in the proof in \cite{bonk_kleiner2002} are not present in our setting. Our argument is shorter and simpler, and thus, we present it in a full detail. After the preparation of the necessary background, our argument for the proof of Theorem~\ref{p:preq2} occupies roughly one page. We discuss the result~\cite[Theorem~3.4]{bonk_kleiner2002} of Bonk and Kleiner and its proof further later on.

The material of the present section remaining after Proposition~\ref{p:regular_bilip_decomp}, which appears under the heading `Optimality of Theorem~\ref{p:preq2}' on page~\pageref{subs:optimality}, can be seen as a complement of Theorem~\ref{p:preq2} and Proposition~\ref{p:regular_bilip_decomp} and is independent of Feige's question and of \cite{bonk_kleiner2002}.

\paragraph{Notation and convention.}
We use the term \emph{$(C,L)$-regular mapping} to denote a Lip\-schitz regular mapping $f$ with $\lip(f)\leq L$ and $\reg(f)\leq C$. Let $R>0$, we call a set $S$ in a metric space $(M,\dist)$ \emph{$R$-separated} if for every two distinct points $x,y\in S$ we have $\dist(x,y)>R$.
We write $\haus^d$ for the $d$-dimensional Hausdorff measure. A ball is always assumed to be open if not said otherwise.
\bigskip 

Before we start the exposition of the results we list general properties of Lipschitz regular mappings that will be needed later. 
\begin{lemma}[{\cite[Lemma 12.3]{DS97}}]\label{l:preserve_meas}
Let $M$ and $M'$ be metric spaces.	If $f\colon M\to M'$ is Lipschitz regular and $d>0$ then there is $C=C(\lip(f), \reg(f), d)$ such that
\samepage{
\begin{linenomath}
$$
C^{-1}\haus^d(f^{-1}(E))\leq\haus^d(E)\leq C\haus^d(f^{-1}(E))
$$
\end{linenomath}
for every $E\subseteq M'$.
}
\end{lemma}

The upper bound in the previous lemma comes from the Lipschitz property, while the lower bound can be derived easily using Lipschitz regularity and the definition of Hausdorff measure. As a corollary, Lipschitz regular mappings possess Luzin's properties $(N)$ and $(N^{-1})$ given in the next Corollary:
\begin{cor}\label{c:Luzin}
	Let $f\colon M\to M'$ be Lipschitz regular. Then
	\begin{description}
		\item[$($\mylabel{luzinN}{$N$}$)$] \quad \quad for every $E\subset M$ such that $\haus^d(E)=0$ we have $\haus^d(f(E))=0$; and
		\item[$($\mylabel{luzinNinv}{$N^{-1}$}$)$] \quad for every $F\subset M'$ such that $\haus^d(F)=0$ we have $\haus^d(f^{-1}(F))=0$.
	\end{description}
\end{cor}

A converse of Lemma~\ref{l:preserve_meas} holds in the setting of Ahlfors regular spaces.
We do not introduce the definition of Ahlfors regularity here\footnote{The definition can be found in~\cite{DS97}, for instance.}, since we will work only in the setting of Euclidean spaces, which are also Ahlfors regular.
\begin{lemma}[{\cite[Lemma 12.6]{DS97}}]\label{l:meas_pres_is_regular}
	Let $M,M'$ be metric spaces and let at least one of them be Ahlfors regular of dimension $d$. If $f\colon M\to M'$ is a Lipschitz map and there is $C>0$ satisfying
	\begin{linenomath}
	$$
	\haus^d(f^{-1}(B_{M'}(x,r)))\leq C\cdot r^d
	$$
	\end{linenomath}
	for every $x\in M'$ and every $r>0$, then $f$ is Lipschitz regular.
\end{lemma}

Since the above lemma plays a key part in our resolution of Feige's question, we give a proof for the required case $M=M'=\R^{d}$.

\begin{proof}
	Without loss of generality we assume $\lip(f)=1$. Fix $x\in \R^{d}$, $r>0$ and let $\Gamma$ be a maximal $2r$-separated subset of $f^{-1}(B(x,r))$. Then the balls $(B(y,r))_{y\in\Gamma}$ are pairwise disjoint and 
		\begin{linenomath}
	\begin{equation*}
	f\left(\bigcup_{y\in\Gamma}(B(y,r))\right)\subseteq B(x,2r).
	\end{equation*}
	\end{linenomath}
	Hence, up to multiplication by a fixed constant depending only on $d$, we have
	\begin{linenomath}
	\begin{equation*}
	\left|\Gamma\right|r^{d}=\leb\left(\bigcup_{y\in\Gamma}(B(y,r))\right)\leq \leb(f^{-1}(B(x,2r)))\leq C\cdot (2r)^{d}.
	\end{equation*}
	\end{linenomath}
	and so
	\begin{linenomath}
	\begin{equation*}
	\left|\Gamma\right|\leq \frac{C\cdot(2r)^{d}}{r^{d}}=C\cdot 2^{d}.
	\end{equation*}
	\end{linenomath}
	We deduce that $f^{-1}(B(x,r))$ can be covered by $C\cdot2^{d}$ balls of radius $2r$. Thus, $f$ is Lipschitz regular with $\reg(f)\leq \max\left\{2,C\cdot2^{d}\right\}$.
\end{proof}
We also add an easy observation, which, however, will prove useful later.
\begin{obs}\label{o:size_preimage}
	Let $f\colon M\to M'$ be Lipschitz regular and $y\in M'$. Then we have
	\begin{linenomath}
	$$
	\abs{f^{-1}(\{y\})}\leq\reg(f).
	$$ 
	\end{linenomath}
\end{obs}
\begin{proof}
	To the contrary, we assume there are pairwise distinct points $x_1,\ldots,x_{\reg(f)+1}\in f^{-1}(\left\{y\right\})$. Let us denote by $r$ the minimum distance between $x_i$ and $x_j$ for $1\leq i<j\leq\reg(f)+1$. Then no ball in $M$ of radius $\frac{r}{2}$ can contain more than one of the points $x_1,\ldots,x_{\reg(f)+1}$. Therefore, $f^{-1}(B_{M'}(y,\frac{r}{2\reg(f)}))$ cannot be covered with at most $\reg(f)$ balls of radius $\frac{r}{2}$ in $M$, a contradiction.
\end{proof}

\paragraph{Short cut.}\label{par:shortcut} At this point, the reader may choose to take a short cut, provided by the result \cite[Theorem~3.4]{bonk_kleiner2002} of Bonk and Kleiner.
As an exercise, Theorem~\ref{p:preq2} can be obtained by combining \cite[Theorem~3.4]{bonk_kleiner2002} and the easy Lemma~\ref{l:convex_bilip} (see also \cite[Lemma~4.2]{bonk_kleiner2002}).
	
After this, readers primarily interested in the resolution of Feige's question are recommended to read only Proposition~\ref{p:regular_bilip_decomp}, before moving on to Section~\ref{section:geometric}.

\paragraph{Euclidean spaces.}\label{par:euclidean}
From now on we work again in the setting of Euclidean spaces. Assume that we have an open set $U\subseteq\R^{d}$ and a Lipschitz regular mapping $f\colon U\to \R^d$. By a variant of Sard's theorem for Lipschitz mappings, which can be found, e.g.\ in Mattila's book~\cite[Thm.\ 7.6]{Mattila}, we know that the set of `critical values'
\begin{linenomath}
$$
\set{f(x)\,\colon Df(x) \text{ does not exist or does not have full rank}}
$$
\end{linenomath}
has zero Lebesgue measure.
Therefore, by Corollary~\ref{c:Luzin}, the set of `non-critical points'
\begin{linenomath}
\begin{equation}\label{eq:nice_points}
N(f):=U\setminus f^{-1}\br*{\set{f(x)\,\colon Df(x) \text{ does not exist or does not have full rank}}}
\end{equation}
\end{linenomath}
occupies almost all of $U$. Notice that for every $x\in N(f)$ we have that $Df(x)$ exists and is invertible\footnote{In fact, for a Lipschitz regular mapping $f$ as above, it is easy to prove that $Df(x)$ is always invertible whenever it exists; if not, then there would be a point $x$ and a direction $v$ such that the distances between $x$ and points of the form $x+tv$ for $t>0$ small enough would be contracted by $f$ by an arbitrarily large factor, which would eventually contradict the Lipschitz regularity of $f$.}, and moreover, that $f^{-1}(\left\{f(x)\right\})\subseteq N(f)$. We use the set $N(f)$ several times later on. 

Occasionally, we will be given an open set $U$ and a Lipschitz regular mapping $f$ defined on $\cl{U}$, the closure of $U$. Then by $N(f)$ we mean the set $N(\rest{f}{U})\subseteq U$. Note that it is then still true that $f(N(f))$ has full measure in $f(U)$---we will use this fact several times.

\paragraph{Topological degree.}
An important tool that we use in our work, besides differentiability, is the notion of topological degree. We briefly introduce it here; for a detailed treatment of this topic, we refer to~\cite[Chapters~1-2]{Deim}.

The degree function 
\begin{linenomath}
\begin{equation*}
\deg\colon \left\{(f,U,y)\colon U\subseteq \R^{d}\text{ open and bounded,}\, f\in C(\overline{U},\R^{d}),\, y\in\R^{d}\setminus f(\partial U)\right\}\to \Z
\end{equation*}
\end{linenomath}
is uniquely determined by the following three properties~\cite[Theorem~1.1]{Deim}:
\begin{itemize}
	\item[(\mylabel{degprop1}{d1})] $\deg(\id,U,y)=1$ for all $y\in U$.
	\item[(\mylabel{degprop2}{d2})] (additivity) $\deg(f,U,y)=\deg(f,U_{1},y)+\deg(f,U_{2},y)$ whenever $U_{1},U_{2}$ are disjoint open subsets of $U$ such that $y\notin f(\overline{U}\setminus (U_{1}\cup U_{2}))$.
	\item[(\mylabel{degprop3}{d3})] (homotopy invariance) $\deg(h_{t},U,y_{t})=\deg(h_{0},U,y_{0})$ whenever the mappings
	\begin{linenomath}
	\begin{equation*}
	[0,1]\to C(\overline{U},\R^{d}),\, t\mapsto h_{t},\qquad [0,1]\to \R^{d},\, t\mapsto y_{t}
	\end{equation*}
	\end{linenomath}
    are continuous and $y_{t}\notin h_{t}(\partial U)$ for all $t\in [0,1]$.
\end{itemize}
The degree function is defined explicitly in \cite[Chapter 2]{Deim}. We just point out that in the special case where $g\in C^{1}(\overline{U},\R^{d})$ and for every point $x\in g^{-1}(\left\{y\right\})$ the derivative $Dg(x)$ is invertible, then the degree function is given by the expression
\begin{linenomath}
\begin{equation}\label{eq:degreeforC1regpoints}
\deg(g,U,y)=\sum_{x\in g^{-1}(\left\{y\right\})}\sign(\jac(g)(x)),
\end{equation}
\end{linenomath}
(see~\cite[Definition~2.1]{Deim}). In particular, we have that $\deg(g,U,y)=0$ whenever $y\in\R^{d}\setminus g(\overline{U})$.

We will require some further properties of the degree which follow easily from the properties (\ref{degprop1}), (\ref{degprop2}) and (\ref{degprop3}). All of the statements of the next Proposition are contained in \cite[Theorem~3.1]{Deim}.
\begin{prop}\label{prop:deg}
	Let $U\subseteq\R^{d}$ be an open, bounded set, $f\in C(\overline{U},\R^{d})$ and $y\in\R^{d}\setminus f(\partial U)$.
	\begin{enumerate}[(i)]
		\item\label{degcstcomp} If $y$ and $y'$ belong to the same connected component of $\R^{d}\setminus f(\partial U)$ then $\deg(f,U,y)=\deg(f,U,y')$.
		\item\label{degzero} If $y\in \R^{d}\setminus f(\overline{U})$ then $\deg(f,U,y)=0$.
		
		\item\label{degagree} If $g\in C(\cl{U}, \R^d)$ and $\lnorm{\infty}{f-g}<\dist(y,f(\partial U))$, then $\deg(f,U,y)=\deg(g,U,y)$. 
	\end{enumerate}
\end{prop}

In the next Proposition, we extend the formula~\eqref{eq:degreeforC1regpoints} to Lipschitz mappings.
\begin{restatable}{prop}{propdeglipregval}
\label{prop:deglipregval}
	Let $U\subseteq\R^{d}$ be an open, bounded set, $f\colon \overline{U}\to\R^{d}$ be a Lipschitz mapping and $y\in \R^{d}\setminus f(\partial U)$ be such that for every $x\in f^{-1}(\left\{y\right\})$ the derivative $Df(x)$ exists and is invertible. Then
	\begin{linenomath}
	\begin{equation*}
	\deg(f,U,y)=\sum_{x\in f^{-1}(\left\{y\right\})}\sign(\jac(f)(x)).
	\end{equation*}
	\end{linenomath}
\end{restatable}
Since the proof of Proposition~\ref{prop:deglipregval} is a rather technical exercise, we include it only in Appendix~\ref{app:regular}.

Before we present our version of the proof of Theorem~\ref{p:preq2}, let us state one additional auxiliary lemma. It says that whenever a continuous mapping in $\R^{d}$ has derivative of full rank at a point, it preserves neighbourhoods of this point. We believe that such a statement may be a folklore; however, we did not find any reference.
\begin{restatable}{lemma}{lemmaball}
\label{l:ball}
Let $a\in\R^d$, $r>0$ and $f\colon \cl{B}(a,r)\to\R^d$ be a continuous mapping (Fréchet) differentiable at the point $a$ with $\rank(Df(a))=d$. Then there is $\delta_0>0$ such that for every $\delta\in(0,\delta_0]$ we have
\begin{linenomath}
$$
B\br*{f(a),\frac{\delta}{2\opnorm{Df(a)^{-1}}}}\subseteq f(B(a,\delta)).
$$
\end{linenomath}
\end{restatable}

Up to an affine transformation, the lemma can be restated in the following way:
\begin{restatable}{lemma}{lemmaBU}
\label{l:BU}
Let $\alpha\in(0,1/3)$ and $f\colon \cl{B}(0,1)\subset\R^{d}\to\R^{d}$ be a continuous mapping such that $\lnorm{\infty}{f-\id}\leq \alpha$. Then $B(f(0),(1-2\alpha))\subseteq f(B(0,1))$.
\end{restatable}
The easy proof of Lemma~\ref{l:ball} using Lemma~\ref{l:BU}
can be found in Appendix~\ref{app:regular}. Since our proof of Lemma~\ref{l:BU} is very short, we provide it here. It relies on Proposition~\ref{prop:deg}.

\begin{proof}
The assumptions imply that $B(f(0), (1-2\alpha))$ is disjoint from $f(\partial B(0,1))$. Therefore, by Proposition~\ref{prop:deg}, part~(\ref{degcstcomp}), the degree $\deg(f,B(0,1),\cdot)$ is constant on the ball $B(f(0),(1-2\alpha))$.
By Proposition~\ref{prop:deg}, part~(\ref{degagree}), we infer that $\deg(f,B(0,1),f(0))=\deg(\id,B(0,1),f(0))=1$, since $\dist(f(0),f(\partial B(0,1)))\geq 1-2\alpha>\alpha\geq\lnorm{\infty}{f-\id}$. The lemma follows from Proposition~\ref{prop:deg}, part~(\ref{degzero}), which implies that every point of $B(f(0),(1-2\alpha))$ has to be included in $f(B(0,1))$.
\end{proof}

\paragraph{Bilipschitz decomposition of Lipschitz regular mappings.}
Our main goal in this section is to show that Lipschitz regular mappings in Euclidean spaces decompose into bilipschitz mappings in a nice way:
\thmbilipdecomp

Before we prove Theorem~\ref{p:preq2}, let us put it briefly into context.
For a general Lipschitz mapping $h\colon\R^d\to\R^d$ it is known that one can obtain a different bilipschitz decomposition using Sard's theorem; see e.g.~\cite[Lemma 3.2.2]{Fed}. One can start with sets
\begin{linenomath}
\begin{align*}
\biggr\{x\in\R^d\colon\, &Dh(x)^{-1} \text{ exists}, \opnorm{Dh(x)^{-1}}\leq k \text{ and } \forall y\in B\br*{x,\frac{1}{k}} \\
&\lnorm{2}{h(y)-h(x)-Dh(x)(y-x)}\leq\frac{\lnorm{2}{x-y}}{2k}\biggl\}
\end{align*}
\end{linenomath}
defined for every $k\in\N$ and then cut these sets into pieces of diameter less than $1/k$ forming a decomposition $\br*{A_n}_{n=1}^\infty$.
Then Sard's theorem implies that $\leb\br*{h\br*{\R^d\setminus\bigcup_{n\in\N} A_n}}=0$. When compared to the decomposition of Theorem~\ref{p:preq2}, the difference is that the sets $A_n$ are not necessarily open, the lower bilipschitz constant of each $\rest{h}{A_n}$ may depend on $n$ and $\bigcup_{n=1}^{\infty}A_{n}$ need not be a large subset of the domain in any sense. 

If the decomposition that was just described is applied to a Lipschitz regular mapping, the resulting sets $A_{n}$ occupy almost all of the domain, since the set $N(f)$ has a full measure in the domain. But the sets $A_n$ still need not be open.
The fact that for Lipschitz regular mappings it is possible to ensure the openness of bilipschitz pieces $A_{n}$ will be of crucial importance to us.

The first quantitative version of the decomposition using Sard's theorem was provided by David~\cite[Proposition~1]{David1988} for general Lipschitz mappings $f\colon \R^d\to\R^m$, where $m\geq d$. David shows that for any ball $B\subset \R^d$, if $\leb(f(B))$ is large in measure, then $B$ contains a~set $E$ large in measure such that $\rest{f}{E}$ is bilipschitz. When applied to a Lipschitz regular mapping $f$, using the measure-preserving property expressed in Lemma~\ref{l:preserve_meas}, the condition that $\leb(f(B))$ is large in measure is satisfied automatically; for this version of David's result, see~\cite[Theorem~4.1]{DS_reg_btwn_dim}.

A well-known result of Jones~\cite{Jon} provides another quantitative version of the decomposition for Lipschitz mappings $I^d\to\R^m$. In the decomposition of Jones as well as that of David the bilipschitz pieces may have empty interior.

\begin{quest*}
Can we hope for any control of the measure of the bilipschitz pieces in a bilipschitz decomposition of Lipschitz regular mappings if one requires the pieces being open?
For example, can we hope for any control of the measure of the set $\bigcup_{n=1}^{\infty}A_{n}$ given by the conclusion of Theorem~\ref{p:preq2}?
\end{quest*}
The answer to the previous question is no: The decomposition from Theorem~\ref{p:preq2} cannot be strengthened in this way for a general Lipschitz regular mapping. A detailed discussion of these questions is contained in the subsection `Optimality of Theorem~\ref{p:preq2}' at the end of this section.

\bigskip

Our proof of Theorem~\ref{p:preq2} can be divided into three parts. The first one is to find, for any given open set in the domain, an open subset on which the given mapping is almost injective (this notion is formalised below). The second part is to show that a Lipschitz regular, almost injective map on an open set is injective and the third part is to prove that a Lipschitz regular, injective map on an open set with a convex image is bilipschitz. In each of these steps we rely on the Lipschitz regularity of the mapping in question.

Let us remark that the first two steps described above, which comprise of Lemmas~\ref{l:regular_almost_injective} and \ref{l:homeo} in the following, may be replaced by an application of \cite[Theorem~3.4]{bonk_kleiner2002}. Bonk and Kleiner work in \cite{bonk_kleiner2002} with much more general mappings;
instead of assuming that $f\colon \overline{U}\subseteq \R^{d}\to\R^{d}$ is Lipschitz regular, they only require that $f$ is continuous and that there is some constant $C>0$ such that $\left|f^{-1}(\left\{y\right\})\right|\leq C$ for all $y\in\R^{d}$. The latter condition is referred to as `bounded multiplicity'. Moreover, the domain $\overline{U}$ may be replaced by any compact metric space $X$ with the property that every non-empty open subset of $X$ has topological dimension $d$.

The argument we present below is different to that of Bonk and Kleiner in \cite{bonk_kleiner2002}. However, a key aspect of both proofs appears to be finding points $x$ in the domain such that $f(x)$ is an interior point of the image $f(O)$ for every neighbourhood $O$ of $x$.
The most difficult part of Bonk and Kleiner's argument is to show that such points exist. However, for Lipschitz regular mappings we can easily find many such points using almost everywhere differentiability of Lipschitz mappings, the regularity condition and Lemma~\ref{l:ball}. Indeed, note that all points in the set $N(f)$ have this property. Therefore, our argument below may be a more accessible approach to \cite[Theorem~3.4]{bonk_kleiner2002} for the special case where the mappings considered are Lipschitz regular.

We start with the following definition.
\begin{define}
	We say that a mapping $h\colon A\subseteq\R^d\to\R^d$ is \emph{almost injective} if there is a set $B\subseteq A$ such that $\leb(A\setminus B)=0$ and $h|_B$ is injective.
\end{define}

As advertised above, we begin by showing that a Lipschitz regular mapping is almost injective on some open set:

\begin{lemma}\label{l:regular_almost_injective}
	Let $U\subseteq\R^d$ be non-empty and open and $f\colon U\to\R^d$ be Lipschitz regular. Then there is a non-empty open set $V\subseteq U$ such that $\rest{f}{V}$ is almost injective and $f(V)$ is an open ball. 
\end{lemma}
\begin{proof} The proof relies heavily on the special properties of the set $N(f)$ (see~\eqref{eq:nice_points} on page~\pageref{par:euclidean}). Choose $y\in f(N(f))$ with $\left|f^{-1}(\left\{y\right\})\right|$ maximal and set $\left\{x_1, \ldots, x_k\right\}=f^{-1}\left(\{y\}\right)\subseteq N(f)$.
We choose pairwise disjoint, open balls $C_1, \ldots, C_k$ in $U$ centred at $x_1, \ldots, x_k$, respectively. By Lemma~\ref{l:ball}, there is a non-empty, open ball $G\subseteq\bigcap_{i=1}^k f(C_i)$ centred at $y$. Hence, by the choice of $y$, the mapping $f$ is injective on each set of the form $C_i\cap f^{-1}(G)\cap N(f)$, for $i\in[k]$. Since $N(f)\cap C_i$ occupies almost all of $C_i$, any $\rest{f}{C_i\cap f^{-1}(G)}$ is almost injective.
\end{proof}

As the next step, we use the degree to show that whenever a Lipschitz regular mapping is almost injective on an open set $U$, it is injective on $U$.

\begin{lemma}\label{l:homeo}
Let $U\subseteq \R^{d}$ be an open set, $f\colon \overline{U}\to\R^{d}$ be a Lipschitz regular, almost injective mapping. Then $\rest{f}{U}$ is injective.
\end{lemma}
\begin{proof}
	Suppose not. This means we can find two points $x_1\neq x_2$ in $U$ such that $y:=f(x_1)=f(x_2)\in f(U)$. 
	We pick two disjoint balls $B_1, B_2$ in $U$ centred at $x_1, x_2$, respectively, whose boundaries do not intersect the finite set $f^{-1}(\left\{y\right\})$. We may then choose $\delta>0$ sufficiently small so that $B(y,\delta)\subseteq \R^{d}\setminus (f(\partial B_{1})\cup f(\partial B_{2}))$.
	
	By Proposition~\ref{prop:deg}, part~\eqref{degcstcomp} the degree $\deg_i:=\deg(f,B_i, \cdot)$ is constant on $B(y,\delta)$ for $i=1,2$. If for both $i=1,2$ we have $\rest{\deg_i}{B(y, \delta)}\not\equiv 0$, then by Proposition~\ref{prop:deg}, part~\eqref{degzero} every point in $B(y,\delta)$ has a preimage in both $B_1$ and $B_2$, which is impossible. Hence, say, $\rest{\deg_1}{B(y,\delta)}\equiv 0$. Since $N(f)$ is dense in $B_1$, there are points of $f(N(f))$ in $f(B_1)\cap B(y,\delta)$. Any such point has at least two preimages in $B_1$ by Proposition~\ref{prop:deglipregval}; again, this is a contradiction.
\end{proof}

The third step towards the proof of Theorem~\ref{p:preq2} is to show that a Lipschitz regular, injective mapping with a convex image is bilipschitz.

\begin{lemma}\label{l:convex_bilip}
Let $U\subseteq\R^d$ be an open set and $f\colon U\to\R^{d}$ be an injective, Lipschitz regular mapping such that $f(U)$ is convex. Then $f$ is bilipschitz with lower bilipschitz constant at least $\frac{1}{2\reg(f)^{2}}$.
\end{lemma}

We note that the same statement also appears in~\cite[Lemma 3.8]{bonk_kleiner2002}. For reader's convenience, we include its short proof here as well.
\begin{proof}
By Brouwer's invariance of domain~\cite[Thm.\ 2B.3]{Hatcher} the mapping $f$ is a homeomorphism onto its image.

For every two distinct points $x,y\in U$ we consider the line segment $\overline{f(x)f(y)}\subset f(U)$ connecting their images. Its preimage under $f$, we denote it by $\gamma(x,y):=f^{-1}\br*{\overline{f(x)f(y)}}$, is a curve with endpoints $x$ and $y$.
By Lipschitz regularity, the curve $\gamma(x,y)$ can be covered by at most $\reg(f)$ balls of radius $\reg(f)\lnorm{2}{f(y)-f(x)}$. Consequently, the distance between $x$ and $y$ cannot be larger than $2\reg(f)^2\lnorm{2}{f(y)-f(x)}$.
\end{proof}

Finally, we have gathered all the ingredients needed for the proof of Theorem~\ref{p:preq2}.
\begin{proof}[Proof of Theorem~\ref{p:preq2}]
We start with a countable basis $(U_n)_{n\in\N}$ for the subspace topology on $U$. By a~consecutive application of Lemmas~\ref{l:regular_almost_injective}, \ref{l:homeo} and \ref{l:convex_bilip} we get a collection of open sets $(V_n)_{n\in\N}$ such that for every $n\in\N$ we have $V_n\subseteq U_n$ and that $\rest{f}{V_n}$
is bilipschitz with lower bilipschitz constant $b=1/\br*{2\reg(f)^2}$.

	Now we set $A_1:=V_1$ and inductively define $A_n:=V_n\setminus\bigcup_{j=1}^{n-1}\cl{A_j}$. By construction, the set $\bigcup_{n=1}^\infty A_n$ is dense in $U$, and hence, also in $\cl{U}$.
	\end{proof}

Using Theorem~\ref{p:preq2} we can deduce that a Lipschitz regular mapping on an open set can be expressed, on some open subset of the image, as a sum of bilipschitz homeomorphisms.
Such form of a decomposition is needed for the non-realisability results contained in the next section and, ultimately, for the resolution of Feige's question~\ref{q:Feige_orig}.

\begin{prop}\label{p:regular_bilip_decomp}
	Let $U\subseteq\R^d$ be non-empty and open and $f\colon \cl{U}\to\R^{d}$ be a Lipschitz regular mapping. Then there exist a non-empty open set $T\subseteq f(\cl{U})$, $N\in[\Reg(f)]$ and pairwise~disjoint open sets $W_{1},\ldots,W_{N}\subseteq \cl{U}$ such that $f^{-1}(T)= \bigcup_{i=1}^{N}W_{i}$ and $f|_{W_{i}}\colon W_{i}\to T$ is a bilipschitz homeomorphism for each $i$ with lower bilipschitz constant $b=b(\reg(f))$. 
\end{prop}
\begin{proof}
	Let $\br*{A_n}_{n=1}^\infty$ be the open sets from the conclusion of Theorem~\ref{p:preq2} applied to the mapping $f$.
	Let $y\in f(\cl{U})$ be such that the number
	\begin{linenomath}
	\begin{equation*}
	N=N_{y}:=\left|\left\{n\in\N\colon f^{-1}(\left\{y\right\})\cap A_{n}\neq \emptyset \right\}\right|
	\end{equation*}
	\end{linenomath}
	is maximal. Note that $N\in[\reg(f)]$ by Observation~\ref{o:size_preimage}. Choose $\beta\in \N^{N}$ such that
	\begin{linenomath}
	\begin{equation*}
	y\in f(A_{n})\quad\Leftrightarrow\quad n\in\left\{\beta_{1},\beta_{2},\ldots,\beta_{N}\right\}.
	\end{equation*}
	\end{linenomath}
	Set $T=\bigcap_{i=1}^{N}f(A_{\beta_{i}})$ and note that $T$ is an open set containing $y$. We claim that $f^{-1}(T)\subseteq\bigcup_{i=1}^{N}A_{\beta_{i}}$. Assuming that this claim is valid we may define the desired sets $(W_{i})_{i=1}^{N}$ by $W_{i}:=f^{-1}(T)\cap A_{\beta_{i}}$ for each $i\in[N]$. 
	
	Thus the proof can be completed by verifying the earlier claim, that is, by proving that $f^{-1}(T)\subseteq \bigcup_{i=1}^{N}A_{\beta_{i}}$. Let $z\in f^{-1}(T)$. Using that the union $\bigcup_{n=1}^{\infty}A_{n}$ is dense in $\cl{U}$, we may find sequences $(\alpha_{k})_{k=1}^{\infty}\subseteq\N$ and $(z_{k})_{k=1}^{\infty}\subseteq \cl{U}$ with $z_{k}\in A_{\alpha_{k}}$ such that $z_{k}\to z$. But then $f(z_{k})\to f(z)\in T$ and so we may choose $K\geq 1$ sufficiently large so that $f(z_{k})\in T$ whenever $k\geq K$. By the choice of $y$ we have that
	\begin{linenomath}
	\begin{equation*}
	f(A_{n})\cap T\neq\emptyset\quad\Leftrightarrow\quad n\in\left\{\beta_{1},\ldots,\beta_{N}\right\}.
	\end{equation*}
	\end{linenomath}
	Thus we conclude that $\alpha_{k}\in\left\{\beta_{1},\ldots,\beta_{N}\right\}$ for all $k\geq K$ and $z=\lim_{k\to\infty}z_{k}\in\bigcup_{i=1}^{N}\overline{A_{\beta_{i}}}$.
	
	If $z\in \partial A_{\beta_{i}}$ for some $i\in[N]$ then we may choose $x\in A_{\beta_{i}}$ such that $f(x)=f(z)$. However, this contradicts the fact that $f$ is bilipschitz on $A_{\beta_{i}}$, and therefore also bilipschitz on $\overline{A_{\beta_{i}}}$. We conclude that $z\in\bigcup_{i=1}^{N}A_{\beta_{i}}$.
\end{proof}

\subsection*{Optimality of Theorem~\ref{p:preq2}.}\label{subs:optimality}
\addcontentsline{toc}{subsection}{Optimality of Theorem~\ref{p:preq2}.}
The remainder of the current section is devoted to discussion of limits and optimality of Theorem~\ref{p:preq2}. The content here is independent of the rest of the article, so the reader interested mainly in the resolution of Feige's question~\ref{q:Feige_orig} can safely skip the rest of this section.

In the above, we have raised a question of optimality of Theorem~\ref{p:preq2} in terms of the measure of the bilipschitz pieces $A_n$. Theorem~\ref{p:preq2} does not offer any control of their measure; below we will show that this is unavoidable. However, in a special case that a Lipschitz regular mapping $f$ has $\reg(f)\leq 2$, we can provide a stronger bilipschitz decomposition; namely, the bilipschitz pieces $A_n$, in addition to the conclusions of Theorem~\ref{p:preq2}, can cover almost all of the domain. 
		\begin{lemma}\label{lemma:decomp2reg}
			Let $U\subseteq\R^{d}$ be a bounded, open set with $\leb(\partial U)=0$ and $f\colon\overline{U}\to\R^{d}$ be a Lipschitz regular mapping with $\reg(f)\leq 2$. Then there exist pairwise disjoint, open sets $(A_{n})_{n=1}^{\infty}$ in $\overline{U}$ such that $\leb(\overline{U}\setminus \bigcup_{n=1}^{\infty}A_{n})=0$ and $f|_{A_{n}}$ is bilipschitz with lower bilipschitz constant $b(\reg(f))$.
		\end{lemma}
		\begin{proof}
			From Observation~\ref{o:size_preimage} we know that every point $y\in f(\overline{U})$ has either one or two preimages. Since $\leb(\partial U)=0$, the set $f(N(f))\setminus f(\partial U)$ has full measure in $f(\cl{U})$ by Luzin's property~\eqref{luzinN}. Let $y\in f(N(f))\setminus f(\partial U)$. Using Lemma~\ref{l:ball}, we may choose $r>0$ sufficiently small so that $B(y,r)\subseteq f(\overline{U})\setminus f(\partial U)$. 
			
			If $\deg(f,U,y)\equiv 1 \pmod{2}$, then Proposition~\ref{prop:deglipregval} implies that $y$ has exactly one preimage. Using Proposition~\ref{prop:deg}, part~\eqref{degcstcomp}, we deduce that the same is true of all points $y'\in f(N(f))\cap B(y,r)$. Thus the mapping $f\colon f^{-1}(B(y,r))\to B(y,r)$ is almost injective. We may now apply Lemma~\ref{l:homeo} and then Lemma~\ref{l:convex_bilip} to conclude that $f|_{f^{-1}(B(y,r))}$ is bilipschitz with lower bilipschitz constant $\frac{1}{2\reg(f)^{2}}$. 
			
			On the other hand, if $\deg(f,U,y)\equiv 0 \pmod{2}$ then $y$ must have two distinct pre-images $x_{1},x_{2}\in N(f)$. Let $B_{1},B_{2}$ be disjoint balls with $x_{1}\in B_{1}$ and $x_{2}\in B_{2}$. From Lemma~\ref{l:ball} we deduce that $f(B_{1})\cap f(B_{2})$ contains a non-empty open ball $G$ containing the point $y$. Then every point in $G$ has exactly one pre-image in each of the balls $B_{1}$ and $B_{2}$. Hence $f|_{f^{-1}(G)\cap B_{i}}$ is injective for $i=1,2$ and, applying Lemma~\ref{l:convex_bilip}, we conclude that these mappings are also bilipschitz with lower bilipschitz constant $\frac{1}{2\reg(f)^{2}}$. 
			
			In the above we established that for every point $y\in f(N(f))\setminus f(\partial U)$ there is an open ball $B$ containing $y$ such that $f^{-1}(B)$ decomposes precisely as the union of at most two sets on which $f$ is bilipschitz with lower bilipschitz constant $\frac{1}{2\reg(f)^{2}}$. The collection of all such balls forms a Vitali cover of $f(N(f))\setminus f(\partial U)$, so we can apply the Vitali covering theorem~\cite[Theorem~2.2, p.~26]{Mattila} to extract a countable, pairwise disjoint subcollection $(B_{n})_{n=1}^{\infty}$ which covers almost all of the set $f(N(f))\setminus f(\partial U)$, and so almost all of $f(U)$. The desired sets $A_{n}$, verifying the statement of the lemma, can now be defined as the components of the sets $f^{-1}(B_{n})$. 
		\end{proof}
		
		On the other hand, for every $\varepsilon>0$ we provide an example of a regular mapping $f\colon I^{d}\to\R^{d}$ with $\reg(f)=3$ and the following property. The set of points $x$ such that there is an open neighbourhood of $x$ on which $f$ is injective has measure at most $\varepsilon$. Consequently, for Lipschitz regular mappings $f$ with $\reg(f)\geq 3$ we cannot hope for any control of the measure of the bilipschitz pieces $A_n$ if we insist on $A_n$ being open.

\begin{example}
\label{ex:reg_not_decomposable}
For any $\varepsilon>0$ there is a $(3,\sqrt{d})$-regular mapping $f\colon I^{d}\to \sqrt{d}I^{d}$ and a set $X\subset I^{d}$ with the following properties:
\begin{enumerate}[(i)]
	\item\label{ex1} $\lnorm{\infty}{f-\sqrt{d}\id}\leq \varepsilon$.
	\item\label{ex2} $\leb(X)\geq 1-\varepsilon$.
	\item\label{ex3} For every $x\in X$ and every $\delta>0$ the mapping $f$ is not injective on the ball $B(x,\delta)$. Moreover, there are disjoint, non-empty, open balls $U_{1},U_{2}\subseteq B(x,\delta)$ such that $\frac{1}{\sqrt{d}}f|_{U_{i}}$ is an isometry for $i=1,2$ and $f(U_{1})=f(U_{2})$.
\end{enumerate}  
\end{example}
\begin{proof}
We give a proof for the case $d=1$. The example for $d\geq 1$ can easily be constructed from this: If $f\colon I\to\R$ is the example for the case $d=1$ with an appropriate choice of $\varepsilon$, then the function $h\colon I^{d}\to \sqrt{d}I^{d}$ defined by
	\begin{linenomath}
	\begin{equation*}
	h(x_{1},x_{2},\ldots,x_{d})=(\sqrt{d}f(x_{1}),\sqrt{d}x_{2},\ldots,\sqrt{d}x_{d}),\qquad (x_{1},x_{2},\ldots,x_{d})\in I^{d}
	\end{equation*}
	\end{linenomath}
	verifies Example~\ref{ex:reg_not_decomposable} for general $d\geq 1$. The formal proof of this is left to the reader.
	
	Given a point $a\in(0,1)$ and $c>0$ we will denote by $F_{a,c}$ the interval $\bs*{a, a+3c}$. Next, we define a $1$-Lipschitz function $g(a,c)\colon I\to I$ that makes two folds on $F_{a,c}$ in a sense; see Figure~\ref{f:fold_function}.
	
	More precisely, we let
	\begin{linenomath}
	$$
	g(a,c)(x):=
	\begin{cases}
	x & \text{if}\ x\leq a+c\\
	2a+2c-x & \text{if}\ x\in\bs*{a+c, a+2c}\\
	x-2c & \text{if}\ x\geq a+2c.
	\end{cases}
	$$
	\end{linenomath}
	
\begin{figure}[htb]
\begin{center}
\includegraphics{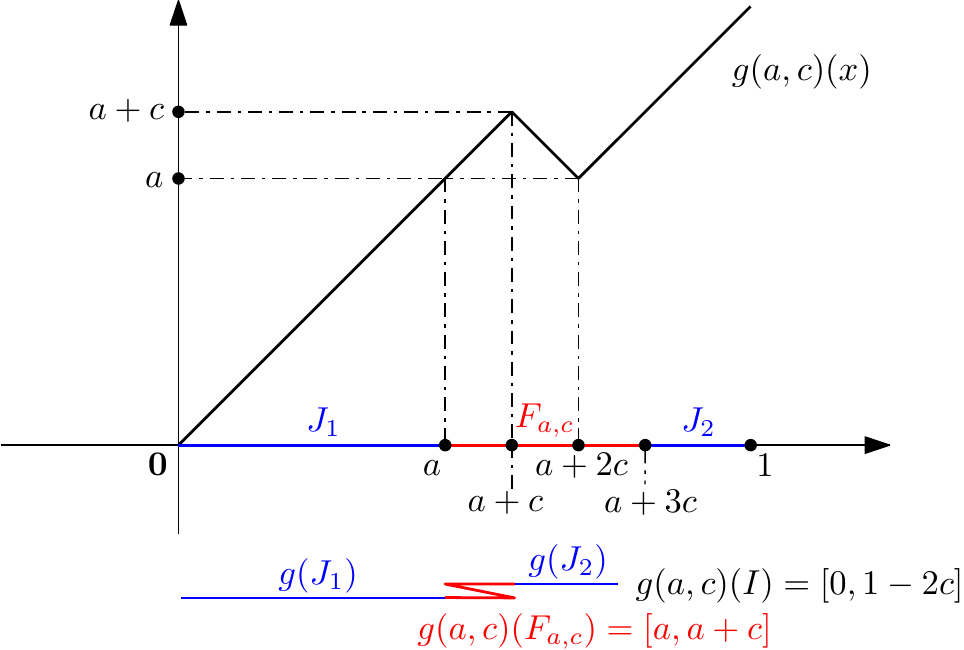}
\caption{The function $g(a,c)$ and the interval $F_{a,c}$ together with their images.}
\label{f:fold_function}
\end{center}
\end{figure}

	We will now summarise various properties of the function $g(a,c)$ which will be needed in the following construction. It is clear that $g(a,c)$ is $1$-Lipschitz and $\lnorm{\infty}{g(a,c)-\id}\leq 2c$. Moreover, $g$ isometrically maps each of the three subintervals $[a+(i-1)c,a+ic]$ of $F_{a,c}$ for $i\in[3]$ onto the same interval $[a,a+c]$. Denoting by $J_{1}$, $J_{2}$ the two components of the set $I\setminus F_{a,c}$ we further point out that the sets $g(a,c)(J_{1})$, $g(a,c)(J_{2})$ and $g(a,c)(F_{a,c})$ are pairwise disjoint subsets of $I$, and that $g$ restricted to each $J_{i}$ is a translation. Therefore, for any interval $U\subseteq g(a,c)(I)$, the preimage $g(a,c)^{-1}(U)$ is an isometric copy of $U$ whenever $U$ does not intersect $g(a,c)(F_{a,c})$, and $g(a,c)^{-1}(U)$ may be covered by $3$ intervals of length $\leb(U)$ whenever $U$ intersects $g(a,c)(F_{a,c})$.

Let $X\subseteq I$ be a fat Cantor set\footnote{See, e.g.~\cite{Mattila}.} with $\leb(X)\geq 1-\varepsilon$ and $(A_{n})_{n=1}^{\infty}$ be an enumeration of the components of $I\setminus X$. In what follows we will use the fact that every neighbourhood of a given point $x\in X$ contains some of the intervals $(A_{n})_{n=1}^{\infty}$. The idea of the construction is to `pleat' inside each of the intervals $A_{n}$ using mappings of the form $g(a,c)$ defined above; see Figure~\ref{f:many_folds}.

Now we describe the construction more formally. We start with $f_0:=\id$. Let $a_n$ be a midpoint of the interval $A_n$. For $n\in\N$ we write $g_n:=g(f_{n-1}(a_n), c_n)$ and $f_n:=g_n\circ f_{n-1}$, where $c_n>0$ are chosen small enough with respect to several constraints, which will be described during the course of the construction. Then we define $f$ as the limit of $f_n$.

The first requirement on $c_n$ is that $F_{a_n,c_n}\subset A_n$. Second, in order for $f$ to be well-defined, we want to choose $c_n$ so that the sequence $\br*{f_n}_{n=1}^\infty$ is Cauchy. We have already observed that $\lnorm{\infty}{g_n-\id}\leq 2c_n$. Thus, choosing $c_n\leq\frac{\varepsilon}{2^{n+1}}$, we get that $\lnorm{\infty}{f_{n}-f_{n-1}}\leq 2c_n\leq \frac{\varepsilon}{2^n}$ and that $f$ is well-defined. Moreover, $f$ clearly satisfies condition~\eqref{ex1}.

\begin{figure}[htb]
\begin{center}
\includegraphics{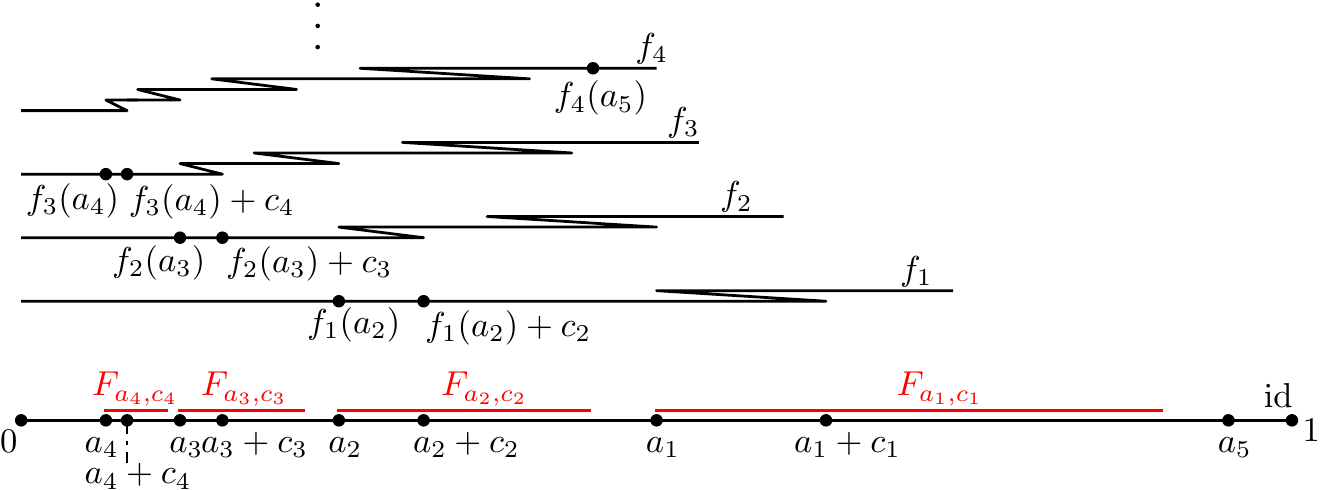}
\caption{The first four steps in the construction of the functions $\br*{f_n}_{n=1}^{\infty}$.}
\label{f:many_folds}
\end{center}
\end{figure}

To see that $f$ is $(3,1)$-regular, note that $f$ is $1$-Lipschitz because $f$ is the uniform limit of a sequence of $1$-Lipschitz functions. However, obtaining the regularity estimate is a bit more tricky.

Given any open interval $U\subseteq f(I)$ we have 
	\begin{linenomath}
	\begin{equation*}
	f^{-1}(U)\subseteq f_{n}^{-1}(\overline{B}(U,\lnorm{\infty}{f_{n}-f}),\qquad n\in\N.
	\end{equation*}
	\end{linenomath}
	Imagine, for the time being, that the latter set can be covered by $3$ intervals of length $D(\leb(U)+2\lnorm{\infty}{f_{n}-f})$ for some $D<3$. Then letting $n\to\infty$ we deduce that $f^{-1}(U)$ can be covered by $3$ closed intervals of length $D\leb(U)$, which in turn can be covered by $3$ open intervals of length $3\leb(U)$.
	
Therefore, we fix a strictly increasing sequence of numbers $D_{n}\in[1,3)$ such that $D:=\sup D_{n}<3$ and show that $c_n$ can be chosen so that the following holds true: For every $n\in\N$ and every open interval $U\subset I$ the set $f_n^{-1}(U)$ can be covered by $3$ intervals of length $D_n\leb(U)$.

For $n=0$ the condition is clearly satisfied by any $D_0\geq 1$. For a general $n\in\N$ we will distinguish three cases.
If $U$ is disjoint from $f_n\left(F_{a_n, c_n}\right)$, then $f_{n}^{-1}(U)=f_{n-1}^{-1}(g_{n}^{-1}(U))$ is the preimage under $f_{n-1}$ of a translation of $U$,  which can be covered by $3$ intervals of length $D_{n-1}\leb(U)$ by induction, which is less than $D_n\leb(U)$.

If $U$ intersects the interval $f_n\left(F_{a_n, c_n}\right)$, but is disjoint from the set $\bigcup_{i=1}^{n-1} f_n\left(F_{a_i, c_i}\right)$, then $f_n^{-1}(U)$ is a translation of $g_n^{-1}(U)$, which can be covered by $3$ intervals of length $\leb(U)$, as was already noted in the discussion of the properties of $g(a,c)$ above. 

We are left with the option that $U$ intersects $f_n\left(F_{a_n, c_n}\right)$ and also the set $\bigcup_{i=1}^{n-1} f_n\left(F_{a_i, c_i}\right)$. However, since the intervals in $\br*{A_n}_{n=1}^\infty$ are pairwise disjoint, this means that $\leb(U)$ must be quite large; namely, $\leb(U)\geq\frac{\leb(A_n)}{2}-3c_n$, since $a_n$ is the midpoint of $A_n$. On the other hand, the inequality $\lnorm{\infty}{f_{n}-f_{n-1}}\leq 2c_{n}$ implies that
	\begin{linenomath}
	\begin{equation*}
	f_{n}^{-1}(U)\subseteq f_{n-1}^{-1}(\overline{B}(U,2c_{n})).
	\end{equation*}
	\end{linenomath}
	By induction, the latter set can be covered by $3$ intervals of length $D_{n-1}(\leb(U)+4c_{n})$. The last quantity can be made smaller than $D_n\leb(U)$ using the lower bound on $\leb(U)$ and choosing $c_{n}$ small enough. This finishes the proof that the function $f$ is $(3,1)$-regular.
	
	By construction, the function $f$ is not injective on any of the intervals $F_{a_n, c_n}\subset A_n$, but it maps each of the three subintervals $[a_n+(j-1)c_n,a_n+jc_n]$ for $j\in[3]$ isometrically onto the same interval. Since every neighbourhood of any point of 	$X$ contains some of the intervals $(A_{n})_{n=1}^{\infty}$, this verifies condition~\eqref{ex3}.
\end{proof}

\bigskip

A natural question that can come to the reader's mind is where we can put Lipschitz regular mappings on the imaginary scale between bilipschitz and Lipschitz mappings? Are they closer to general Lipschitz mappings or rather to bilipschitz ones?

We can show that a typical $1$-Lipschitz mapping, in the sense of the Baire Category Theorem, is not injective on any open subset of the domain, and hence, in the light of Theorem~\ref{p:preq2}, a typical $1$-Lipschitz mapping is not Lipschitz regular:

\begin{prop}\label{obs:typical_lip_not_injective}
	Let $\mc V$ denote the complete metric space of $1$-Lipschitz mappings $I^{d}\to\R^{d}$ equipped with the metric induced by the supremum norm. Then the set of all $1$-Lipschitz mappings which are injective on some non-empty, open subset of $I^{d}$ is a meagre subset of~$\mc V$. 
\end{prop}
\begin{proof}
	Let us write $\mc B$ for a countable base of the topology on $I^d$ consisting of open balls. Moreover, for every $D\in\mc B$ we denote by $\mc I(D)$ the subset of $\mc V$ consisting of mappings that are injective on $D$.
	It is sufficient to show that the set $\mc I(D)$ forms a nowhere dense subset of $\mc V$ for every $D\in\mc B$. 
	
	Let $D=B(u,r)\in \mc B$, $g\in\mc I(D)$ and $\eta>0$. To verify that $\mc I(D)$ is nowhere dense we will find $g'\in \mc V$ and $s>0$ such that $\lnorm{\infty}{g'-g}<\eta$ and $B(g',s)\cap \mc I(D)=\emptyset$. 
	
	Choose $\varepsilon<\min\left\{\leb(B(u,r/2),r/2,\eta\right\}$ and let $f\colon I^{d}\to I^{d}$ and $X\subseteq I^{d}$ be given by Example~\ref{ex:reg_not_decomposable}. Then $g\circ \frac{1}{\sqrt{d}}f\in \mc V$, $\lnorm{\infty}{g\circ \frac{1}{\sqrt{d}}f-g}\leq \lnorm{\infty}{\frac{1}{\sqrt{d}}f-\id}\leq \varepsilon<\eta$ and, by the choice of $\varepsilon$, there exists $x\in B(u,r/2)\cap X$. By Example~\ref{ex:reg_not_decomposable}, part~\eqref{ex3} there exist disjoint, non-empty, open balls $U_{1},U_{2}\subset B(u,r/2)$ such that $\frac{1}{\sqrt{d}}f|_{U_{i}}$ is an isometry for $i=1,2$ and $\frac{1}{\sqrt{d}}f(U_{1})=\frac{1}{\sqrt{d}}f(U_{2})=:G$. Note that $\lnorm{\infty}{\frac{1}{\sqrt{d}}f-\id}\leq \varepsilon<r/2$ implies that $G\subseteq B(u,r)=D$. Therefore $g|_{G}$ is injective and, by Brouwer's Invariance of Domain~\cite[Thm.\ 2B.3]{Hatcher}, a homeomorphism. We have now established that $g\circ \frac{1}{\sqrt{d}}f$ maps each of the two disjoint, non-empty, open balls $U_{1},U_{2}\subseteq I^{d}$ homeomorphically onto the same open set $g(G)$. It follows that we can choose $s>0$ sufficiently small so that whenever $h\colon I^{d}\to\R^{d}$ is a continuous mapping with $\lnorm{\infty}{h-g\circ \frac{1}{\sqrt{d}}f}<s$ the sets $h(U_{1})$ and $h(U_{2})$ have non-empty intersection, implying that $h$ is not injective. The verification of this fact is an exercise in the topological degree and is included in Appendix~\ref{app:regular}; see Lemma~\ref{lemma:overlap}. For $s$ chosen as above, we have $B(g\circ \frac{1}{\sqrt{d}}f,s)\cap \mc I(D)=\emptyset$.
\end{proof}

\bigskip

Another question that a curious reader may ask is whether Lipschitz regular mappings can be characterised as Lipschitz mappings admitting a bilipschitz decomposition as in Theorem~\ref{p:preq2}.

However, this turns out not to be the case. It is easy to construct an example with infinitely many overlapping images of bilipschitz pieces. But even more is true: It is possible to construct an injective $1$-Lipschitz function on the unit interval that has a decomposition as in Theorem~\ref{p:preq2}, but, at the same time, is not Lipschitz regular. An example $f$ is given by the formula
\begin{linenomath}
\begin{equation*}
f(x)=\int_{0}^{x}g(t)dt,\qquad x\in I,
\end{equation*}
\end{linenomath}
where $g\colon I\to I$ is any positive, bounded, measurable function which is constant and equal to one on a dense collection of open subintervals of $I$ and not a.e.\ bounded away from zero.
\section{Geometric properties of bilipschitz mappings.}\label{section:geometric}

Bilipschitz mappings of a Euclidean space $\R^{d}$ transform volume according to the formula $\leb (f(A))=\int_{A}\abs{\jac(f)}\,d\leb $. In this section we establish that bilipschitz mappings cannot transform volume too wildly. In some sense we show that sufficiently fine grids of cubes must witness `continuity' of the volume transform. This in turn places rather restrictive conditions on the Jacobian of a bilipschitz mapping, which we will exploit in Section~\ref{section:realizability} in order to find non-realisable densities. Our work in this section is an interpretation of the construction of Burago and Kleiner~\cite{BK1}, which we modify in various ways, leading to some extensions of the results in \cite{BK1}. Critically for our solution of Feige's question, we adapt Burago and Kleiner's construction so that it treats multiple bilipschitz mappings simultaneously. In light of the statements Theorem~\ref{p:preq2} and Proposition~\ref{p:regular_bilip_decomp} obtained for Lipschitz regular mappings in the previous section, this will make Burago and Kleiner's techniques applicable to Lipschitz regular mappings.

 \paragraph{Notation.} We write $\mb{e}_{1},\ldots,\mb{e}_{d}$ for the standard basis of $\R^{d}$. For $\lambda>0$ we let $\mathcal{Q}^{d}_{\lambda}$ denote the standard tiling of $\R^{d}$ by cubes of sidelength $\lambda$ and vertices in the set $\lambda\Z^{d}$. We call a family of cubes \emph{tiled} if it is a subfamily of $\mathcal{Q}^{d}_{\lambda}$ for some $\lambda>0$. We say that two cubes $S,S'\in \mathcal{Q}^{d}_{\lambda}$ are $\mb{e}_{1}$-\emph{adjacent} if $S'=S+\lambda\mb{e}_{1}$. For mappings $h\colon \R^{d}\to\R^{k}$ we denote by $h^{(1)},\ldots,h^{(k)}$ the co-ordinate functions of $h$. 

The main result of this section will be the following lemma:
\begin{lemma}\label{lemma:geometric}
Let $d,k\in\N$ with $d\geq 2$, $L\geq 1$ and $\eta,\zeta\in(0,1)$. Then there exists $r=r(d,k,L,\eta,\zeta)\in\N$ such that for every non-empty open set $U\subseteq\R^{d}$ there exist finite tiled families $\Sq_{1},\Sq_{2},\ldots,\Sq_{r}$ of cubes contained in $U$ with the following properties:
\begin{enumerate}
\item\label{lemma:geometric1} For each $1\leq i<r$ and each cube $S\in\Sq_{i}$
\begin{linenomath}
\begin{equation*}
\leb \Biggl(S\cap\bigcup_{j=i+1}^{r}\bigcup\Sq_{j}\Biggl)\leq \eta\leb (S). 
\end{equation*}
\end{linenomath}
\item\label{lemma:geometric2} For any $k$-tuple $(h_{1},\ldots,h_{k})$ of $L$-bilipschitz mappings $h_{j}\colon U\to\R^{d}$ there exist $i\in[r]$ and $\mb{e}_{1}$-adjacent cubes $S,S'\in \Sq_{i}$ such that
\begin{linenomath}
\begin{equation*}
\left|\dashint_{S}\left|\jac(h_{j})\right|-\dashint_{S'}\left|\jac(h_{j})\right|\right|\leq\zeta
\end{equation*}
\end{linenomath}
for all $j\in[k]$.
\end{enumerate}
\end{lemma}
Statement~\ref{lemma:geometric1} expresses that each collection of cubes $\Sq_{i+1}$ is much finer than the previous collection $\Sq_{i}$. The inequality of statement~\ref{lemma:geometric2} can be interpreted geometrically as stating that the volume of the image of the cube $S$ under $h_{j}$ is very close to the volume of the image of its neighbour $S'$. Put differently, we may rewrite the inequality of \ref{lemma:geometric2} in the following form:
\begin{linenomath}
\begin{equation*}
	\left|\leb (h_{j}(S))-\leb (h_{j}(S'))\right|\leq\zeta\leb (S).
\end{equation*}
\end{linenomath}

It is possible to assemble Lemma~\ref{lemma:geometric} using predominantly arguments contained in the article~\cite{BK1} of Burago and Kleiner. However, Burago and Kleiner do not state any version of Lemma~\ref{lemma:geometric} explicitly and to prove Lemma~\ref{lemma:geometric} it is not sufficient to just take some continuous part of their argument. One needs to inspect their whole proof in detail and work considerably to put together all of the pieces correctly. Therefore, we present a complete proof of Lemma~\ref{lemma:geometric} in which we introduce some new elements.
The proof of Lemma~\ref{lemma:geometric} requires some preparation and will be given at the end of this section.

Variants of the Burago-Kleiner construction with additional details have been employed in a pure discrete setting in the works \cite{garber2009equivalence}, \cite{Magazinov2011} and \cite{navas}.

Lying behind all of the results of the present section is a simple property of Lipschitz mappings of an interval: If $[0,c]\subseteq \R$ is an interval and a Lipschitz mapping $h\colon[0,c]\to\R^{n}$ stretches the endpoints $0$, $c$ almost as much as its Lipschitz constant allows, then it is intuitively clear that the mapping $h$ is close to affine. The next dichotomy can be thought of as a `discretised' version of this statement: Statement~\ref{1dimdich2} is a discrete formulation of the condition that the Lipschitz constant of $h$ is not almost realised by the endpoints $0$, $c$. Statement~\ref{1dimdich1} expresses in a discrete way that $h$ is close to affine; after partitioning the interval $[0,c]$ into $N$ subintervals of equal length this statement asserts that $h$ looks like an affine mapping on nearly all pairs of adjacent subintervals.
\begin{restatable}{lemma}{lemmaonedimdich}\label{lemma:1dimdich}
	Let $L\geq 1$ and $\varepsilon>0$. Then there exist parameters
	\begin{linenomath}
	\begin{equation*}
	M=M(L,\varepsilon)\in\N,\qquad \varphi=\varphi(L,\varepsilon)>0
	\end{equation*}
	\end{linenomath}
	such that for all $c>0$, $n\in\N$, $N\in \N$, $N\geq 2$ and all $L$-Lipschitz mappings $h\colon [0,c]\to\R^{n}$ at least one of the following statements holds:
	\begin{enumerate}
		\item\label{1dimdich1} There exists a set $\Omega\subset[N-1]$ with $\left|\Omega\right|\geq (1-\varepsilon)(N-1)$ such that for all $i\in\Omega$ and for all $x\in\left[\frac{(i-1)c}{N},\frac{ic}{N}\right]$
		\begin{linenomath}
		\begin{equation*}
		\left\|h\left(x+\frac{c}{N}\right)-h(x)-\frac{1}{N}(h(c)-h(0))\right\|_{2}\leq\frac{c\varepsilon}{N}.
		\end{equation*}
		\end{linenomath}
		\item\label{1dimdich2} There exists $z\in\frac{c}{NM}\Z\cap[0,c-\frac{c}{NM}]$ such that
		\begin{linenomath}
		\begin{equation*}
		\frac{\lnorm{2}{h(z+\frac{c}{NM})-h(z)}}{\frac{c}{NM}}>(1+\varphi)\frac{\lnorm{2}{h(c)-h(0)}}{c}.
		\end{equation*}
		\end{linenomath}
	\end{enumerate}
\end{restatable}
We now formulate a multi-dimensional version of Lemma~\ref{lemma:1dimdich}; see Figure~\ref{f:dichotomy}. We consider thin cuboids in $\R^{d}$ of the form $[0,c]\times[0,c/N]^{d-1}$ and prove that when such a cuboid is sufficiently thin, that is when $N$ is sufficiently large, then the one-dimensional statement for $L$-Lipschitz mappings $f:[0,c]\to\R^{n}$ given in Lemma~\ref{lemma:1dimdich} can, in a sense, be extended to $L$-bilipschitz mappings $f\colon [0,c]\times [0,c/N]^{d-1}\to\R^{n}$.

\begin{restatable}{lemma}{lemmadichotomy}
\label{lemma:dichotomy}
Let $d\in\N$, $L\geq 1$ and $\varepsilon>0$. Then there exist parameters
\begin{linenomath}
\begin{equation*}
M=M(d,L,\varepsilon)\in\N,\quad \varphi=\varphi(d,L,\varepsilon)\in(0,1),\quad N_{0}=N_{0}(d,L,\varepsilon)\in\N 
\end{equation*}
\end{linenomath}
such that for all $c>0$, $n\geq d$, $N\in \N$, $N\geq N_{0}$
 and all $L$-bilipschitz mappings 
\begin{linenomath}
\begin{equation*}
h\colon [0,c]\times[0,c/N]^{d-1}\to\R^{n}
\end{equation*} 
\end{linenomath}
at least one of the following statements holds:
\begin{enumerate}
\item\label{dich1} There exists a set $\Omega\subset[N-1]$ with $\left|\Omega\right|\geq(1-\varepsilon)(N-1)$ such that for all $i\in \Omega$ and for all $\mb{x}\in\left[\frac{(i-1)c}{N},\frac{ic}{N}\right]\times\left[0,\frac{c}{N}\right]^{d-1}$
\begin{linenomath}
\begin{equation}\label{eq:translation}
\left\|h\left(\mb{x}+\frac{c}{N}\mb{e}_{1}\right)-h(\mb{x})-\frac{1}{N}(h(c\mb{e}_{1})-h(\mb{0}))\right\|_{2}\leq\frac{c\varepsilon}{N}.
\end{equation}
\end{linenomath}

\item\label{dich2} There exists $\mb{z}\in\frac{c}{NM}\Z^{d}\cap([0,c-\frac{c}{NM}]\times[0,\frac{c}{N}-\frac{c}{NM}]^{d-1})$ such that 
\begin{linenomath}
\begin{equation*}
\frac{\left\|h(\mb{z}+\frac{c}{NM}\mb{e}_{1})-h(\mb{z})\right\|_{2}}{\frac{c}{NM}}>(1+\varphi)\frac{\left\|h(c\mb{e}_{1})-h(\mb{0})\right\|_{2}}{c}.
\end{equation*}
\end{linenomath}
\end{enumerate}
\end{restatable}

\begin{figure}[htb]
\begin{center}
\includegraphics[scale=0.95]{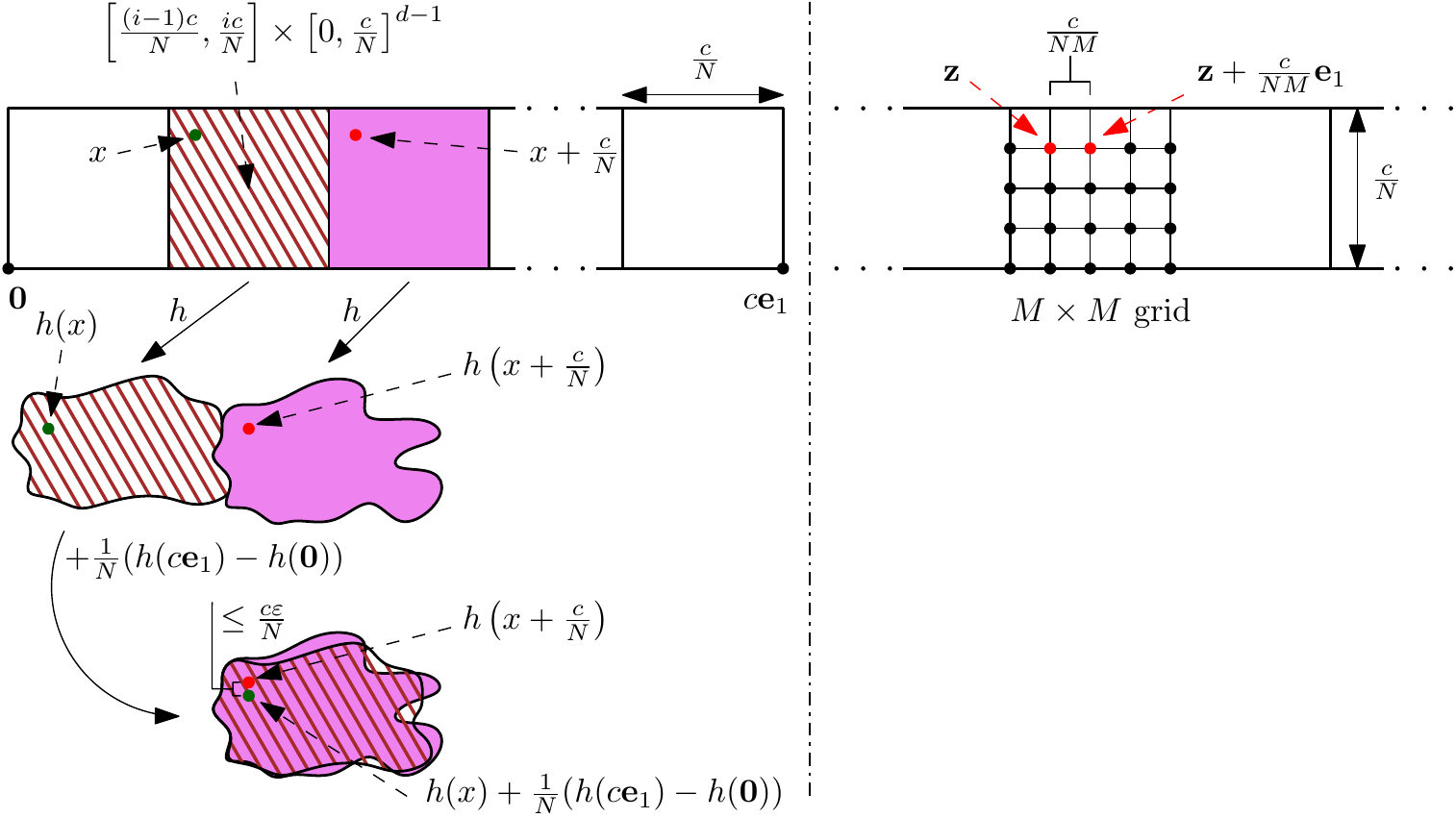}
\caption{An illustration of statement~\ref{dich1} (left) and statement~\ref{dich2} (right) of Lemma~\ref{lemma:dichotomy}. The left-hand side illustrates that $h$ maps two neighbouring cubes to `similar' images; after a translation by $\frac{1}{N}(h(c\mb{e}_{1})-h(\mb{0}))$, the image of the left cube is pointwise at least $\frac{c\varepsilon}{N}$-close to the image of its neighbouring cube.}
\label{f:dichotomy}
\end{center}
\end{figure}
We save the proof of both Lemma~\ref{lemma:1dimdich} and Lemma~\ref{lemma:dichotomy} until Appendix~\ref{app:geometric}, since a slightly weaker version of Lemma~\ref{lemma:dichotomy}, though not explicitly, is present in \cite[Lemma~3.2]{BK1}. Our proof of the one-dimensional statement Lemma~\ref{lemma:1dimdich} follows \cite{BK1} closely, but we develop a new induction argument to deduce Lemma~\ref{lemma:dichotomy} from Lemma~\ref{lemma:1dimdich}. In doing so we hope to expose clearly that the property of bilipschitz mappings established in Lemma~\ref{lemma:dichotomy} is of a one-dimensional nature.

For now, let us demonstrate how we intend to apply Lemma~\ref{lemma:dichotomy}. First, we show that whenever statement~\ref{dich1} of Lemma~\ref{lemma:dichotomy} holds for a bilipschitz mapping $h$ into $\R^{d}$, there are adjacent cubes $S_{i}$ and $S_{i+1}$ whose images under $h$ have almost the same measure. Eventually this will lead to conclusion~\ref{lemma:geometric2} of Lemma~\ref{lemma:geometric}.
\begin{lemma}\label{lemma:volume}
	Let $L\geq 1$, $\varepsilon\in(0,1/2L)$, $d\in\N$ and $N_{0}=N_{0}(d,L,\varepsilon)$ be given by the conclusion of Lemma~\ref{lemma:dichotomy}. Let $N\geq N_{0}$, $c>0$, $h\colon [0,c]\times[0,c/N]^{d-1}\to\R^{d}$ be an $L$-bilipschitz mapping, $i\in[N-1]$ and suppose that $h$ satisfies inequality \eqref{eq:translation} on $S_{i}:=\left[\frac{(i-1)c}{N},\frac{ic}{N}\right]\times\left[0,\frac{c}{N}\right]^{d-1}$. Then
	\begin{linenomath}
	\begin{equation*}
	\left|\leb (h(S_{i}))-\leb (h(S_{i+1}))\right|\leq 2L^{d+1}d\varepsilon\leb (S_{i}).
	\end{equation*}
	\end{linenomath}
\end{lemma}
\begin{proof}
	Define a translation $\phi\colon h([0,c]\times[0,c/N]^{d-1})\to\R^{d}$ by
	\begin{linenomath}
	\begin{equation*}
	\phi(h(\mb{x}))=h(\mb{x})+\frac{1}{N}(h(c\mb{e}_{1})-h(\mb{0})), \qquad \mb{x}\in[0,c]\times[0,c/N]^{d-1}.
	\end{equation*}
	\end{linenomath}
	Let the mappings $f_{1}\colon S_{i}\to\R^{d}$, $f_{2}\colon S_{i}\to\R^{d}$ be defined by $f_{1}:=\phi\circ h$ and $f_{2}(\mb{x})=h(\mb{x}+\frac{c}{N}\mb{e}_{1})$. Then $f_{1},f_{2}$ are both $L$-bilipschitz mappings of the cube $S_{i}\in\mathcal{Q}_{c/N}^{d}$ which satisfy $\lnorm{\infty}{f_{1}-f_{2}}\leq c\varepsilon/N$, due to \eqref{eq:translation}. These conditions imply a bound on the difference in volume of the images $f_{1}(S_{i})$ and $f_{2}(S_{i})$, namely
	\begin{linenomath}
	\begin{equation*}\label{eq:volumebound1}
	\left|\leb (\phi(h(S_{i})))-\leb (h(S_{i+1}))\right|\leq 2L^{d+1}d\varepsilon\leb (S_{i}).
	\end{equation*}
	\end{linenomath}
	For a verification see Lemma~\ref{lemma:bilcubvol} in the appendix. Since $\phi$ is a translation, this establishes the required inequality.
\end{proof}
Given a bilipschitz mapping $g\colon [0,c]\times[0,c/N]^{d-1}\to\R^{n}$, we now seek to repetitively apply Lemma~\ref{lemma:dichotomy} on smaller and smaller scales in order to, in some sense, eliminate statement~\ref{dich2} of the dichotomy of Lemma~\ref{lemma:dichotomy}. Consequently, we find cubes (scaled and translated copies of the sets $[(i-1)c/N,ic/N]\times[0,c/N]^{d-1}$) on which $g$ satisfies inequality~\eqref{eq:translation} of statement~\ref{dich1} of Lemma~\ref{lemma:geometric1}. This will allow us to apply Lemma~\ref{lemma:volume}. 

\paragraph{Sketch of the elimination of statement~\ref{dich2} from Lemma~\ref{lemma:dichotomy}.}
Let all parameters $d$, $L$, $\varepsilon$, $M$, $\varphi$, $N_{0}$, $c$, $n$ and $N$ be given by the statement of Lemma~\ref{lemma:dichotomy}. We consider an $L$-bilipschitz mapping $g\colon[0,c]\times[0,c/N]^{d-1}\to\R^{n}$. If statement~\ref{dich2} holds for $g$, there is a pair of points $\mb{a}_{1}:=\mb{z},\mb{b}_{1}:=\mb{z}+\frac{c}{NM}\mb{e}_{1}$ which the mapping $g$ stretches by a factor $(1+\varphi)$ more than it stretches the pair $\mb{a}_{0}:=\mb{0}$ and $\mb{b}_{0}:=c\mb{e}_{1}$. We may now consider the restriction of $g$ to a rescaled copy of the original cuboid $[0,c]\times[0,c/N]^{d-1}$ with vertices $\mb{a}_{1}$ and $\mb{b}_{1}$ corresponding to $\mb{0}$ and $c\mb{e}_{1}$ respectively. If, again, it is the case that statement~\ref{dich2} is valid for this mapping, then we find points $\mb{a}_{2},\mb{b}_{2}$ inside the new cuboid which $g$ stretches by a factor $(1+\varphi)$ more than it stretches the pair $\mb{a}_{1}$ and $\mb{b}_{1}$, and so a factor $(1+\varphi)^{2}$ times more than it stretches $\mb{a}_{0}$ and $\mb{b}_{0}$. The process is illustrated in Figure~\ref{f:elimination_dich2}. We iterate this procedure as long as possible to obtain sequences $(\mb{a}_{i})$ and $(\mb{b}_{i})$ satisfying
\begin{linenomath}
\begin{equation*}
\frac{\left\|g(\mb{b}_{i})-g(\mb{a}_{i})\right\|}{\left\|\mb{b}_{i}-\mb{a}_{i}\right\|}\geq (1+\varphi)^{i}\frac{\left\|g(\mb{b}_{0})-g(\mb{a}_{0})\right\|}{\left\|\mb{b}_{0}-\mb{a}_{0}\right\|}\geq \frac{(1+\varphi)^{i}}{L},
\end{equation*}
\end{linenomath}
where the final bound is given by the lower bilipschitz inequality for $g$. It is clear now that the procedure described above cannot continue forever: Otherwise, for $i$ sufficiently large, the inequality above contradicts the $L$-Lipschitz condition on $g$. Thus, Lemma~\ref{lemma:dichotomy} tells us that after at most $r$-iterations of the procedure, where $r\in\N$ is a number determined by $d$, $L$ and $\varepsilon$, we must have that statement~\ref{dich1} is valid for the appropriate restriction of the mapping $g$.

\begin{figure}[htb]
\begin{center}
\includegraphics[scale=0.95]{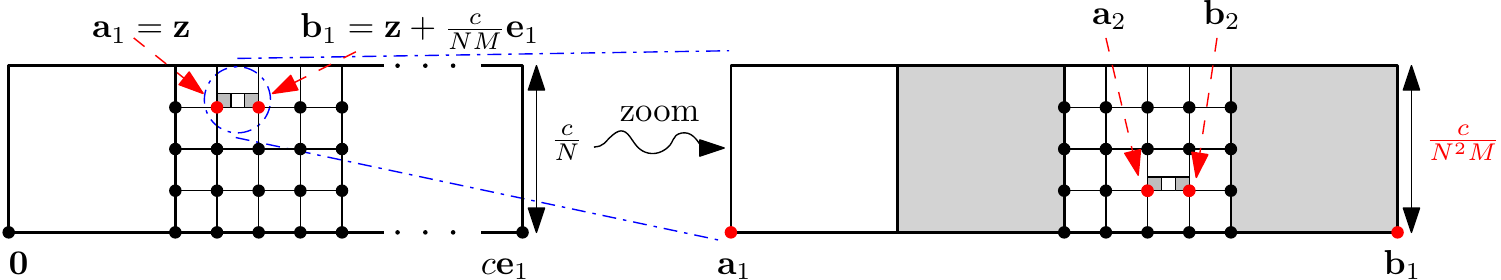}
\caption{An illustration of the strategy to eliminate statement~\ref{dich2} from Lemma~\ref{lemma:dichotomy}. The stretch factor of $g$ on the points $\mb{a}_1$ and $\mb{b}_1$ is at least $(1+\varphi)$-times larger then the stretch factor of $g$ on the points $\mb{0}$ and $c\mb{e}_1$. If statement~\ref{dich2} applies in the next iteration, we find two points stretched by $g$ with factor at least $(1+\varphi)^2$-times the stretch factor on $\mb{0}$ and $c\mb{e}_1$.}
\label{f:elimination_dich2}	
\end{center}
\end{figure}

Let us now present the conclusions of the above sketch formally. We postpone the formal proof of this statement until the appendix.
\begin{restatable}{lemma}{lemmaterminate}
\label{lemma:terminate}
	Let $d\in\N$, $L\geq 1$, $\varepsilon>0$, the parameters $M=M(d,L,\varepsilon)$, and $N_{0}=N_{0}(d,L,\varepsilon)$ be given by Lemma~\ref{lemma:dichotomy}, $c>0$, $n\geq d$, $N\geq N_{0}$, $g\colon[0,c]\times[0,c/N]^{d-1}\to\R^{n}$ be an $L$-bilipschitz mapping and $c_{i}:=\frac{c}{(NM)^{i-1}}$ for $i\in\N$. Then there exist a parameter $r=r(d,L,\varepsilon)\in\N$, $p\in[r]$ and
\begin{linenomath}
\begin{equation*}
	\mb{z}_{1}=\mb{0},\quad \mb{z}_{i+1}\in c_{i+1}\Z^{d}\cap[0,c_{i}-c_{i+1}]\times\left[0,\frac{c_{i}}{N}-c_{i+1}\right]^{d-1}\quad\text{for $i\in[p-1]$}
	\end{equation*} 
	\end{linenomath}
	such that statement~\ref{dich1} of Lemma~\ref{lemma:dichotomy} is valid for the mapping $g_{p}\colon[0,c_{p}]\times[0,c_{p}/N]^{d-1}\to\R^{n}$ defined by
	\begin{linenomath}
	\begin{equation}\label{eq:g_p}
	g_{p}(\mb{x})=g(\mb{x}+\sum_{i=1}^{p}\mb{z}_{i}).
	\end{equation}	
	\end{linenomath}
\end{restatable}
We are now ready to give a proof of Lemma~\ref{lemma:geometric}:
\begin{proof}[Proof of Lemma~\ref{lemma:geometric}]
	Let $\varepsilon=\varepsilon(\zeta,d,L,k)\in(0,\zeta)$ be a parameter to be determined later in the proof, $M=M(d,L\sqrt{k},\varepsilon)$, $\varphi=\varphi(d,L\sqrt{k},\varepsilon)$ and $N_{0}=N_{0}(d,L\sqrt{k},\varepsilon)$ be given by the statement of Lemma~\ref{lemma:dichotomy}, $N\geq N_{0}$ and $r=r(d,L\sqrt{k},\varepsilon)\in\N$ be given by the conclusion of Lemma~\ref{lemma:terminate}. We impose additional conditions on $\varepsilon$ and $N$ in the course of the proof.
	
	Let $U\subseteq \R^{d}$ be a non-empty open set. Since the conclusion of Lemma~\ref{lemma:geometric} is invariant under translation of the set $U\subseteq \R^{d}$, we may assume that $\mb{0}\in U$ and choose $c>0$ such that
	\begin{linenomath}
	\begin{equation*}
	[0,c]\times[0,c/N]^{d-1}\subseteq U.
	\end{equation*} 
	\end{linenomath}
	We are now ready to define the families of cubes $\Sq_{1},\ldots,\Sq_{r}$, making use of the sequence $c_{i}:=\frac{c}{(NM)^{i-1}}$ defined in Lemma~\ref{lemma:terminate};  see also Figure~\ref{f:tiled_families}.
	\begin{define}\label{construction}
		For each $i\in[r]$ we define the family $\Sq_{i}\subseteq \mathcal{Q}_{c_{i}/N}$ as the collection of all cubes of the form
		\begin{linenomath}
		\begin{equation*}
		\sum_{j=1}^{i}\mb{z}_{j}+\left[\frac{(l-1)c_{i}}{N},\frac{lc_{i}}{N}\right]\times\left[0,\frac{c_{i}}{N}\right]^{d-1}
		\end{equation*}
		\end{linenomath}
		where $\mb{z}_{1}=0$, $\mb{z}_{j+1}\in c_{j+1}\Z^{d}\cap[0,c_{j}-c_{j+1}]\times[0,\frac{c_{j}}{N}-c_{j+1}]^{d-1}$ for each $j\geq 1$ and $l\in[N]$.
	\end{define}
	
\begin{figure}[htb]
\begin{center}
\includegraphics{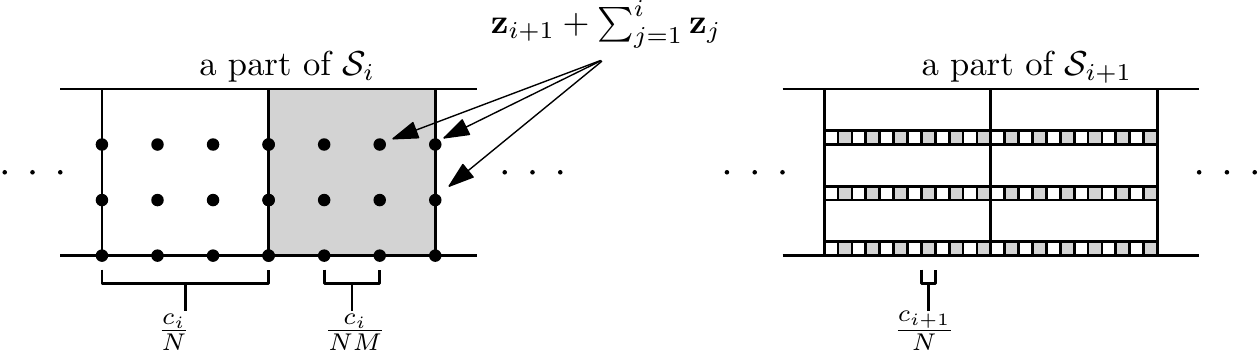}
\caption{Left: Two cubes from the family $\mc{S}_i$ with points of the form $\sum_{j=1}^{i+1}\mb{z}_{j}$, where $\mb{z}_j$ are fixed for $j=1,\ldots, i$. Right: The resulting part of the family $\mc{S}_{i+1}$.}
\label{f:tiled_families}
\end{center}	
\end{figure}
	
	Let us now verify that the above defined families $\Sq_{1},\ldots,\Sq_{r}$ satisfy condition~\ref{lemma:geometric1} in the statement of Lemma~\ref{lemma:geometric}. 
	It is immediate from Definition~\ref{construction} that
	\begin{linenomath}
	\begin{equation*}\label{eq:nested}
	\bigcup \Sq_{r}\subseteq \bigcup \Sq_{r-1}\subseteq\ldots\subseteq \bigcup\Sq_{1}.
	\end{equation*}
	\end{linenomath}
	
	Thus, given $1\leq i<r$ and $S\in\Sq_{i}$, we have that 
	\begin{linenomath}
	\begin{equation*}
	S\cap\bigcup_{j=i+1}^{r}\bigcup\Sq_{j}\subseteq S\cap \bigcup\Sq_{i+1}
	\end{equation*}
	\end{linenomath}
	Note that any cube in the collections $\Sq_{i+1}$ has the form 
	\begin{linenomath}
	\begin{equation*}
	\mb{w}+\left[\frac{(l-1)c_{i+1}}{N},\frac{lc_{i+1}}{N}\right]\times\left[0,\frac{c_{i+1}}{N}\right]^{d-1}
	\end{equation*}
	\end{linenomath}
	for some $\mb{w}\in c_{i+1}\Z^{d}$ and $l\in[N]$. Since $S\in\mathcal{Q}_{c_{i}/N}$ and $c_{i}/N=Mc_{i+1}$, such a cube can only intersect $S$ in a set of positive Lebesgue measure when $\mb{w}\in S$. Therefore, the number of cubes in $\Sq_{i+1}$ that can intersect $S\in\mc Q_{c_{i}/N}$ in a set of positive Lebesgue measure is bounded above by
	\begin{linenomath}
	\begin{equation*}
	N\left|c_{i+1}\Z^{d}\cap S\right|\leq N\left(\frac{c_{i}/N}{c_{i+1}}+1\right)^{d}=N(M+1)^{d}.
	\end{equation*}
	\end{linenomath}
	It follows that
	\begin{linenomath}
	\begin{align*}
	\leb (S\cap \bigcup_{j=i+1}^{r}\bigcup\Sq_{j})&\leq N(M+1)^{d}(c_{i+1}/N)^{d}=\frac{(M+1)^{d}}{M^{d}N^{d-1}}\left(\frac{c_{i}}{N}\right)^{d}\leq\eta\leb (S).
	\end{align*}
	\end{linenomath}
	where, in the above, we use $c_{i+1}=c_{i}/NM$ and $\leb (S)=(c_{i}/N)^{d}$ and prescribe that $N$ is sufficiently large so that the inequality holds. Thus, statement~\ref{lemma:geometric1} is satisfied. 
	
	Turning now to statement~\ref{lemma:geometric2}, we consider a $k$-tuple $(h_{1},\ldots,h_{k})$ of $L$-bilipschitz mappings $h_{i}\colon U\to\R^{d}$ and define a mapping $g\colon U\to \R^{kd}$ co-ordinate-wise by
	\begin{linenomath}
	\begin{equation*}
	g^{((i-1)d+j)}(\mb{x})=h_{i}^{(j)}(\mb{x})
	\end{equation*}
	\end{linenomath}
	for $i\in[k]$ and $j\in[d]$. It is straightforward to verify that $g$ is $L\sqrt{k}$-bilipschitz. The conditions of Lemma~\ref{lemma:terminate} are now satisfied for $d$, $L=L\sqrt{k}$, $\varepsilon$, $M$, $N_{0}$, $c$, $n=kd$ and $g\colon [0,c]\times[0,c/N]^{d-1}\to\R^{kd}$. Let $p\in [r]$ and $\mb{z}_{1},\ldots,\mb{z}_{p}\in\R^{d}$ be given by the conclusion of Lemma~\ref{lemma:terminate}. Then statement~\ref{dich1} of Lemma~\ref{lemma:dichotomy} holds for the mapping $g_{p}\colon [0,c_{p}]\times[0,c_{p}/N]^{d-1}\to\R^{kd}$ defined by~\eqref{eq:g_p}. Let $\Omega$ be given by the assertion of Lemma~\ref{lemma:dichotomy}, statement~\ref{dich1} for $g_{p}$. The co-ordinate functions of the mapping $g_{p}\colon [0,c_{p}]\times[0,c_{p}/N]^{d-1}\to\R^{kd}$ are defined by
	\begin{linenomath}
	\begin{equation*}
	g_{p}^{((t-1)d+s)}(\mb{x})=g^{((t-1)d+s)}(\mb{x}+\sum_{j=1}^{p}\mb{z}_{j})=h_{t}^{(s)}(\mb{x}+\sum_{j=1}^{p}\mb{z}_{j})
	\end{equation*}
	\end{linenomath}
	for $t\in[k]$, $s\in[d]$. Therefore for each $i\in\Omega$ and each $h_{t,p}\colon [0,c_{p}]\times[0,c_{p}/N]^{d-1}\to\R^{d}$ defined by $h_{t,p}(\mb{x})=h_{t}(\mb{x}+\sum_{j=1}^{p}\mb{z}_{j})$ for $t\in[k]$, we have that $h=h_{t,p}$ satisfies inequality~\ref{eq:translation} on $S_{i}$.

	We fix $i\in\Omega$ and impose the condition $\varepsilon<\frac{1}{2L}$ on $\varepsilon$. Then the conditions of Lemma~\ref{lemma:volume} are satisfied for $L'=L\sqrt{k}$, $\varepsilon$, $d$, $N$, $c=c_{p}$, $h=h_{t,p}$ for each $t\in[k]$ and $i$. Hence, 
	\begin{linenomath}
	\begin{equation*}
	\left|\leb (h_{t,p}(S_{i}))-\leb (h_{t,p}(S_{i+1}))\right|\leq 2(L\sqrt{k})^{d+1}d\varepsilon\leb (S_{i})\leq\zeta\leb (S_{i}),
	\end{equation*}
	\end{linenomath}
	when we prescribe that $\varepsilon\leq \frac{\zeta}{2(L\sqrt{k})^{d+1}d}$. Set $S=\sum_{j=1}^{p}\mb{z}_{j}+S_{i}$ and $S'=\sum_{j=1}^{p}\mb{z}_{j}+S_{i+1}$. It is clear upon reference to Definition~\ref{construction} that $S$ and $S'$ are $\mb{e}_{1}$-adjacent cubes belonging to the family $\Sq_{p}$. Moreover, we have $h_{t}(S)=h_{t,p}(S_{i})$ and $h_{t}(S')=h_{t,p}(S_{i+1})$ for all $t\in[k]$. Therefore $S$ and $S'$ verify statement~\ref{lemma:geometric2} of Lemma~\ref{lemma:geometric} for the $k$-tuple $(h_{1},\ldots,h_{k})$. This completes the proof of Lemma~\ref{lemma:geometric}.
\end{proof}
\section{Realisability in spaces of functions.}\label{section:realizability}
The objective of the present section is to prove that in some sense almost all continuous functions $\rho\in C(I^{d},\R)$ do not admit a Lipschitz regular mapping $f\colon I^{d}\to\R^{d}$ such that 
\begin{linenomath}
\begin{equation}\label{eq:pushforward}
f_{\sharp}\rho\leb=\leb|_{f(I^{d})},
\end{equation}
\end{linenomath}
where we view $C(I^{d},\R)$ as a Banach space with the supremum norm $\lnorm{\infty}{-}$. More~precisely, we prove the following result:
\renewcommand{\pushforwardeq}{\eqref{eq:pushforward}\ }
\renewcommand{\unitinterval}{I}
\thmregrlzprs
\begin{remark}
To be able to work with functions $\rho\in C(I^{d},\R)$ attaining negative values as well, we extend the definition of the pushforward measure to such functions:
\begin{linenomath}
$$
f_{\sharp}\rho\leb:=f_{\sharp}\rho^{+}\leb-f_{\sharp}\rho^{-}\leb,
$$
\end{linenomath}
where by $\rho^+, \rho^-$ we mean the positive and the negative part of $\rho$. Technically speaking, the pushforward measure is no longer a measure, but a difference of two measures\footnote{Sometimes this is called a \emph{signed measure} or a \emph{charge} in the literature.}. However, we will use it only in the form of \eqref{eq:pushforward}, that is, when the result is again a measure.

This is only a technical tool that helps us treat functions attaining negative values properly, but it does not bring in any additional difficulty to the present work. An alternative option would be to say that, by definition, no function with negative values satisfies \eqref{eq:pushforward}, but the statement of Theorem~\ref{thm:regrlzprs} would be then seemingly weaker.
\end{remark}

Burago and Kleiner~\cite{BK1} and McMullen~\cite{McM} prove the existence of a positive function $\rho\colon I^{2}\to \R$ for which equation \eqref{eq:pushforward} has no bilipschitz solutions $f\colon I^{2}\to\R^{2}$. We point out that Theorem~\ref{thm:regrlzprs} extends this result in various ways. Firstly, Lipschitz regular mappings of $I^{d}$ into $\R^{d}$ form a larger class than the class of bilipschitz mappings from $I^{d}$ to $\R^{d}$. Thus, Theorem~\ref{thm:regrlzprs} establishes the existence of a density $\rho$ which admits no solutions $f$ to equation~\eqref{eq:pushforward} inside a larger class of mappings. Secondly, Theorem~\ref{thm:regrlzprs} asserts the existence of not only one such density $\rho$, but states that almost all continuous functions $\rho\in C(I^{d},\R)$ are not realisable in the sense of \eqref{eq:pushforward} for Lipschitz regular mappings $f\colon I^{d}\to\R^{d}$. That bilipschitz non-realisable functions contain a dense $G_{\delta}$ subset of both the set of positive continuous functions and the set of positive, $L^{\infty}$-bounded, measurable functions on the unit square $[0,1]^{2}$, was recently proved by Viera in \cite{viera2016densities}, but \cite{viera2016densities} is completely independent from the present work.

\begin{remark}
	We point out that there are positive, bilipschitz non-realisable densities in $C(I^{d},\R)$ which fail to be Lipschitz regular non-realisable, i.e.\ positive functions $\rho\in C(I^{d},\R)$ for which equation \eqref{eq:pushforward} admits Lipschitz regular but not bilipschitz solutions $f\colon I^{d}\to \R^{d}$. An example may be constructed as follows. We split the unit cube $I^{d}$ in half, distinguishing two pieces $D_{1}:=[0,\frac{1}{2}]\times I^{d-1}$ and $D_{2}:=[\frac{1}{2},1]\times I^{d-1}$ and write $f\colon I^{d}\to D_{1}$ for the mapping which `folds $D_{2}$ onto $D_{1}$'. More precisely, the mapping $f$ is defined as the identity mapping on $D_{1}$ and as the reflection in the hyperplane $\left\{\frac{1}{2}\right\}\times\R^{d-1}$ on $D_{2}$. Let $\psi\in C(D_{1},\R)$ be a positive, bilipschitz non-realisable density with values in $(0,1)$. We impose the additional mild condition that $\psi$ is constant with value $\frac{1}{2}$ inside the hyperplane $\frac{1}{2}\times \R^{d-1}$. The existence of such a density $\psi$ follows easily from the $d$-dimensional analog of \cite[Theorem~1.2]{BK1}. 
	
	Set $\rho=\psi$ on $D_{1}$. The bilipschitz non-realisability of $\rho$ is now already assured, no matter how we define $\rho$ on $D_{2}$. To make $\rho$ Lipschitz regular realisable, we define $\rho$ on $D_{2}$ by
	\begin{linenomath}
	\begin{equation*}
	\rho(x)=1-\psi(f(x)).
	\end{equation*} 
	\end{linenomath}
	The function $\rho\colon I^{d}\to\R$ is continuous and positive, whilst the mapping $f\colon I^{d}\to \R^{d}$ is Lipschitz regular and satisfies $f(I^{d})=D_{1}$. Moreover, for any measurable set $S\subseteq D_{1}$ we have
	\begin{linenomath}
	\begin{align*}
	f_{\sharp}\rho\leb(S)&=\int_{f^{-1}(S)\cap D_{1}}\rho\,d\leb+\int_{f^{-1}(S)\cap D_{2}}\rho\,d\leb\\
	&=\int_{S}\psi\,d\leb+\int_{f^{-1}(S)\cap D_{2}}(1-\psi(f(x)))\,d\leb\\
	&=\int_{S}\psi\,d\leb+\int_{S}(1-\psi)\,d\leb=\leb(S),
	\end{align*}
	\end{linenomath}
	where, for the penultimate equation, we use the change of variables formula and the fact that $f$ restricted to the set $D_{2}$ is an affine isometry. This verifies the Lipschitz regular realisability of $\rho$.

\end{remark}
\paragraph[\texorpdfstring{Porous and $\sigma$-porous sets.}{Porous and sigma-porous sets.}]{Porous and $\sigma$-porous sets.} We recall the definitions of porosity and $\sigma$-porosity according to \cite[Definition~2.1]{zajicek2005}, where they are referred to as `lower porosity' and `lower $\sigma$-porosity' respectively.
\begin{define}\label{def:porous}
	Let $(X,\left\|-\right\|)$ be a Banach space. \begin{enumerate}[(i)]
		\item A set $P\subseteq X$ is called porous at a point $x\in X$ if there exist $\varepsilon_{0}>0$ and $\alpha\in (0,1)$ such that for every $\varepsilon\in(0,\varepsilon_{0})$ there exists $y\in X$ such that 
		\begin{linenomath}
		\begin{equation*}
		\left\|y-x\right\|\leq\varepsilon\quad\text{ and }\quad B(y,\alpha\varepsilon)\cap P=\emptyset.
		\end{equation*}
		\end{linenomath}
		\item A set $P\subseteq X$ is called porous if $P$ is porous at $x$ for every point $x\in P$.
		\item A set $E\subseteq X$ is called $\sigma$-porous if $E$ may be expressed as a countable union of porous subsets of $X$.
	\end{enumerate} 
\end{define}
The class of $\sigma$-porous subsets of a Banach space $X$ is strictly contained in the class of subsets of $X$ of the first category in the sense of the Baire~Category~Theorem. Further, the notions of porosity and $\sigma$-porosity extend to metric spaces in the natural way. For a survey on porous and $\sigma$-porous sets we refer the reader to \cite{zajicek2005}. Due to its relevance later in the paper, we point out that porosity of a set $P\subseteq X$ is a weaker condition than requiring $P$ to be porous at all points $x\in X$ (not just at points $x\in P$). For example, the set $\left\{\frac{1}{n}\colon n\in\Z\setminus\left\{0\right\}\right\}$ is porous in $\R$ but is not porous at the point $0$.

Let us begin working towards the proof of Theorem~\ref{thm:regrlzprs}. 

\paragraph[\texorpdfstring{Porous decompositon of $\E$.}{Porous decompositon of E.}]{Porous decompositon of $\E$.} In the present paragraph we describe how to partition the set $\E$ into a countable family of porous sets $(\E_{C,L,n})$. We will need the lower bilipschitz constant $b(\cdot)$ given by the conclusion of Proposition~\ref{p:regular_bilip_decomp}. Let $(O_{n})_{n=1}^{\infty}$ be a countable basis for the topology of $I^{d}$. For $C,L,n\in\N$ we let $\E_{C,L,n}$ denote the set of all functions $\rho\in C(I^{d},\R)$ which admit $N\in[C]$, pairwise disjoint, non-empty, open~sets $Y_{1},\ldots,Y_{N}\subseteq I^{d}$ with $Y_{1}:=O_{n}$, an open set $V\subseteq \R^{d}$ and $(b(C),L)$-bilipschitz homeomorphisms $f_{i}\colon Y_{i}\to V$ such that 
\begin{linenomath}
\begin{equation}\label{eq:sumtransdens'}
\rho(y)=\left|\jac(f_{1})(y)\right|-\sum_{i=2}^{N}\rho(f_{i}^{-1}\circ f_{1}(y))\left|\jac(f_{i}^{-1}\circ f_{1})(y)\right| \qquad \text{for a.e.\ $y\in O_{n}$.}
\end{equation}
\end{linenomath}
Note that the basis set $O_{n}$ `generates' the diagram of bilipschitz homeomorphisms $f_{i}\colon Y_{i}\to V$ in the sense that we have $Y_{i}=f_{i}^{-1}\circ f_{1}(O_{n})$ for $i\in[N]$ and $V=f_{1}(O_{n})$; see Figure~\ref{f:diagram_decomp}. However, the critical role of $O_{n}$ in the definition above is to prescribe the portion of the domain $I^{d}$ on which all functions $\rho\in \E_{C,L,n}$ have the special form given by \eqref{eq:sumtransdens'}. 

\begin{figure}[htb]
\begin{center}
\includegraphics{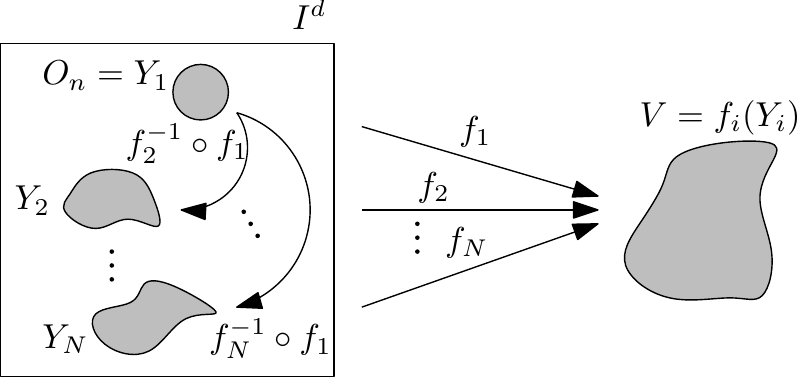}
\caption{The diagram of a bilipschtz decomposition for a density $\rho\in \E_{C,L,n}$.}
\label{f:diagram_decomp}
\end{center}
\end{figure}

To explain the origins of equation~\eqref{eq:sumtransdens'}, we refer the reader back to Proposition~\ref{p:regular_bilip_decomp}. Consider a Lipschitz regular mapping $f\colon I^{d}\to \R^{d}$ and the non-empty open set $T\subseteq f(I^{d})$ given by the conclusion of Proposition~\ref{p:regular_bilip_decomp}. Because the preimage $f^{-1}(T)$ decomposes precisely as a union of $N$ sets on which $f$ defines a bilipschitz homeomorphism to $T$, a pushforward (signed) measure of the form $f_{\sharp}\rho\leb$ with $\rho\in C(I^{d},\R)$ can be expressed on $T$ as a sum of integrals involving $\rho$ and Jacobians of $N$ bilipschitz homeomorphisms $f_{1},\ldots,f_{N}$.
Thus, whenever $f_{\sharp}\rho\leb=\leb|_{f(I^{d})}$ we obtain some equation relating $\rho$ to finitely many bilipschitz homeomorphisms and their Jacobians. We will see that this equation has precisely the form of \eqref{eq:sumtransdens'}.

 For $C,L\in\N$, let $\E_{C,L}$ denote the subset of $C(I^{d},\R)$ consisting of all functions $\rho$ for which there exists a $(C,L)$-regular mapping $f\colon I^{d}\to \R^{d}$ such that $f_{\sharp}\rho\leb=\leb|_{f(I^{d})}$. Clearly we have $\E=\bigcup_{C,L\in\N}\E_{C,L}$. In the next lemma we prove that $\E_{C,L}$ is covered by the sets $(\E_{C,L,n})$.
\begin{lemma}\label{lemma:sigporcov}
	Let $C,L\in\N$. Then $\E_{C,L}\subseteq\bigcup_{n\in\N}\E_{C,L,n}$.
\end{lemma}
\begin{proof}
	Let $\rho\in\E_{C,L}$ and choose a $(C,L)$-regular mapping $f\colon I^{d}\to\R^{d}$ such that $f_{\sharp}\rho\leb=\leb|_{f(I^{d})}$. Let the integer $N\in[C]$ and the open sets $T\subseteq f(I^{d})$ and $W_{1},\ldots,W_{N}\subseteq I^{d}$ be given by the conclusion of Proposition~\ref{p:regular_bilip_decomp}. We choose $n\in\N$ such that $O_{n}\subseteq W_{1}$ and define $Y_{i}=f^{-1}(f(O_{n}))\cap W_{i}$, $V=f(O_{n})$ and $f_{i}:=f|_{Y_{i}}\colon Y_{i}\to V$ for each $i\in[N]$.
	
	To see that these choices witness that $\rho\in\E_{C,L,n}$, it only remains to verify equation~\eqref{eq:sumtransdens'}. Note that $f^{-1}(V)= \bigcup_{i=1}^{N}Y_{i}$. Therefore, for every measurable set $S\subseteq V$ we have that 
	\begin{linenomath}
	\begin{multline*}
	\leb(S)=f_{\sharp}\rho\leb(S)=\sum_{i=1}^{N}\int_{f^{-1}(S)\cap Y_{i}}\rho d\leb=\sum_{i=1}^{N}\int_{f_{i}^{-1}(S)}\rho d\leb\\=\sum_{i=1}^{N}\int_{S}(\rho\circ f_{i}^{-1})\left|\jac(f_{i}^{-1})\right|d\leb=\int_{S}\sum_{i=1}^{N}\frac{\rho}{\left|\jac(f_{i})\right|}\circ f_{i}^{-1}d\leb.
	\end{multline*}
	\end{linenomath}
	We conclude that
	\begin{linenomath}
	\begin{equation*}
	\sum_{i=1}^{N}\frac{\rho}{\left|\jac(f_{i})\right|}\circ f_{i}^{-1}(x)=1\qquad\text{ for a.e.\ $x\in V$.}
	\end{equation*}
	\end{linenomath}
	Recall that the sets $Y_{i}$ and $V$ are all bilipschitz homeomorphic via the mappings $f_{i}\colon Y_{i}\to V$. Therefore, we may make the substitution $x=f_{1}(y)$ in the above equation, after which a simple rearrangement and an application of a `chain rule identity' for Jacobians yields \eqref{eq:sumtransdens'}.	
\end{proof}
If, for the time being, we treat the terms $\rho(f_{i}^{-1}\circ f_{1}(y))$ in \eqref{eq:sumtransdens'} as constants, then, on the open set $O_{n}$, functions $\rho\in \E_{C,L,n}$ are linear combinations of at most $C$ Jacobians of $L/b(C)$-bilipschitz mappings. The purpose of the next lemma is to provide, for given constants $k$ and $L$, a function $\psi\in C(I^{d},\R)$ which is small in supremum norm, but far away from being a linear combination of $k$ $L$-bilipschitz Jacobians.
\begin{restatable}{lemma}{lemmapeturb}
	\label{lemma:perturb}
Let $\varepsilon,\zeta\in (0,1)$, $L\geq 1$, $k\in \N$ and $U\subseteq I^{d}$ be an open set. Then there exists a function $\psi\in C(I^{d},\R)$ such that $\left\|\psi\right\|_{\infty}\leq \varepsilon$, $\supp(\psi)\subseteq U$ and for every $k$-tuple $(h_{1},h_{2},\ldots,h_{k})$ of $L$-bilipschitz mappings $h_{i}\colon U\to\R^{d}$ there exist $\mb{e}_{1}$-adjacent cubes $S,S'\subseteq U$ such that
\begin{linenomath}
\begin{equation}\label{eq:jaccontrol}
\left|\dashint_{S}\left|\jac(h_{i})\right|-\dashint_{S'}\left|\jac(h_{i})\right|\right|\leq\zeta
\end{equation}
\end{linenomath}
for all $i\in[k]$ and
\begin{linenomath}
\begin{equation}\label{eq:chessboard}
\left|\dashint_{S}\psi-\dashint_{S'}\psi\right|\geq \varepsilon.
\end{equation}
\end{linenomath}
\end{restatable}
To prove Lemma~\ref{lemma:perturb} it suffices to consider the families of tiled cubes $\Sq_{1},\ldots,\Sq_{r}$ given by the conclusion of Lemma~\ref{lemma:geometric} applied to $d$, $k$, $U$, $L$, $\zeta$ and some very small $\eta\in(0,1)$, and to define $\psi$ as a `chessboard function' whose average value on $\mb{e}_{1}$-adjacent cubes makes jumps of size at least $\varepsilon$. From the conclusion~\ref{lemma:geometric1} of Lemma~\ref{lemma:geometric} we may essentially regard the cubes from two different families $\Sq_{i}$, $\Sq_{j}$ as pairwise disjoint; choosing $\eta$ sufficiently small ensures that the values of $\psi$ on $\bigcup \Sq_{j}$ have negligible impact on the average values of $\psi$ on cubes in $\Sq_{i}$ for $i<j$. We postpone the formal description of this construction until Appendix~\ref{app:realisability}, since it is a standard argument without any deep ideas. For now, let us proceed to the key proof of the present section, namely the verification of porosity of the sets $(\E_{C,L,n})$. We actually prove that the sets $(\E_{C,L,n})$ possess a stronger property:
\begin{lemma}\label{lemma:por}
	For every $C,L,n\in\N$, $\E_{C,L,n}$ is a porous subset of $C(I^{d},\R)$. In fact, the set $\E_{C,L,n}$ is porous at every point $\phi\in C(I^{d},\R)$.
\end{lemma}
Before we begin the proof, we will outline the strategy. For given $C,L,n\in\N$, $\phi\in C(I^{d},\R)$ and $\varepsilon\in(0,1)$, our task is to find a function $\widetilde{\phi}\in C(I^{d},\R)$ so that $\left\|\widetilde{\phi}-\phi\right\|_{\infty}\leq\varepsilon$ and $B(\widetilde{\phi},\alpha\varepsilon)\cap \E_{C,L,n}=\emptyset$ for some $\alpha=\alpha(C,L,n,\phi)$. We will exploit the uniform continuity of $\phi$: By prescribing at the start a sufficiently small open set $U\subseteq O_{n}$ we may treat $\phi$ as constant (relative to $\varepsilon$) on $U$ and indeed on any $L/b(C)$-bilipschitz image of $U$. Thus, when using the condition \eqref{eq:sumtransdens'} for functions $\rho\in \E_{C,L,n}$ we will always be able to treat the terms $\rho(f_{i}^{-1}\circ f_{1}(y))$ as constant. In other words, on $U$ we will have that all functions in $\E_{C,L,n}$ are linear combinations of at most $C$ $L/b(C)$-bilipschitz Jacobians. We set $\widetilde{\phi}=\phi+\psi$ where $\psi$ is given by the conclusion of Lemma~\ref{lemma:perturb} for $\zeta=\alpha \varepsilon$, $L'=L/b(C)$ and an appropriate choice of $k$. If $B(\widetilde{\phi},\alpha\varepsilon)\cap\E_{C,L,n}$ is non-empty, then, up until addition by the `constant' $\phi$, $\psi$ is approximately a linear combination of at most $C$ $L/b(C)$-bilipschitz Jacobians on $U$. This will be incompatible with the conclusion of Lemma~\ref{lemma:perturb}.
\begin{proof}[Proof of Lemma~\ref{lemma:por}]
	Let $C,L,n\in\N$, $\phi\in C(I^{d},\R)$ and $\varepsilon\in(0,1)$. We will construct $\widetilde{\phi}\in C(I^{d},\R)$ with $\left\|\widetilde{\phi}-\phi\right\|_{\infty}\leq \varepsilon$ and $B(\widetilde{\phi},\zeta)\cap \E_{C,L,n}=\emptyset$ for a parameter $\zeta\in(0,\varepsilon)$ to be determined later in the proof. 
	
	 Using that $\phi$ is uniformly continuous, we may choose $\delta>0$ sufficiently small so that 
	\begin{linenomath} 
	\begin{equation}\label{eq:unicon}
	 \left|\phi(y)-\phi(x)\right|\leq \zeta\qquad\text{whenever }y,x\in I^{d}\text{ and }\lnorm{2}{y-x}\leq \delta.
	 \end{equation}
	\end{linenomath}
	 Next, we choose an open subset $U\subseteq O_{n}$ with $\diam(U)\leq \delta b(C)/L$.
	 
	 Let $\psi\in C(I^{d},\R)$ be given by the conclusion of Lemma~\ref{lemma:perturb} applied to $\varepsilon$, $\zeta$, $L'=L/b(C)$, $k=C$ and $U$. We define the function $\widetilde{\phi}\in C(I^{d},\R)$ by 
	\begin{linenomath}
	\begin{equation*}
	\widetilde{\phi}=\phi+\psi.
	\end{equation*}
	\end{linenomath}
	From the conclusion of Lemma~\ref{lemma:perturb} we have that $\left\|\widetilde{\phi}-\phi\right\|_{\infty}\leq\varepsilon$ and $\widetilde{\phi}=\phi$ outside of the set $U\subseteq O_{n}$. Let us now verify that $B(\widetilde{\phi},\zeta)\cap \E_{C,L,n}=\emptyset$.
	
	Let $\rho\in B(\widetilde{\phi},\zeta)$ and suppose for a contradiction that $\rho\in \E_{C,L,n}$. Choose $N\in[C]$, pairwise-disjoint, non-empty, open sets $Y_{1},\ldots,Y_{N}\subseteq I^{d}$, $V\subseteq \R^{d}$ and $(b(C),L)$-bilipschitz homeomorphisms $f_{i}\colon Y_{i}\to V$ witnessing that $\rho\in \E_{C,L,n}$. By the choice of $\psi$ and Lemma~\ref{lemma:perturb} there exist $\mb{e}_{1}$-adjacent cubes $S,S'\subseteq U\subseteq O_{n}$ such that \eqref{eq:jaccontrol} holds for each of the mappings
	\begin{linenomath}
	\begin{equation*}
	h_{i}=\begin{cases}
	f_{1} & \text{ if }i=1,\\
	f_{i}^{-1}\circ f_{1} & \text{ if } 2\leq i\leq N,
	\end{cases}
	\end{equation*}
	\end{linenomath}
	and \eqref{eq:chessboard} holds for $\psi$. Using \eqref{eq:sumtransdens'}, we may now write
	\begin{linenomath}
	\begin{multline}\label{eq:psisum}
	\psi(y)=\widetilde{\phi}(y)-\phi(y)=(\widetilde{\phi}(y)-\rho(y))+\rho(y)-\phi(y)\\
	=(\widetilde{\phi}(y)-\rho(y))+\left|\jac(h_{1}(y))\right|-\sum_{i=2}^{N}\rho(h_{i}(y))\left|\jac(h_{i})(y)\right|-\phi(y)
	\end{multline}
	\end{linenomath}
	for a.e.\ $y\in O_{n}$. To complete the proof we will show that the average value of the final expression over the cube $S$ is too close to its average value over $S'$, that is, closer than the condition \eqref{eq:chessboard} on $\psi$ allows.
	
	Let $i\in\left\{2,3,\ldots,N\right\}$. Then we have that
	\begin{linenomath}
	\begin{equation*}
	\lnorm{2}{h_{i}(z)-h_{i}(y)}\leq \frac{L}{b(C)}\lnorm{2}{z-y}\leq \frac{L}{b(C)}\diam(U)\leq \delta
	\end{equation*}
	\end{linenomath}
	whenever $y,z\in S\cup S'\subseteq U$. Therefore, in light of \eqref{eq:unicon} and the fact that $\left|\rho(x)-\phi(x)\right|=\left|\rho(x)-\widetilde{\phi}(x)\right|\leq \zeta$ for  all points $x$ in the image of $h_{i}$, we may fix $y_{0}\in S$ such that
	\begin{linenomath}
	\begin{equation}\label{eq:almconst}
	\left|\rho(h_{i}(y))-\rho(h_{i}(y_{0}))\right|\leq\left|\phi(h_{i}(y))-\phi(h_{i}(y_{0}))\right|+2\zeta\leq 3\zeta
	\end{equation}
	\end{linenomath}
	for all $y\in S\cup S'$. Thus, we have
	\begin{linenomath}
	\begin{equation}\label{eq:avvalbound}
	\begin{split}
	 &\left|\dashint_{S}\rho(h_{i}(y))\left|\jac(h_{i})(y)\right|-\dashint_{S'}\rho(h_{i}(y))\left|\jac(h_{i})(y)\right|\right|\\
	 &\leq\left|\rho(h_{i}(y_{0}))\right|\left|\dashint_{S}\left|\jac(h_{i})(y)\right|-\dashint_{S'}\left|\jac(h_{i})(y)\right|\right|+\dashint_{S}\left|\rho(h_{i}(y))-\rho(h_{i}(y_{0}))\right|\left|\jac(h_{i})(y)\right|\\
	&+\dashint_{S'}\left|\rho(h_{i}(y))
	-\rho(h_{i}(y_{0}))\right|\left|\jac(h_{i})(y)\right|\leq (\left\|\phi\right\|_{\infty}+1)\zeta+6(L/b(C))^{d}\zeta.
	\end{split}
	\end{equation}
	\end{linenomath}

	For the final inequality above we used $\left\|\rho\circ h_{i}-\phi\circ h_{i}\right\|_{\infty}\leq \zeta<1$, \eqref{eq:jaccontrol} for $h_{i}$, \eqref{eq:almconst} and $\left|\jac(h_{i})\right|\leq(L/b(C))^{d}$.

	Using $\left\|\widetilde{\phi}-\rho\right\|_{\infty}<\zeta$, \eqref{eq:jaccontrol} for $h_{1}$, \eqref{eq:avvalbound} for $i\in\left\{2,3,\ldots,N\right\}$, $N\leq C$, \eqref{eq:unicon} and $\diam(S\cup S')\leq\diam (U)<\delta$ we deduce that the average value of the right hand side of \eqref{eq:psisum} over the cube~$S$ differs from its average value over $S'$ by at most
	\begin{linenomath}
	\begin{equation*}
	2\zeta+\zeta+C(\left\|\phi\right\|_{\infty}+1+6(L/b(C))^{d})\zeta+\zeta.
	\end{equation*}
	\end{linenomath}
	However, with the setting 
	\begin{linenomath}
	\begin{equation*}
	\zeta=\frac{\varepsilon}{2(4+C(\left\|\phi\right\|_{\infty}+1+6(L/b(C))^{d}))}
	\end{equation*}
	\end{linenomath}
	this number is strictly less than $\varepsilon$, contrary to \eqref{eq:chessboard}. Thus, we conclude that
	\begin{linenomath}
	\begin{equation*}
	B\left(\widetilde{\phi},\frac{\varepsilon}{2(4+C(\left\|\phi\right\|_{\infty}+1+6(L/b(C))^{d}))}\right)\cap \E_{C,L,n}=\emptyset,
	\end{equation*}
	\end{linenomath}
	which demonstrates the porosity of $\E_{C,L,n}$ at $\phi$.
	\end{proof}
	It is now a simple task to combine the previous Lemmas for a proof of Theorem~\ref{thm:regrlzprs}.
	\begin{proof}[Proof of Theorem~\ref{thm:regrlzprs}]
		From Lemma~\ref{lemma:sigporcov} we have
		\begin{linenomath}
		\begin{equation*}
		\E=\bigcup_{C,L\in \N}\E_{C,L}\subseteq\bigcup_{C,L,n\in\N}\E_{C,L,n},
		\end{equation*}
		\end{linenomath}
		whilst Lemma~\ref{lemma:por} asserts that each of the sets in the union on the right hand side is porous.
	\end{proof}
	Readers interested in the resolution of Feige's question may proceed immediately to Section~\ref{section:feige}. In the remainder of the present section we discuss an independent topic of interest.
    \subsection*{Realisability in $L^{\infty}$ spaces.}
		\addcontentsline{toc}{subsection}{Realisability in L\^\ infinity spaces.}
	Until now we have only studied realisability in spaces of continuous functions. However, functions $\rho$ admitting a bilipschitz or Lipschitz regular solution $f\colon I^{d}\to\R^{d}$ of the equation 
	\begin{linenomath}
	\begin{equation}\label{eq:pushforward2}
	f_{\sharp}\rho\leb=\leb|_{f(I^{2})}
	\end{equation}
	\end{linenomath}
	need not be continuous. Therefore, it is natural to study the set of realisable functions in the less restrictive setting of $L^{\infty}(I^{d})$, the space of all Lebesgue measurable, real-valued functions $\rho$ defined on $I^{d}$, which are bounded with respect to the $L^{\infty}$-norm 
	\begin{linenomath}
	\begin{equation*}
	\lnorm{\infty}{\rho}:=\inf\left\{C>0\colon \left|\rho(x)\right|\leq C\text{ for a.e.\ }x\in I^{d}\right\}. 
	\end{equation*}
	\end{linenomath}
	We will prove that the set of all bilipschitz realisable functions in $L^{\infty}(I^{d})$ is a $\sigma$-porous set. For bilipschitz mappings $f$, \eqref{eq:pushforward2} is equivalent to the equation
	\begin{linenomath}
	\begin{equation}\label{eq:realisablejac}
	\left|\jac(f)\right|=\rho \quad\text{a.e.}
	\end{equation} 
	\end{linenomath}
	The question of whether Lipschitz regular realisable densities are also $\sigma$-porous, or in some sense negligible, in $L^{\infty}$ spaces remains open. 
	\begin{thm}\label{thm:linftprs}
		Let 
		\begin{linenomath}
		\begin{equation*}
		\G:=\left\{\rho\in L^{\infty}(I^{d})\colon \text{\eqref{eq:realisablejac} admits a bilipschitz solution $f\colon I^{d}\to\R^{d}$}\right\}.
		\end{equation*}
		\end{linenomath}
		Then $\G$ is a $\sigma$-porous subset of $L^{\infty}(I^{d})$. In fact $\G$ may be decomposed as a countable union of sets $(\G_{L})_{L=1}^{\infty}$ so that each $\G_{L}$ is porous at every point $\rho\in L^{\infty}(I^{d})$.
	\end{thm}
	\begin{remark*}
		For $1\leq p<\infty$, the question of whether the set of bilipschitz realisable densities is small in $L^{p}(I^{d})$ is not interesting because the set of all a.e.\ bounded functions is already $\sigma$-porous in this space.
	\end{remark*}
		The proof of Theorem~\ref{thm:linftprs} will require the following lemma, for which we recall the notation of Section~\ref{section:geometric}. We postpone the proof of the lemma until Appendix~\ref{app:realisability}, since it is based on a slightly more delicate version of the construction employed in the proof of Lemma~\ref{lemma:perturb}.
		\begin{restatable}{lemma}{Linflemma}
\label{lemma:Linflemma}
				Let $\lambda>0$, $\Sq\subseteq\mc Q_{\lambda}^{d}$ be a finite collection of tiled cubes in $I^{d}$, $\rho\in L^{\infty}(I^{d})$ and $\varepsilon>0$. Then there exists a function $\psi=\psi(\Sq,\rho,\varepsilon)\in L^{\infty}(I^{d})$ such that $\lnorm{\infty}{\psi-\rho}\leq\varepsilon$ and $\left|\dashint_{S}\psi-\dashint_{S'}\psi\right|\geq\varepsilon$ whenever $S,S'\in\Sq$ are $\mb{e}_{1}$-adjacent cubes.
		\end{restatable}
		
		\begin{proof}[Proof of Theorem~\ref{thm:linftprs}]
			We decompose $\G$ as $\G=\bigcup_{L=1}^{\infty}\G_{L}$ where 
			\begin{linenomath}
			\begin{equation*}
			\G_{L}:=\left\{\rho\in L^{\infty}(I^{d})\colon \text{\eqref{eq:realisablejac} admits an $L$-bilipschitz solution $f\colon I^{d}\to\R^{d}$}\right\}.
			\end{equation*}
			\end{linenomath}
			Fix $L\geq 1$, $\rho\in L^{\infty}(I^{d})$ and $\varepsilon>0$. We will find $\widetilde{\rho}\in L^{\infty}(I^{d})$ with $\lnorm{\infty}{\widetilde{\rho}-\rho}\leq\varepsilon$ and $B(\widetilde{\rho},\varepsilon/16)\cap\G_{L}=\emptyset$. This will verify the porosity of the set $\G_{L}$ and complete the proof of the theorem. 
			
			Let $U\subseteq I^{d}$ be an arbitrary, non-empty, open set, $\zeta=\varepsilon/2$ and let $\eta\in(0,1)$ be a parameter to be determined later in the proof. Let $r\in\N$ and the tiled families $\Sq_{1},\ldots,\Sq_{r}$ of cubes contained in $U$ be given by the conclusion of Lemma~\ref{lemma:geometric} applied to $d$, $k=1$, $L$, $\eta$ and $\zeta$. We define a sequence of functions $(\widetilde{\rho}_{i})_{i=1}^{r}$ in $L^{\infty}(I^{d})$ by
			\begin{linenomath}
			\begin{equation*}
			\widetilde{\rho}_{i}=\psi(\Sq_{i},\rho,\varepsilon) \qquad\text{for }i\in[r],
			\end{equation*}
			\end{linenomath}
			where $\psi(\Sq_{i},\rho,\varepsilon)$ is given by the conclusion of Lemma~\ref{lemma:Linflemma}. Now let $\widetilde{\rho}\in L^{\infty}(I^{d})$ be defined by
			\begin{linenomath}
			\begin{equation*}
			\widetilde{\rho}(x)=\begin{cases}
			\widetilde{\rho}_{i}(x) & \text{if }x\in\bigcup\Sq_{i}\setminus \bigcup_{j=i+1}^{r}\bigcup\Sq_{j},\,i\in[r],\\
			\rho(x) & \text{if }x\in I^{d}\setminus\bigcup_{i=1}^{r}\bigcup\Sq_{i}.
			\end{cases}
			\end{equation*}
			\end{linenomath}
			It is clear that $\lnorm{\infty}{\widetilde{\rho}-\rho}\leq\varepsilon$. Let $\phi\in B(\widetilde{\rho},\varepsilon/16)$. Then, given $i\in[r]$ and $\mb{e}_{1}$-adjacent cubes $S,S'\in \Sq_{i}$ we let $T:=S\cap\bigcup_{j=i+1}^{r}\bigcup\Sq_{j}$ and $T':=S'\cap\bigcup_{j=i+1}^{r}\bigcup\Sq_{j}$. From Lemma~\ref{lemma:geometric}, part~\ref{lemma:geometric1} we have that $\max\left\{\leb(T),\leb(T')\right\}\leq\eta\leb(S)$. We deduce
			\begin{linenomath}
			\begin{multline*}
			\left|\dashint_{S}\phi-\dashint_{S'}\phi\right|\geq\left|\dashint_{S}\widetilde{\rho}_{i}-\dashint_{S'}\widetilde{\rho}_{i}\right|-\left|\dashint_{S}(\phi-\widetilde{\rho})-\dashint_{S'}(\phi-\widetilde{\rho})\right|\\
			-\frac{1}{\leb(S)}\left|\int_{T}(\widetilde{\rho}-\widetilde{\rho}_{i})-\int_{T'}(\widetilde{\rho}-\widetilde{\rho}_{i})\right|
			\geq \varepsilon-\frac{2\varepsilon}{16}-4\varepsilon\eta>\frac{\varepsilon}{2},
			\end{multline*}
			\end{linenomath}
			when we set $\eta=\frac{1}{16}$. Together with Lemma~\ref{lemma:geometric}, part~\ref{lemma:geometric2} and the setting $\zeta=\varepsilon/2$, this implies that equation~\eqref{eq:realisablejac} with $\rho=\phi$ has no $L$-bilipschitz solutions $f\colon I^{d}\to\R^{d}$. Hence $\phi\notin \G_{L}$.
		\end{proof}
\section{Feige's question.}\label{section:feige}
In the 90's Uriel Feige asked a fascinating question~\cite[Problem 2.12]{Matousek_open}\footnote{Feige asked the question for $d=2$, as it was announced in the introduction in Question~\ref{q:Feige_orig}.}, see also~\cite[Problem 5.5]{Haifa}:
\begin{question}\label{q:Feige}
Is there a constant $L>0$ such that for every $n\in\N$ and every set $S\subset\Z^d$ such that $\abs{S}=n^d$ there is an $L$-Lipschitz bijection $f\colon S\to [n]^d$? 
\end{question}

We provide a negative solution to Question~\ref{q:Feige} in a strong sense in all dimensions $d\in\N, d\geq 2$. For $d=1$ the answer is trivially positive.
First, we transform Feige's question into a question about densities of measures supported on $I^d$. In order to do so, we adapt a construction of Burago and Kleiner~\cite{BK1} that encodes densities into separated nets in $\R^d$.
Then we prove that if we plug in any of the positive non-realisable functions whose existence is ensured by Theorem~\ref{thm:regrlzprs}, the sequence of discrete sets that arise from the chosen function provide a negative solution to Question~\ref{q:Feige}.

We provide an equivalent version of Question~\ref{q:Feige} that fits better the tools we have.
\begin{question}\label{q:Feige2}
For every $r>0$, is there a constant $L=L(r)>0$ such that for every $n\in\N$ and every $r$-separated set $S\subset \R^d$ such that $\abs{S}=n^d$ there is an $L$-Lipschitz bijection $f\colon S\to\{1,\ldots, n\}^d$?
\end{question}

The equivalence of Question~\ref{q:Feige} and Question~\ref{q:Feige2} is easy to see. Thus, we provide its formal justification only in Appendix~\ref{app:feige} in Observation~\ref{obs:Feige_equiv}.

The following theorem, in conjunction with Theorem~\ref{thm:regrlzprs}, proves that the answer to Question~\ref{q:Feige2} is negative. In other words, it proves Theorem~\ref{thm:mainresult} stated in the introduction.
\begin{thm}\label{thm:BK_construction}
Assume that the answer to Question~\ref{q:Feige2} is positive. Then for every measurable function $\rho\colon I^d\to (0,\infty)$ such that $0<\inf\rho\leq\sup\rho<\infty$ there is a Lipschitz regular mapping $f\colon I^d\to\R^d$ verifying the equation
	\begin{linenomath}
	$$
	f_\#(\rho\leb)=\rest{\leb}{f(I^d)}.
	$$  
	\end{linenomath}
\end{thm}

Before we start a presentation of the proof of Theorem~\ref{thm:BK_construction}, let us add a convenient observation, which asserts that it is sufficient to prove Theorem~\ref{thm:BK_construction} only for densities $\rho$ with average value $1$.
\begin{obs}\label{obs:normalise}
Let $\rho\colon I^{d}\to (0,\infty)$ be a measurable function. Then the equation $f_\sharp(\xi\leb)=\rest{\leb}{f(I^d)}$ admits a Lipschitz regular solution $f\colon I^{d}\to\R^{d}$ for $\xi=\rho$ if and only if it admits Lipschitz regular solutions for $\xi=\alpha\rho$ for every $\alpha>0$.
\end{obs}

\begin{proof}
We write $f\colon I^d\to\R^d$ for a Lipschitz regular mapping satisfying $f_\sharp(\rho\leb)=\rest{\leb}{f(I^d)}$.
Let us consider a mapping $\phi_\alpha(x):=\sqrt[d]{\alpha}\cdot x$ and observe that $\phi_\alpha\circ f$ is the sought after solution:
\begin{linenomath}
\begin{align*}
\br*{\phi_\alpha\circ f}_\sharp(\alpha\rho\leb)(A)&=\alpha\lint{f^{-1}\circ \phi_\alpha^{-1}(A)}{\rho}=\alpha f_\sharp(\rho\leb)\br*{\phi_\alpha^{-1}(A)}=\alpha\rest{\leb}{f(I^d)}\br*{\phi_\alpha^{-1}(A)}\\
&=\alpha\niceint{A}{\abs{\jac\phi_\alpha}}{\rest{\leb}{\phi_\alpha\circ f(I^d)}}=\rest{\leb}{\phi_\alpha\circ f(I^d)}(A)
\end{align*}
\end{linenomath}
for any measurable set $A\subseteq \br*{\phi_\alpha\circ f}(I^d)$.
\end{proof}

The proposed proof of Theorem~\ref{thm:BK_construction} relies on the discretisation procedure of Burago and Kleiner~\cite{BK1}, which they used to encode a bilipschitz non-realisable density $\rho$ into a separated net $S$ in $\R^d$ that cannot be bijectively mapped onto $\Z^d$ in a bilipschitz way. We will use their procedure, with a small technical modification, to encode a given bounded, measurable function $\rho\colon I^d\to(0,\infty)$ into a sequence of separated sets $S_i$ in $\R^d$ such that each $S_i$ has cardinality $n_i^d$ for some $n_i\in\N$.

Burago and Kleiner~\cite{BK1} showed that a bilipschitz bijection $S\to\Z^d$ would yield a bilipschitz solution to the equation $f_\sharp(\rho\leb)=\rest{\leb}{f(I^d)}$. We will show that if there are $L$-Lipschitz bijections $f_i\colon S_i\to[n_i]^d$ for some $L>0$ and every $i\in\N$, then there is also a Lipschitz regular solution to the equation $f_\sharp(\rho\leb)=\rest{\leb}{f(I^d)}$. This part of our proof is different than that of Burago and Kleiner, although we follow their overall idea\footnote{In fact, in their article~\cite[Section~2]{BK1} Burago and Kleiner do not provide the full details for that part of their argument. However, we were able to verify their result. The arguments that we present here are not sufficient in their setting, since in the present case the image $f_i(S_i)$ is much nicer.}.

The preceding explanation implies that if we use any of the positive continuous functions $\rho\in C(I^{d},\R)\setminus\E$, where $\E$ is the set from Theorem~\ref{thm:regrlzprs}, we get a sequence of separated sets that will provide a counterexample to Question~\ref{q:Feige2}, and thus, also a negative answer to the question of Feige. 

In the next paragraph, we describe our modified version of the Burago--Kleiner construction~\cite{BK1}.

\paragraph{Encoding a bounded, positive function into a sequence of separated sets.}

To begin with, let us describe the main ideas of the construction informally. Each set $S_i$ in the sequence of sets `discretising' a given bounded, measurable function $\rho\colon I^d\to(0,\infty)$ represents `a picture' of $\rho$ taken with a resolution increasing with $i$. More precisely, to define $S_i$ we first blow up the domain of $\rho$ by some factor $l_i$ and then subdivide it into $m_i^d$ cubes of the same size. Let $T$ be one of these cubes. The set $S_i\cap T$ will be then formed by regularly spaced points inside $T$ of number proportional to the average value of the blow up of $\rho$ by the factor $l_i$ over $T$. If we choose $l_i$ and $m_i$ so that $\frac{l_i}{m_i}$ goes to infinity with $i$, the set $S_i$ will capture more and more precisely variations of $\rho$ on small scales. This idea is illustrated in Figure~\ref{f:BK_construction}.

Since we want to use the sets $S_i$ in Question~\ref{q:Feige2}, we need to make sure that the total number of points forming each $S_i$ is a $d$-th power of some natural number. That's the technical difference in comparison to the original construction of Burago and Kleiner~\cite{BK1}. Now, we will write everything formally. 

\begin{figure}[htb]
\begin{center}
\includegraphics{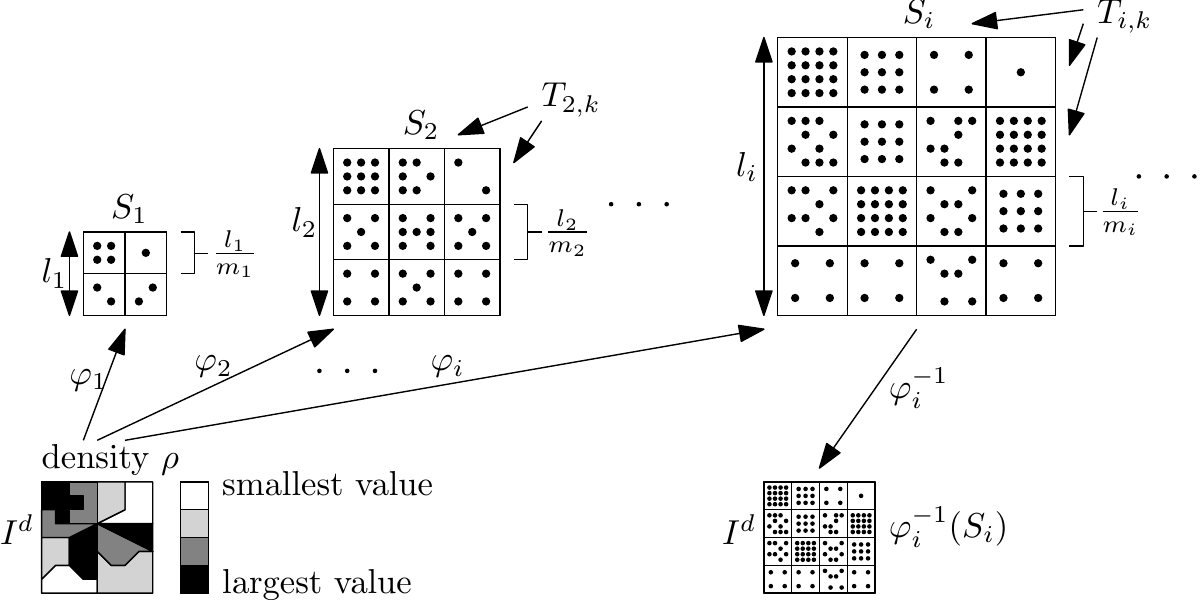}
\caption{An illustration of the construction that encodes a given density $\rho$ into a sequence of separated sets $\set{S_i}_{i=1}^\infty$. The position and the number of points inside $S_i$ approximates $\rho$, after rescaling by $\varphi_i^{-1}$, with a precision increasing with $i$.}
\label{f:BK_construction}
\end{center}
\end{figure}

We assume that we are given a measurable function $\rho\colon I^d\to (0,\infty)$ such that $0<\inf\rho\leq\sup\rho<\infty$. We choose two sequences $\{l_i\}_{i\in\N}\subset\R^+$ and $\{m_i\}_{i\in\N}\subset\N$. We put several conditions on them, which we describe a bit later. First, let us introduce some notation.

	We write $\varphi_i(x):=l_i\cdot x$ for a blow up by factor $l_i$. It is clear that $\varphi_i(I^d)=[0,l_i]^d$. Each cube $[0,l_i]^d$ naturally decomposes into $m_i^d$ cubes of side $\frac{l_i}{m_i}$; we denote them by $\{T_{i,k}\}_{k=1}^{m_i^d}$. We choose $n_{i,k}\in\left\{\floor*{\lint{T_{i,k}}{\rho\circ\varphi_i^{-1}}}, \floor*{\lint{T_{i,k}}{\rho\circ\varphi_i^{-1}}}+1 \right\}$. This number will stand for $\abs{S_i\cap T_{i,k}}$. The possibility to choose $n_{i,k}$ among two different values will allow us to ensure that $\abs{S_i}=n_i^d$ for some $n_i\in\N$.
	
	The change of variables formula implies
	\begin{linenomath}
	\begin{equation}\label{eq:bound2}
	\lint{T_{i,k}}{\rho\circ\varphi_i^{-1}}=\lint{\varphi_i^{-1}(T_{i,k})}{\rho\abs{\jac(\varphi_i)}}=l_i^d\lint{\varphi_i^{-1}(T_{i,k})}{\rho}.
	\end{equation}
	\end{linenomath}
	Using the bounds on $\rho$ we infer that
	\begin{linenomath}
	\begin{equation}\label{eq:bound1}
	\floor*{\frac{\sup\rho\cdot l_i^d}{m_i^d}}+1\geq n_{i,k}\geq\floor*{\frac{\inf\rho\cdot l_i^d}{m_i^d}}.
	\end{equation}
	\end{linenomath}
	Now we can state the required conditions on $l_i, m_i$ and $n_{i,k}$: 
	\begin{enumerate}
		\item $l_i\to\infty$, $m_i\to\infty$ and $\frac{l_i}{m_i}\to\infty$ as $i\to\infty$.
		\item for every $i\in\N$ we choose each $n_{i,k}$ from the two possibilities so that there is $n_i\in\N$ such that $n_i^d=\sum_{k=1}^{m_i^d} n_{i,k}$.
	\end{enumerate}
	
	We show that these conditions can be satisfied at once. We set $l_i:=m_i^{1+p}$ for a suitable $p>0$ and choose $\{m_i\}_{i=1}^\infty\subset\N$ as an increasing sequence. This will satisfy the first condition. 
	
	In order to satisfy the second condition, it is sufficient to make sure that the interval
	$\bs*{\sum_{k=1}^{m_i^d}\floor*{\lint{T_{i,k}}{\rho\circ\varphi_i^{-1}}}, m_i^d+\sum_{k=1}^{m_i^d}\floor*{\lint{T_{i,k}}{\rho\circ\varphi_i^{-1}}}}$
	contains a $d$-th power of a natural number. If we denote by $a_i$ the largest integer such that
	$a_i^d< \sum_{k=1}^{m_i^d}\floor*{\lint{T_{i,k}}{\rho\circ\varphi_i^{-1}}}$,
	we need that $(a_i+1)^d\leq m_i^d+\sum_{k=1}^{m_i^d}\floor*{\lint{T_{i,k}}{\rho\circ\varphi_i^{-1}}}$ . Since $(a_i+1)^d-a_i^d\leq C(d)a_i^{d-1}$, where $C(d)$ is a constant depending only on the dimension $d$, it is sufficient to choose $m_i$ so that $m_i^d>C(d)a_i^{d-1}$. From the equation \eqref{eq:bound2} we get that
	\begin{linenomath}
	$$
	a_i^d < \sup\rho\cdot l_i^d=\sup\rho\cdot m_i^{(1+p)d};
	$$
	\end{linenomath}
	thus we see that $m_i$ satisfies $m_i^d>C(d)a_i^{d-1}$ provided we choose $p<\frac{1}{d-1}$ and $m_i$ sufficiently large.
	
	After setting up the parameters $l_i$ and $m_i$ properly, we can construct the separated sets $S_i$.
	We first form sets $S_{i,k}$ by placing $n_{i,k}$ distinct points inside each $T_{i,k}$ and then set $S_i:=\bigcup_{k=1}^{m_i^d}S_{i,k}$. But instead of providing an explicit formula for $S_{i,k}$, it will be enough to consider any sufficiently separated set of $n_{i,k}$ points inside $T_{i,k}$ and argue that the separation constant can be chosen independently of $i,k$.
	
Since each $T_{i,k}$ has a side of length $\frac{l_i}{m_i}$, given $r>0$ satisfying
\begin{linenomath}
\begin{equation}\label{eq:separation}
\frac{l_{i}}{m_{i}\ceil{\sqrt[d]{n_{i,k}}}}\geq r,
\end{equation}
\end{linenomath}
we may define $S_{i,k}$ as any $r$-separated set of $n_{i,k}$ points inside $T_{i,k}$ that, in addition, satisfies $\dist\br*{{S_{i,k},\partial T_{i,k}}}\geq r/2$; an example is depicted in Figure~\ref{f:separated_set}. The last condition ensures that the set $S_i$ is $r$-separated as well.
	
\begin{figure}[htb]
\begin{center}
\includegraphics{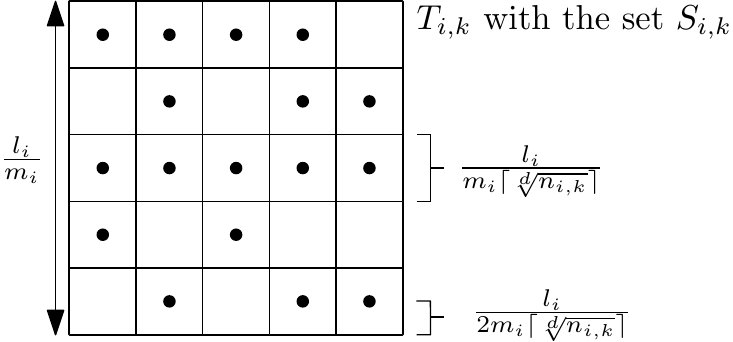}
\caption{The construction of $S_{i,k}$ inside $T_{i,k}$.}
\label{f:separated_set}
\end{center}
\end{figure}
	
	It remains to verify the existence of the separation constant $r>0$ satisfying the inequality~\eqref{eq:separation} for all $i,k$. Using the inequality~(\ref{eq:bound1}) and standard estimates we obtain that
	\begin{linenomath}
	$$
	\ceil{\sqrt[d]{n_{i,k}}}^d\leq 2^d n_{i,k}\leq2^d\br*{\frac{\sup\rho\cdot l_i^d}{m_i^d}+1}\leq2^{2d}\frac{\sup\rho\cdot l_i^d}{m_i^d}
	$$
	\end{linenomath}
	for $n_{i,k}\geq 1$, which is ensured by taking $l_i,m_i$ sufficiently large. This in turn provides us with the bound $\ceil{\sqrt[d]{n_{i,k}}}\leq\frac{4\sqrt[d]{\sup\rho}\cdot l_i}{m_i}$. Substituting this bound into the inequality~\eqref{eq:separation}, we see that we may take $r:=\frac{1}{4\sqrt[d]{\sup\rho}}$. This also finishes the description of the construction of the separated sets $S_i$.

\begin{proof}[Proof of Theorem~\ref{thm:BK_construction}]
Using Observation~\ref{obs:normalise}, we may assume that $\dashint_{I^{d}}\rho=1$.
We use the construction described above on $\rho$ and get sequences of $r$-separated sets $S_i$ and $S_{i,k}$ together with the parameters $l_i, m_i, n_i, n_{i,k}$, mappings $\varphi_i$ and cubes $T_{i,k}$.

Assuming the positive answer to Question~\ref{q:Feige2}, we get $L>0$ and a sequence of $L$-Lipschitz bijections $f_i\colon S_i\to[n_i]^d$.
	We pull each $f_i$ back to $I^d$ in the following way. We write $X_i$ for the set $\varphi_i^{-1}(S_i)$ and define a mapping
	$g_i\colon X_i\to\R^d$ as $g_i(x):=\frac{1}{n_i}\cdot f_i\circ\varphi_{i}(x)$.
	It is not hard to see that the mappings $g_i$ are uniformly Lipschitz:
	\begin{linenomath}
	$$
	\lnorm{2}{g_i(x)-g_i(y)}\leq\frac{1}{n_i}\lnorm{2}{f_i\circ\varphi_i(x)-f_i\circ\varphi_i(y)}\leq\frac{L}{n_i}\lnorm{2}{\varphi_i(x)-\varphi_i(y)}=L\frac{l_i}{n_i}\lnorm{2}{x-y},
	$$
	\end{linenomath}
	for every $x,y\in X_i$. Thus, we need to examine the behaviour of the sequence $\br*{\frac{l_i}{n_i}}_{i\in\N}$.
	
	From the definition of $n_{i,k}$ we have the following bounds on $n_i^d=\sum_{k=1}^{m_i^d}n_{i,k}$:
	\begin{linenomath}
	$$
	\lint{\varphi_i(I^d)}{\rho\circ\varphi_i^{-1}}+m_i^d\geq n_{i}^{d}\geq\lint{\varphi_i(I^d)}{\rho\circ\varphi_i^{-1}}-m_i^d.
	$$
	\end{linenomath}
	Using the identity $\dashint_{I^{d}}\rho=1$ and the fact that $\jac(\varphi_i)=l_i^d$ we immediately obtain that
	\begin{linenomath}
	$$
	l_i^d+m_i^d\geq n_{i}^{d}\geq l_i^d-m_i^d.
	$$
	\end{linenomath}
	Since $\frac{m_i}{l_i}\to 0$, the sequence $\br*{\frac{n_i}{l_i}}_{i\in\N}$ is bounded and converges to $1$, and hence, $\frac{l_i}{n_i}\to 1$ as well.
	Consequently, for any $L'>L$ we can find $i_0\in\N$ such that for every $i\geq i_0$ the mappings $g_i$ are $L'$-Lipschitz. By trimming off the initial segment of the sequence $(g_i)_{i\in\N}$ up to $i_0$, we can assume that all mappings $g_i$ are $L'$-Lipschitz for any chosen $L'>L$.
	
	We extend each $g_i$, by Kirszbraun's extension theorem~\cite{Kirszbraun1934}, \cite[2.10.43]{Fed}, to a mapping $\bar{g_i}\colon I^d\rightarrow\R^d$ such that $\lip(\bar{g_i})=\lip(g_i)$.
	By the Arzelà--Ascoli theorem we know that the sequence $\br*{\bar{g_i}}_{i=1}^\infty$ subconverges to a limit $f$, which is also $L'$-Lipschitz\footnote{In fact, it is $L$-Lipschitz, as follows from the previous discussion.}.
	By passing to a convergent subsequence, we may assume that $\bar{g_i}\rightrightarrows f$.

	Before we prove that $f$ transforms measure according to the equation
	\begin{linenomath}
	\begin{equation}\label{eq:meas_transform}
	f_\#(\rho\leb)=\rest{\leb}{f(I^d)}
	\end{equation}
	\end{linenomath}
	we have to establish additional notation and present two lemmas about weak convergence of measures\footnote{By weak convergence of finite Borel measures $\mu_i$ to a finite Borel measure $\mu$ on a metric measure space $X$ we mean the convergence of $\niceint{X}{\varphi}{\mu_i}$ to $\niceint{X}{\varphi}{\mu}$ for every $\varphi\in C(X, \R)$ with compact support.}.
	
	For $i\geq 1$ we define a measure $\mu_{i}$ on $I^{d}$ by
	\begin{linenomath}
	\begin{equation*}
	\mu_{i}(A)=\frac{1}{n_{i}^{d}}\left|A\cap X_{i}\right|, \qquad  A\subseteq I^{d}.
	\end{equation*}
	\end{linenomath}
	
	In order to show that $f_{\sharp}(\rho\leb)=\rest{\leb}{f(I^d)}$ we first prove that $\mu_{i}$ converges weakly to $\rho\leb$ on $I^d$. Moreover, this will be shown to imply that $(\bar{g_i})_{\sharp}(\mu_{i})$ converges weakly to $f_{\sharp}(\rho\leb)$. Finally, we prove that  $(\bar{g_i})_{\sharp}(\mu_i)$ also converges weakly to $\leb|_{f(I^d)}$, and hence, $f_{\sharp}(\rho\leb)$ and $\rest{\leb}{f(I^d)}$ must be the same by the uniqueness of weak limits.

To this end we use the following two lemmas, which are probably a part of a common knowledge in measure theory, but the authors were unable to find a proper reference.
Since their proofs are straightforward, we include them only in Appendix~\ref{app:feige}.
\setcounter{dummycounterA}{\value{define}} 
\begin{restatable}{lemma}{lemmaweakconvcrit}\label{lemma:weakconvcrit}
	Let $\nu$ and $(\nu_{n})_{n=1}^{\infty}$ be finite Borel measures on a compact metric space $K$.
	Moreover, assume that there is, for each $n\in\N$, a finite collection $\mc{Q}_{n}$ of Borel subsets of $K$ that cover $\nu$-almost all of $K$ and, at the same time, 
	\begin{linenomath}
	\begin{equation*}
	\sum_{Q\in\mc Q_n}\nu(Q)=\nu(K),\quad \lim_{n\to\infty}\max_{Q\in\mc Q_n}\diam(Q)=0,\quad \text{and }\max_{Q\in\mc{Q}_{n}}\abs{\nu_{n}(Q)-\nu(Q)}\in o\br*{\frac{1}{\abs{\mc{Q}_n}}}.
	\end{equation*}
	\end{linenomath}
	Then $\nu_{n}$ converges weakly to $\nu$.
\end{restatable}

\begin{restatable}{lemma}{lemmaweakconvcrittwo}\label{lemma:weakconvcrit2}
Let $K$ be a compact space and $\br{\nu_n}_{n\in\N}$ be a sequence of finite, Borel measures on $K$ converging weakly to a finite Borel measure $\nu$. Let $X$ be a metric space and $h_n\colon K\to X$ be a sequence of continuous mappings converging uniformly to $h$. Then $(h_n)_\sharp(\nu_n)$ converges weakly to $h_\sharp(\nu)$.
\end{restatable}

Equipped with the two lemmas above, we resume proving Theorem~\ref{thm:BK_construction}.
\setcounter{dummycounterB}{\value{define}} 
\setcounter{define}{\value{dummycounterA}-1} 
\begin{claim}\label{claim:weakconvdomain}
The sequence of measures $(\mu_{i})_{i\in\N}$ converges weakly to $\rho\leb$.
\end{claim}
\begin{proof}
We verify that the measures $\mu_i$ and $\rho\leb$ satisfy the assumptions of Lemma~\ref{lemma:weakconvcrit}. The only non-trivial assumption in this case is the existence of the collection $\mc Q_i$. We take $\mc Q_i:=\set{\varphi_i^{-1}(T_{i,k})\colon k\in[m_i^d]}$. Clearly, the sets $\varphi_i^{-1}(T_{i,k})$, which form a partition of $I^d$ into $m_i^d$ cubes of side $1/m_i$, cover the whole $I^d$ and satisfy $\sum_{Q\in\mc Q_i}\rho\leb(Q)=\rho\leb(I^d)$. Since $m_i$ goes to infinity with $i$, the diameter of $\varphi_i^{-1}(T_{i,k})$ goes to zero. 

It remains to check that $\max_{Q\in\mc{Q}_{i}}\abs{\mu_{i}(Q)-\rho\leb(Q)}\in o\br*{\frac{1}{m_i^d}}$. To see this, recall that $\mu_i$ is supported on $\varphi_i^{-1}(T_{i,k})$ by the set $\varphi_i^{-1}(T_{i,k})\cap X_i$ consisting of $n_{i,k}$ points. For any $i\in\N, k\in[m_i^d]$ we can write
\begin{linenomath}
\begin{align*}
\mu_i(\varphi_i^{-1}(T_{i,k}))=&\frac{n_{i,k}}{n_{i}^d}\leq  \frac{1}{n_{i}^d} \br*{\floor*{\lint{T_{i,k}}{\rho\circ\varphi_i^{-1}}}+1}\\
\leq& \frac{l_i^d}{n_{i}^{d}}\lint{\varphi_{i}^{-1}(T_{i,k})}{\rho}+\frac{1}{n_i^d}
\leq\frac{l_i^d}{n_i^d}\rho\leb(\varphi_{i}^{-1}(T_{i,k}))+\frac{1}{n_i^d},
\end{align*}
\end{linenomath}
and similarly,
\begin{linenomath}
\begin{align*}
\mu_i(\varphi_i^{-1}(T_{i,k}))\geq  \frac{1}{n_{i}^d} \br*{\lint{T_{i,k}}{\rho\circ\varphi_i^{-1}}-1}
\geq \frac{l_i^d}{n_i^d}\rho\leb(\varphi_{i}^{-1}(T_{i,k}))-\frac{1}{n_i^d},
\end{align*}
\end{linenomath}
Therefore, we can bound $\abs{\mu_i(\varphi_i^{-1}(T_{i,k}))-\rho\leb(\varphi_i^{-1}(T_{i,k}))}$ above as
\begin{linenomath}
$$
\frac{1}{n_i^d}+\abs{\frac{l_i^d}{n_i^d}\rho\leb(\varphi_i^{-1}(T_{i,k})-\rho\leb(\varphi_i^{-1}(T_{i,k}))}= \frac{1}{n_i^d}+\rho\leb(\varphi_i^{-1}(T_{i,k}))\abs{\frac{l_i^d}{n_i^d}-1}.
$$
\end{linenomath}
Noting that $\rho\leb(\varphi_{i}^{-1}(T_{i,k}))\leq\frac{\sup\rho}{m_i^d}$, $\frac{l_i}{m_i}\to\infty$ and $\frac{l_i}{n_i}\to 1$ as $i\to\infty$, this proves that
\begin{linenomath}
$$
\abs{\mu_i(\varphi_i^{-1}(T_{i,k}))-\rho\leb(\varphi_i^{-1}(T_{i,k}))}\in o\br*{\frac{1}{m_i^d}}.
$$
\end{linenomath}
Hence, $\mu_i$ and $\rho\leb$ satisfy the assumptions of Lemma~\ref{lemma:weakconvcrit}.
\end{proof}

	\begin{claim}\label{claim:weakconvimage}
	The sequence of measures $\br*{(\bar{g_i})_{\sharp}(\mu_{i})}_{i\in\N}$ converges weakly to $f_{\sharp}(\rho\leb)$.
	\end{claim}
	\begin{proof}
	We know that $\bar{g_i}\rightrightarrows f$ and that $\mu_i$ converge weakly to $\rho\leb$ by Claim~\ref{claim:weakconvdomain}. Thus, it is sufficient to directly apply Lemma~\ref{lemma:weakconvcrit2}.
	\end{proof}
	
	\begin{claim}
	The sequence of measures $\br*{(\bar{g_i})_\sharp(\mu_i)}_{i\in\N}$ converges weakly to $\leb|_{f(I^d)}$.
	\end{claim}
	\begin{proof}

We first observe that $(\bar{g_i})_\sharp(\mu_i)$ converges weakly to $\rest{\leb}{I^d}$. To see this, note that for every $i\in\N$ the set $f_i(S_i)$ is exactly the set $[n_i]^d$. Therefore, the set $g_i(X_i)$, which is the support of the measure $(\bar{g_i})_\sharp(\mu_i)$, is precisely the set $\set{\frac{1}{n_i}, \frac{2}{n_i}, \ldots, \frac{n_i}{n_i}}^d$, that is, a~regular grid with $n_i^d$ points inside $I^d$. The situation is depicted in Figure~\ref{f:convergence_to_lebesgue}.

Since the weight assigned to each point of $X_i$ by $\mu_i$ is exactly $\frac{1}{n_i^d}$, it is intuitively clear that $(\bar{g_i})_\sharp(\mu_i)$ converges weakly to $\rest{\leb}{I^d}$.
For a formal justification, the conditions of Lemma~\ref{lemma:weakconvcrit} are easily verified for  $\nu_{n}:=(\bar{g_n})_\sharp(\mu_n)$, $\nu:=\rest{\leb}{I^d}$ and 
	\begin{linenomath}
	$$
	\mc Q_{n}:=\set{\prod_{j=1}^d\left(\frac{b_j-1}{n_i}, \frac{b_j}{n_i}\right]\colon (b_{1},\ldots,b_{d})\in[n_i]^d}.
	$$
	\end{linenomath}

\begin{figure}[htb]
\begin{center}
\includegraphics{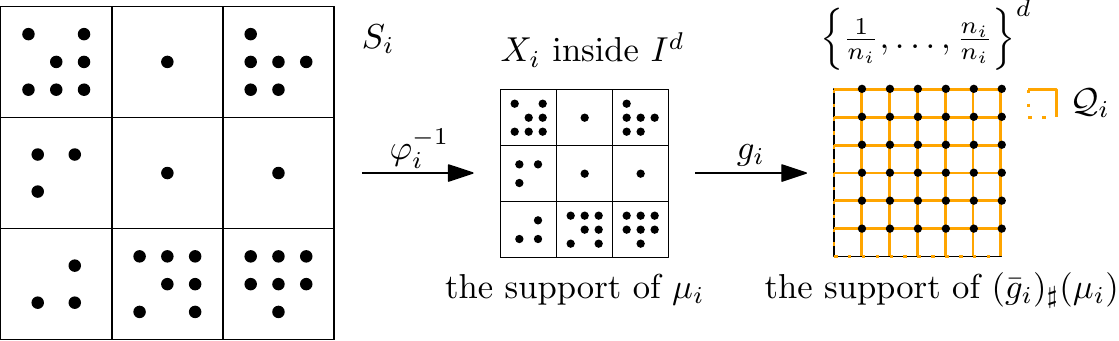}
\caption{Weak convergence of $(\bar{g_i})_\sharp(\mu_i)$ to $\rest{\leb}{I^d}$.}
\label{f:convergence_to_lebesgue}
\end{center}
\end{figure}

Combining Claim~\ref{claim:weakconvimage} and the last observation we infer that $f_\sharp(\rho\leb)=\rest{\leb}{I^d}$ by the uniqueness of weak limits. 

It remains to observe that $f(I^d)=I^d$. The equation $f_\sharp(\rho \leb)=\leb|_{I^{d}}$ and $\rho>0$ imply that for any positive measure set $A\subseteq I^{d}$, both the image $f(A)$ and the pre-image $f^{-1}(A)$ have positive measure with respect to the measure $\leb|_{I^{d}}$. It follows that $I^{d}$ is dense in $f(I^{d})$ and vice-versa. Since both sets are closed, they must coincide.
\end{proof}
\setcounter{define}{\value{dummycounterB}} 
	
	Since we have already shown that $f_\sharp(\rho\leb)=\leb|_{f(I^d)}$, we get for every measurable set $A\subseteq f(I^d)$ that
	\begin{linenomath}
	$$
	\leb(A)=f_\sharp(\rho\leb)(A)=\lint{f^{-1}(A)}{\rho}\geq\leb(f^{-1}(A))\inf\rho;
	$$
	\end{linenomath}
	thus, we have shown that $\leb(f^{-1}(A))\leq\frac{\leb(A)}{\inf\rho}$. Applying Lemma~\ref{l:meas_pres_is_regular} we conclude that $f$ is Lipschitz regular. This finishes the proof of Theorem~\ref{thm:BK_construction}.
\end{proof}

\begin{remark}\label{rmk:kirszbraun_bilip}
	A key ingredient in the proof of Theorem~\ref{thm:BK_construction} is Kirszbraun's Extension Theorem for Lipschitz mappings~\cite{Kirszbraun1934}. We remark that there is no bilipschitz analogue of Kirszbraun's Theorem: In general a bilipschitz mapping between subsets of Euclidean spaces cannot be extended to a bilipschitz homoemorphism (see \cite[p.~1]{daneri2015planar} for a nice example). The existence of a bilipschitz extension in some special cases in the discrete setting is discussed in~\cite{navas2016}.
\end{remark}

\bigskip

\subsection*{Acknowledgements.}
\addcontentsline{toc}{section}{Acknowledgements.}
The authors would like to thank Uriel Feige for posing the question on which this work is based and for providing them with details of its motivation.
Moreover, they would like to thank Florian Baumgartner for fruitful discussions during the early stages of the presented work and Martin Tancer for his helpful remarks regarding the manuscript.
Furthermore, they would like to thank Guy C.\ David for pointing out the reference to the article~\cite{bonk_kleiner2002} of Bonk and Kleiner. Last but not least, the authors thank the anonymous referees for careful reading and many helpful suggestions that improved the article.

The authors wish to place on record their gratitude to Jirka Matoušek, whose deep insight into the problem gave us the right direction from the start of the presented research. The second named author would also like to express his gratitude to Jirka Matoušek for valuable guidance during his undergraduate studies.
\appendix
\section*{Appendices}
\addcontentsline{toc}{section}{Appendices}
\section{Appendix to Section~\ref{section:regular}: Lipschitz regular mappings.}
\label{app:regular}

\propdeglipregval*
\begin{proof}
	If $f^{-1}(\left\{y\right\})=\emptyset$ then $\deg(f,U,y)=0$, by Proposition~\ref{prop:deg}, part~\eqref{degzero}, and the formula holds. Thus, we assume that $f^{-1}(\left\{y\right\})\neq\emptyset$. Note that $f^{-1}(\left\{y\right\})$ is finite. Otherwise we may find an accumulation point $x$ of $f^{-1}(\left\{y\right\})$. Then $f(x)=y$ and there is a sequence $(x_{n})_{n=1}^{\infty}$ in $f^{-1}(\left\{y\right\})\setminus\left\{x\right\}$ such that $x_{n}\to x$ as $n\to\infty$.
	Since $f(x_n)-f(x)=0$, by differentiability at $x$ we have	$\lnorm{2}{Df(x)(x_n-x)}\in o(\lnorm{2}{x_n-x})$. It follows that $\inf_{y\in S^{d-1}}\lnorm{2}{Df(x)(y)}=0$, and thus, $Df(x)$ is not invertible. 
	
	Let us write $f^{-1}(\left\{y\right\})=\left\{x_{1},x_{2},\ldots,x_{n}\right\}$ where $n\in\N$. Fix $i\in[n]$.
	Given $\varepsilon>0$, we may choose $\delta_{i}>0$ sufficiently small so that
	\begin{linenomath}
	\begin{equation*}
	\lnorm{2}{f(x)-f(x_i)-Df(x_i)(x-x_i)}\leq \varepsilon\lnorm{2}{x-x_i}
	\end{equation*}
	\end{linenomath}
	for all $x\in \cl{B}(x_i,\delta_i)\subseteq U$.	If we define an affine mapping $g\colon \overline{U}\to\R^{d}$ by $g(x)=f(x_i)+Df(x)(x-x_i)$, the inequality above yields $\lnorm{2}{f(x)-g(x)}\leq\varepsilon\lnorm{2}{x-x_i}$ for every $x\in \cl{B}(x_i,\delta_i)$. On the other hand, from the above inequality, we can also deduce that
	\begin{linenomath}
	\begin{equation*}
	\lnorm{2}{f(x)-f(x_i)}\geq \frac{1-\varepsilon\opnorm{Df(x_i)^{-1}}}{\opnorm{Df(x_i)^{-1}}}\lnorm{2}{x-x_i}>\varepsilon\lnorm{2}{x-x_{i}}
	\end{equation*}
	\end{linenomath}
	for all $x\in \overline{B}(x_{i},\delta_{i})$, where the final inequality is obtained by choosing $\varepsilon$ sufficiently small. Therefore, for all $\delta\in(0,\delta_{i}]$, we have that $\lnorm{\infty}{\rest{f}{B(x_i,\delta)}-\rest{g}{B(x_i,\delta)}}\leq\varepsilon\delta<\dist(f(x_i),f(\partial B(x_i,\delta)))$. Applying Proposition~\ref{prop:deg}, part~(\ref{degagree}) we infer that	
\begin{linenomath}
\begin{equation*}
	\deg(f,B(x_{i},\delta),y)=\deg(g,B(x_{i},\delta),y)=\sign(\jac(g)(x_{i}))=\sign(\jac(f)(x_{i}))
	\end{equation*}
	\end{linenomath}
	for every $\delta\in(0,\delta_{i}]$. In the above we used the formula~\eqref{eq:degreeforC1regpoints} for the degree function of mappings in $C^{1}(\overline{U},\R^{d})$. Next, we choose $\delta<\min\left\{\delta_{1},\ldots,\delta_{n}\right\}$ sufficiently small so that the sets $(B(x_{i},\delta))_{i=1}^{n}$ are pairwise disjoint subsets of $U$. In the case $n=1$ we choose a (possibly empty) open set $U_{2}\subseteq U\setminus B(x_{1},\delta)$ and use $y\in \R^{d}\setminus f(\overline{U_{2}})$, Proposition~\ref{prop:deg}, part~\eqref{degzero} and property (\ref{degprop2}) with $U_{1}=B(x_{1},\delta)$ to obtain the desired result. When $n>1$, we iteratively apply property (\ref{degprop2}) to get
	\begin{linenomath}
	\begin{equation*}
	\deg(f,U,y)=\sum_{i=1}^{n}\deg(f,B(x_{i},\delta),y)=\sum_{i=1}^{n}\sign(\jac(f)(x_{i}))=\sum_{x\in f^{-1}(\left\{y\right\})}\sign(\jac(f(x))).
	\end{equation*}
	\end{linenomath}
\end{proof}

\lemmaball*
\begin{proof}
Without loss of generality, we assume that $a=0$ and that $f(0)=0$. We write $z(y):=f(y)-Df(0)(y)$. For every $y\in\R^{d}$ we have $\lnorm{2}{z(y)}\in o(\lnorm{2}{y})$.
We pick $\beta>0$ small, whose precise value will be set later, and choose $\delta_0$ small enough so that $\lnorm{2}{z(y)}\leq \beta\lnorm{2}{y}$ for every $y\in B(0,\delta_0)$.

Since the linear mapping $Df(0)$ has full rank, its inverse $Df(0)^{-1}$ is well defined and has finite norm. We get that $\br*{Df(0)^{-1}\circ f}(y)-y=Df(0)^{-1}(z(y))$.
We fix $\delta \in(0,\delta_0]$.
Let us write $g_\delta(y):=\frac{1}{\delta}\br*{Df(0)^{-1}\circ f}(\delta y)$. 
The mapping $g_\delta$ is defined on the ball $\cl{B}(0,1)$ and continuous. We also get that
\begin{linenomath}
$$
\lnorm{\infty}{g_\delta-\id}\leq\frac{\beta\delta\opnorm{Df(0)^{-1}}}{\delta}=\beta\opnorm{Df(0)^{-1}}.
$$
\end{linenomath}

We set $\beta:=\frac{1}{4\opnorm{Df(0)^{-1}}}$. Applying Lemma~\ref{l:BU} we get that
$g_\delta\br*{B(0,1)}\supseteq B\br*{0,\frac{1}{2}}$; hence, we infer that $\br*{Df(0)^{-1}\circ f}\br*{B(0,\delta)}\supseteq B\br*{0,\frac{\delta}{2}}$. It is not hard to observe that  $\inf_{x\in S^{d-1}} \lnorm{2}{Df(0)(x)}=\frac{1}{\opnorm{Df(0)^{-1}}}$. Consequently, $f\br*{B(0,\delta)}\supseteq B\br*{0,\frac{\delta}{2\opnorm{Df(0)^{-1}}}}$.
\end{proof}

\begin{lemma}\label{lemma:overlap}
	Let $U_{1},U_{2}\subseteq \R^{d}$ be disjoint, non-empty open balls and $f\colon \overline{U_{1}}\cup \overline{U_{2}}\to\R^{d}$ be a continuous mapping such that $f|_{U_{i}}$ is a homeomorphism for $i=1,2$ and $f(U_{1})=f(U_{2})$. Then there exists $s>0$ such that for any continuous mapping $h\colon \overline{U_{1}}\cup \overline{U_{2}}\to\R^{d}$ with $\lnorm{\infty}{h-f}<s$ we have $h(U_{1})\cap h(U_{2})\neq \emptyset$.
\end{lemma}
\begin{proof}
Fix $y\in f(U_{1})=f(U_{2})$.
The multiplication theorem for the topological degree of a composition, see e.g.~\cite[Theorem~2.10]{FonsecaDeg}, implies that the degree of the composition of two injective mappings is equal to the `chain rule type product' of the degrees of the two mappings. Applying this to $\br*{\rest{f}{U_i}}^{-1}\circ \rest{f}{U_i}=\id$ and using the property~\eqref{degprop1} of the degree we have that $\deg(f,U_{i},y)\in\left\{-1,1\right\}$ for $i=1,2$.

	Let $s>0$ be small enough so that $B(y,s)\subseteq f(U_{1})=f(U_{2})$ and $h\colon \overline{U_{1}}\cup \overline{U_{2}}\to\R^{d}$ be a continuous mapping with $\lnorm{\infty}{h-f}<s<\dist(y,f(\partial U_{i}))$ for $i=1,2$. Then we may apply Proposition~\ref{prop:deg}, part~\eqref{degagree} to obtain $\deg(h,U_{i},y)=\deg(f,U_{i},y)\in\left\{-1,1\right\}$ for $i=1,2$. Finally, we use Proposition~\ref{prop:deg}, part~\eqref{degzero} to deduce that $y\in h(U_{1})\cap h(U_{2})\neq\emptyset$.
\end{proof}
\section{Appendix to Section~\ref{section:geometric}: Geometric properties of bilipschitz mappings.}
\label{app:geometric}
\lemmaonedimdich*
\begin{proof}
	Let $M\in\N$ and $\varphi\in(0,1)$ be parameters to be determined later in the proof. Let $c>0$, $n\in\N$, $N\in \N$ and $h\colon[0,c]\to\R^{n}$ be an $L$-bilipschitz mapping. The assertion of the Lemma holds for $h$ if and only if the assertion holds for $\rho\circ h$, where $\rho\colon\R^{n}\to\R^{n}$ is any distance preserving transformation. Therefore, we may assume that $h(0)=(0,0,\ldots,0)$ and $h(c)=(A,0,\ldots,0)$ where $A:=\left\|h(c)-h(0)\right\|_{2}$.
	
	Assume that the second statement does not hold for $h$. In other words we have that
	\begin{linenomath}
	\begin{equation}\label{eq:not2}
	\frac{\left\|h(x+\frac{c}{NM})-h(x)\right\|_{2}}{\frac{c}{NM}}\leq(1+\varphi)\frac{A}{c}
	\end{equation} 
	\end{linenomath}
	for all $x\in\frac{c}{NM}\Z\cap[0,c-\frac{c}{NM}]$. We complete the proof, by verifying that the first statement holds for $h$. 
	
	For later use, we point out that \eqref{eq:not2} implies
	\begin{linenomath}
	\begin{equation}\label{eq:lipbdgrid}
	\left\|h(b)-h(a)\right\|_{2}\leq (1+\varphi)\frac{A}{c}\left\|b-a\right\|_{2}
	\end{equation}
	\end{linenomath}
	whenever $a,b\in\frac{c}{NM}\Z\cap[0,c]$. Let $S_{i}=\left[\frac{(i-1)c}{N},\frac{ic}{N}\right]$ for $i\in[N]$, $t=t(L,\varepsilon)\in(\varphi,1)$ be some parameter to be determined later in the proof and 
	\begin{linenomath}
	\begin{equation*}
	P=\left\{x\in \frac{c}{NM}\Z\cap\left[0,c-\frac{c}{N}\right]\colon h^{(1)}\left(x+\frac{c}{N}\right)-h^{(1)}(x)>\frac{(1-t)A}{N}\right\}.
	\end{equation*}
	\end{linenomath}
	For $x\in P$ we have 
	\begin{linenomath}
	\begin{equation*}
	\left|h^{(1)}\left(x+\frac{c}{N}\right)-h^{(1)}(x)-\frac{A}{N}\right|\leq\frac{tA}{N}.
	\end{equation*}
	\end{linenomath}
	This inequality follows from the definition of $P$, the inequality \eqref{eq:lipbdgrid} and $t>\varphi$. For the remaining co-ordinate functions we have
	\begin{linenomath}
	\begin{equation*}
	\sum_{i=2}^{n}\left|h^{(i)}\left(x+\frac{c}{N}\right)-h^{(i)}(x)\right|^{2}\leq \frac{(1+\varphi)^{2}A^{2}}{N^{2}}-\frac{(1-t)^{2}A^{2}}{N^{2}}\leq \frac{4tA^{2}}{N^{2}}.
	\end{equation*}
	\end{linenomath}
	Combining the two inequalities above we deduce
	\begin{linenomath}
	\begin{equation}\label{eq:P}
	\left\|h\left(x+\frac{c}{N}\right)-h(x)-\frac{1}{N}(h(c)-h(0))\right\|_{2}\leq\frac{\sqrt{t^{2}+4t}A}{N}\leq\frac{\sqrt{5t}Lc}{N}\qquad \forall x\in P.
	\end{equation}
	\end{linenomath}
	Let $\Gamma\subset[0,1]$ be a maximal $c/N$-separated subset of $\frac{c}{NM}\Z\cap\left[0,c-\frac{c}{N}\right]\setminus P$ and let $x_{1},\ldots,x_{\left|\Gamma\right|}$ be the elements of $\Gamma$. Then the intervals $([x_{i},x_{i}+\frac{c}{N}])_{i=1}^{\left|\Gamma\right|}$ can only intersect in the endpoints.
	Therefore the set $[0,c]\setminus \bigcup_{i=1}^{\left|\Gamma\right|}[x_{i},x_{i}+\frac{c}{N}]$ is a finite union of intervals with endpoints in $\frac{c}{NM}\Z\cap[0,c]$ and with total length $c-\frac{\left|\Gamma\right|c}{N}$. Using $\Gamma\cap P=\emptyset$ and \eqref{eq:lipbdgrid} we deduce that
	\begin{linenomath}
	\begin{align*}
	A&=h^{(1)}(c)-h^{(1)}(0)\leq\left|\Gamma\right|\frac{(1-t)A}{N}+(1+\varphi)\frac{A}{c}\left(c-\frac{\left|\Gamma\right|c}{N}\right).
	\end{align*}
	\end{linenomath}
	Rearranging this inequality, we obtain
	\begin{linenomath}
	\begin{equation*}
	\left|\Gamma\right|\leq\frac{\varphi}{\varphi+t}N\leq\frac{2\varphi}{\varphi+t}(N-1),
	\end{equation*}
	\end{linenomath}
	where, for the last inequality, we apply $N\geq 2$. It follows that the set $\frac{c}{NM}\Z\cap[0,c-\frac{c}{N}]\setminus P$ can intersect at most $\frac{6\varphi}{\varphi+t}(N-1)$ intervals $S_{i}$. Letting 
	\begin{linenomath}
	\begin{equation*}
	\Omega=\left\{i\in[N-1]\colon \frac{c}{NM}\Z\cap S_{i}\subseteq P\right\}
	\end{equation*}
	\end{linenomath}
	we deduce that $\left|\Omega\right|\geq (1-\frac{6\varphi}{\varphi+t})(N-1)$. Moreover for any $i\in\Omega$ and $x\in S_{i}$, we can find $x'\in P$ with $\left|x'-x\right|\leq c/NM$. This allows us to apply \eqref{eq:P} to get
	\begin{linenomath}
	\begin{multline*}
	\left\|h\left(x+\frac{c}{N}\right)-h(x)-\frac{1}{N}(h(c)-h(0))\right\|_{2}\\
	\leq\left\|h\left(x'+\frac{c}{N}\right)-h(x')-\frac{1}{N}(h(c)-h(0))\right\|_{2}+\frac{2Lc}{NM}\leq\frac{c(\sqrt{5t}L+\frac{2L}{M})}{N}.
	\end{multline*}
	\end{linenomath}
	We are now ready to specify the parameters $t$, $M$ and $\varphi$, so that the inequalities obtained above verify statement~\ref{1dimdich1}. First, we prescribe that $t\in(0,1)$ is sufficiently small and $M\in \N$ is sufficiently large so that $\sqrt{5t}L+\frac{2L}{M}<\varepsilon$. Finally we demand that $\varphi\in(0,t)$ is small enough so that $\frac{6\varphi}{\varphi+t}<\varepsilon$.
\end{proof}

\lemmadichotomy*

\begin{proof}
	In this proof we will sometimes add the superscript $d$ or $d-1$ to objects such as the Lebesgue measure $\leb$ or vectors $\mb{e}_{i}$, $\mb{0}$ in order to emphasise the dimension of the Euclidean space to which they correspond. For $d\geq 2$, we will express points in $\R^{d}$ in the form $\mb{x}=(x_{1},x_{2},\ldots,x_{d})$. Given $\mb{x}=(x_{1},\ldots,x_{d})\in \R^{d}$ and $s\in\R$ we let
	\begin{linenomath}
	\begin{equation*}
	\mb{x}\wedge s=(x_{1},\ldots,x_{d},s)
	\end{equation*}
	\end{linenomath}
	denote the point in $\R^{d+1}$ formed by concatenation of $\mb{x}$ and $s$.
	
	The case $d=1$ is dealt with by Lemma~\ref{lemma:1dimdich}. Let $d\geq 2$ and suppose that the statement of the lemma holds when $d$ is replaced with $d-1$. Given $L\geq 1$ and $\varepsilon>0$ we let $M=M(d,L,\varepsilon)\in\N$, $\varphi=\varphi(d,L,\varepsilon)\in(0,1)$ and $N_{0}(d,L,\varepsilon)\in\N$ be parameters on which we impose various conditions in the course of the proof. For now, we just prescribe that $N_{0}\geq N_{0}(d-1,L,\theta)$, $0<\varphi<\frac{1}{2}\varphi(d-1,L,\theta)$ and $M\in M_{d-1}\N$ for $M_{d-1}:=M(d-1,L,\theta)$, where $\theta=\theta(d,L,\varepsilon)$ is an additional parameter to be determined later in the proof. It is important to choose $M$ as a multiple of $M_{d-1}$ so that $\frac{c}{NM_{d-1}}\Z\subseteq\frac{c}{NM}\Z$ whenever $N\in\N$.
	
	Let $c>0$, $n\geq d$, $N\geq N_{0}$ and $h\colon [0,c]\times [0,c/N]^{d-1}\to\R^{n}$ be an $L$-bilipschitz mapping. For each $s\in[0,c/N]$ we apply the induction hypothesis to the mapping $h\wedge s\colon [0,c]\times[0,c/N]^{d-2}\to\R^{n}$ defined by
	\begin{linenomath}
	\begin{equation*}
	h\wedge s(\mb{x})=h(\mb{x}\wedge s)=h(x_{1},x_{2},\ldots,x_{d-1},s).
	\end{equation*} 
	\end{linenomath}
	Thus, we get that for each $s\in[0,c/N]$ at least one of the following statements holds:
	\begin{itemize}
		\item[(\mylabel{dich1ind}{1$_{s}$})]There exists a set $\Omega_{s}\subset[N-1]$ with $\left|\Omega_{s}\right|\geq (1-\theta)(N-1)$ such that for all $i\in\Omega_{s}$
		\begin{linenomath}
		\begin{equation*}
		\left\|h\wedge s\left(\mb{x}+\frac{c}{N}\mb{e}_{1}^{d-1}\right)-h\wedge s(\mb{x})-\frac{1}{N}(h\wedge s(c\mb{e}_{1}^{d-1})-h\wedge s(\mb{0}^{d-1}))\right\|_{2}\leq\frac{c\theta}{N}
		\end{equation*}
		\end{linenomath}
		for all $\mb{x}\in \left[\frac{(i-1)c}{N},\frac{ic}{N}\right]\times\left[0,\frac{c}{N}\right]^{d-2}$.
		\item[(\mylabel{dich2ind}{2$_{s}$})] There exists $\mb{z}_{s}\in\frac{c}{NM_{d-1}}\Z^{d-1}\cap([0,c-\frac{c}{NM_{d-1}}]\times[0,\frac{c}{N}-\frac{c}{NM_{d-1}}]^{d-2})$ such that
		\begin{linenomath}
		\begin{equation*}
		\frac{\left\|h\wedge s(\mb{z}_{s}+\frac{c}{NM_{d-1}}\mb{e}_{1}^{d-1})-h\wedge s(\mb{z}_{s})\right\|_{2}}{\frac{c}{NM_{d-1}}}>(1+2\varphi)\frac{\left\|h\wedge s(c\mb{e}_{1}^{d-1})-h\wedge s(\mb{0}^{d-1})\right\|_{2}}{c}.
		\end{equation*}
		\end{linenomath}
	\end{itemize}
	Suppose first that statement~\eqref{dich2ind} holds for some $s\in [0,c/N]$. Then we choose a number $s'\in\frac{c}{NM}\Z\cap[0,\frac{c}{N}-\frac{c}{NM}]$ with $s'\leq s$ and $\left|s'-s\right|\leq \frac{c}{NM}$. Setting $\mb{w}=\mb{z}_{s}\wedge s'$ we note that $\mb{w}$ is an element of $\frac{c}{NM}\Z^{d}\cap[0,c-\frac{c}{NM_{d-1}}]\times[0,\frac{c}{N}-\frac{c}{NM}]^{d-1}$, $\left\|\mb{w}-\mb{z}_{s}\wedge s\right\|_{2}\leq\frac{c}{NM}$ and $\lnorm{2}{h\wedge s(c\mb{e}_{1}^{d-1})-h\wedge s(\mb{0}^{d-1})}\geq\lnorm{2}{h(c\mb{e}_{1}^{d})-h(\mb{0}^{d})}-\frac{2Lc}{N}$. We use these inequalities and the inequality of~\eqref{dich2ind} to derive
	\begin{linenomath}
	\begin{align*}
	&\left\|h\left(\mb{w}+\frac{c}{NM_{d-1}}\mb{e}_{1}^{d}\right)-h(\mb{w})\right\|_{2}\geq\left\|h\wedge s\left(\mb{z}_{s}+\frac{c}{NM_{d-1}}\mb{e}_{1}^{d-1}\right)-h\wedge s(\mb{z}_{s})\right\|_{2}-\frac{2Lc}{NM}\\
	&> (1+2\varphi)\left(\frac{\left\|h(c\mb{e}_{1}^{d})-h(\mb{0}^{d})\right\|_{2}}{NM_{d-1}}-\frac{2Lc}{N^{2}M_{d-1}}\right)-\frac{2Lc}{NM}\\
	&\geq\left(1+2\varphi-\frac{2(1+2\varphi)Lc}{N\lnorm{2}{h(c\mb{e}_{1}^{d})-h(\mb{0}^{d})}}-\frac{2LcM_{d-1}}{M\lnorm{2}{h(c\mb{e}_{1}^{d})-h(\mb{0}^{d})}}\right)\frac{\left\|h(c\mb{e}_{1}^{d})-h(\mb{0}^{d})\right\|_{2}}{NM_{d-1}}\\
	&\geq\left(1+2\varphi-\frac{2(1+2\varphi)L^{2}}{N}-\frac{2L^{2}M_{d-1}}{M}\right)\frac{\left\|h(c\mb{e}_{1}^{d})-h(\mb{0}^{d})\right\|_{2}}{NM_{d-1}}\\
	&>(1+\varphi)\frac{\left\|h(c\mb{e}_{1}^{d})-h(\mb{0}^{d})\right\|_{2}}{NM_{d-1}}.
	\end{align*}
	\end{linenomath}
	To deduce the fourth inequality in the sequence above we use the lower bilipschitz bound on $h$. In fact, this is the only place in the proof of Lemma~\ref{lemma:dichotomy} where we use that the mapping $h$ is bilipschitz and not just Lipschitz. The final inequality is ensured by taking $N_{0}$ and $M$ sufficiently large (after fixing $\varphi$). From the final lower bound obtained for the quantity $\left\|h(\mb{w}+\frac{c}{NM_{d-1}}\mb{e}_{1})-h(\mb{w})\right\|_{2}$ it follows that there exists $i\in\left[\frac{M}{M_{d-1}}\right]$ such that the point $\mb{z}:=\mb{w}+\frac{(i-1)c}{NM}\mb{e}_{1}$ verifies statement~\ref{dich2} for $h$. 
	
	We may now assume that the first statement~\eqref{dich1ind} holds for all $s\in[0,c/N]$. We complete the proof by verifying statement \ref{dich1} for $h$. Whenever $\mb{x}\in[0,c]\times[0,c/N]^{d-2}$ and $s\in[0,c/N]$ satisfy the inequality of~\eqref{dich1ind} we have that
	\begin{linenomath}
	\begin{equation}\label{eq:gdpts}
	\left\|h((\mb{x}\wedge s)+\frac{c}{N}\mb{e}_{1}^{d})-h(\mb{x}\wedge s)-\frac{1}{N}(h(c\mb{e}_{1}^{d})-h(\mb{0}^{d}))\right\|\leq\frac{c\theta}{N}+\frac{2Lc}{N^{2}}.
	\end{equation}
	\end{linenomath}
	From this point onwards, let $R$ denote the cuboid $\left[0,c-\frac{c}{N}\right]\times\left[0,\frac{c}{N}\right]^{d-1}$ and
	\begin{linenomath}
	\begin{equation*}
	A=\left\{\mb{x}\in R\colon \mb{x}\text{ satisfies \eqref{eq:translation} with $\varepsilon=\theta+\frac{2L}{N}$}\right\}.
	\end{equation*}
	\end{linenomath}
	Using \eqref{eq:gdpts} and the fact that statement~\eqref{dich1ind} holds for every $s\in[0,c/N]$ we deduce
	\begin{linenomath}
	\begin{equation*}
	\leb^{d-1}(A\cap \left\{\mb{x}\colon x_{d}=s\right\})\geq (1-\theta)\leb^{d-1}(R\cap \left\{\mb{x}\colon x_{d}=s\right\})\qquad \text{for all $s\in[0,c/N]$}.
	\end{equation*}
	\end{linenomath}
	Therefore, by Fubini's theorem,
	\begin{linenomath}
	\begin{equation*}
	\leb^{d}(A)\geq(1-\theta)\leb^{d}(R).
	\end{equation*}
	\end{linenomath}
	For each $i\in[N-1]$ we let $S_{i}:=\left[\frac{(i-1)c}{N},\frac{ic}{N}\right]\times\left[0,\frac{c}{N}\right]^{d-1}$, define
	\begin{linenomath}
	\begin{equation*}
	\Omega=\left\{i\in[N-1]\colon \leb^{d}(A\cap S_{i})\geq (1-\sqrt{\theta})\leb^{d}(S_{i})\right\}
	\end{equation*}
	\end{linenomath}
	and observe that
	\begin{linenomath}
	\begin{align*}
	\leb^{d}(A)\leq \left|\Omega\right|\frac{\leb^{d}(R)}{N-1}+(N-1-\left|\Omega\right|)(1-\sqrt{\theta})\frac{\leb^{d}(R)}{N-1}.
	\end{align*}
	\end{linenomath}
	Combining the two inequalities derived above for $\leb^{d}(A)$, we deduce
	\begin{linenomath}
	\begin{equation*}
	\frac{\left|\Omega\right|}{N-1}\geq (1-\sqrt{\theta}).
	\end{equation*}
	\end{linenomath}
	Moreover, for any $i\in\Omega$ and any cube $Q\subseteq S_{i}$ with side length $(2\sqrt{\theta}\leb^{d}(S_{i}))^{\frac{1}{d}}$ we have $A\cap Q\neq \emptyset$. Therefore, for any $i\in\Omega$ and any $\mb{x}\in S_{i}$ we can find $\mb{x}'\in A\cap S_{i}$ with
	\begin{linenomath}
	\begin{equation*}
	\left\|\mb{x}'-\mb{x}\right\|_{2}\leq \sqrt{d}(2\sqrt{\theta}\leb^{d}(S_{i}))^{\frac{1}{d}}\leq\frac{2\sqrt{d}\theta^{1/2d}c}{N}.
	\end{equation*}
	\end{linenomath}
	Using this approximation, we obtain
	\begin{linenomath}
	\begin{align*}
	&\left\|h\left(\mb{x}+\frac{c}{N}\mb{e}_{1}\right)-h(\mb{x})-\frac{1}{N}(h(c\mb{e}_{1})-h(\mb{0}))\right\|_{2}\\
	&\leq \left\|h\left(\mb{x}+\frac{c}{N}\mb{e}_{1}\right)-h\left(\mb{x}'+\frac{c}{N}\mb{e}_{1}\right)\right\|_{2}\\&+\left\|h\left(\mb{x}'+\frac{c}{N}\mb{e}_{1}\right)-h(\mb{x}')-\frac{1}{N}(h(c\mb{e}_{1})-h(\mb{0}))\right\|_{2}+\left\|h(\mb{x}')-h(\mb{x})\right\|_{2}\\
	&\leq 2L\left\|\mb{x}'-\mb{x}\right\|_{2}+\frac{c(\theta+\frac{2L}{N})}{N}\leq \frac{c(4L\sqrt{d}\theta^{1/2d}+\theta+\frac{2L}{N})}{N}.
	\end{align*}
		\end{linenomath}
	Thus, statement~\ref{dich1} is verified when we prescribe that $\theta>0$ is sufficiently small and $N_{0}$ is sufficiently large so that
	\begin{linenomath}
	\begin{equation*}
	(1-\sqrt{\theta})\geq 1-\varepsilon\quad\text{and}\quad 4L\sqrt{d}\theta^{1/2d}+\theta+\frac{2L}{N_{0}}<\varepsilon.
	\end{equation*} 
	\end{linenomath}
\end{proof}

\lemmaterminate*
\begin{proof}
	The appropriate condition on the parameter $r=r(d,L,\varepsilon)\in\N$ will be determined later in the proof. Let $\varphi=\varphi(d,L,\varepsilon)$ be the parameter given by the conclusion of Lemma~\ref{lemma:dichotomy}. We implement the following algorithm.
	\begin{algorithm}\label{algorithm}
		Set $i=1$, $\mb{z}_{1}=\mb{0}$ and $g_{1}=g$.
		\begin{enumerate}
			\item\label{step1} If statement $1$ of Lemma~\ref{lemma:dichotomy} holds for $h=g_{i}$ and $c=c_{i}$ then stop. If not proceed to step 2.
			\item\label{step2} Choose $\mb{z}_{i+1}\in c_{i+1}\Z^{d}\cap[0,c_{i}-c_{i+1}]\times[0,\frac{c_{i}}{N}-c_{i+1}]^{d-1}$ such that
			\begin{linenomath}
			\begin{equation}\label{eq:2i}
			\frac{\left\|g_{i}(\mb{z}_{i+1}+c_{i+1}\mb{e}_{1})-g_{i}(\mb{z}_{i+1})\right\|_{2}}{c_{i+1}}>(1+\varphi)\frac{\left\|g_{i}(c_{i}\mb{e}_{1})-g_{i}(\mb{0})\right\|_{2}}{c_{i}}
			\end{equation}	 
			\end{linenomath}
			and define $g_{i+1}\colon [0,c_{i+1}]\times[0,c_{i+1}/N]^{d-1}\to\R^{kd}$ by
			\begin{linenomath}
			\begin{equation*}
			g_{i+1}(\mb{x})=g_{i}(\mb{x}+\mb{z}_{i+1})=g\left(\mb{x}+\sum_{j=1}^{i+1}\mb{z}_{j}\right)
			\end{equation*}
			\end{linenomath}
			\item Set $i=i+1$ and return to step 1.
		\end{enumerate}
	\end{algorithm}
	At each potential iteration $i\geq 1$ of Algorithm~\ref{algorithm}, the conditions of Lemma~\ref{lemma:dichotomy} are satisfied for $d$, $L$, $\varepsilon$, $M$, $\varphi$, $N_{0}$, $c=c_{i}$, $n=kd$, $N$ and $h=g_{i}\colon [0,c_{i}]\times [0,c_{i}/N]^{d-1}\to \R^{kd}$. Therefore, whenever the algorithm does not terminate in step~\ref{step1}, we have that such a point $\mb{z}_{i+1}$ required by step~\ref{step2} exists by Lemma~\ref{lemma:dichotomy}. 

	To complete the proof, it suffices to verify that Algorithm~\ref{algorithm} terminates after at most $r$ iterations. This is clear, after rewriting \eqref{eq:2i} in the form
	\begin{linenomath}
	\begin{equation*}
	\frac{\lnorm{2}{g_{i+1}(c_{i+1}\mb{e}_{1})-g_{i+1}(\mb{0})}}{c_{i+1}}>(1+\varphi)\frac{\left\|g_{i}(c_{i}\mb{e}_{1})-g_{i}(\mb{0})\right\|_{2}}{c_{i}}>(1+\varphi)^{i}\frac{1}{L},
	\end{equation*}
	\end{linenomath}
	where the latter inequality follows by induction and the lower bilipschitz inequality for $g=g_{1}$. If Algorithm~\ref{algorithm} completes $r+1$ iterations then, for $r>\frac{2\log L}{\log (1+\varphi)}$, the above inequality with $i=r$ contradicts the $L$-Lipschitz condition on $g$.
\end{proof}
The next lemma is an intuitively clear statement: It says that if two $L$-bilipschitz mappings defined on a cube $S\in\mathcal{Q}_{\lambda}^{d}$ are close in supremum norm, then the volumes of their images are also close. 
\begin{lemma}\label{lemma:bilcubvol}
Let $\lambda>0$, $S\in \mathcal{Q}^{d}_{\lambda}$, $L\geq 1$ and $f_{1},f_{2}\colon S\to \R^{d}$ be $L$-bilipschitz mappings. Let $\varepsilon\in(0,1/2L)$ and suppose that
\begin{linenomath}
\begin{equation}\label{eq:translation2}
\left\|f_{2}(\mb{x})-f_{1}(\mb{x})\right\|_{\infty}\leq\varepsilon\lambda.
\end{equation}
\end{linenomath}
Then
\begin{linenomath}
\begin{equation*}
\left|\leb (f_{1}(S))-\leb (f_{2}(S))\right|\leq 2L^{d+1}d\varepsilon\leb (S).
\end{equation*}
\end{linenomath}
\end{lemma}
\begin{proof}
	For a set $A\subseteq \R^{d}$ and $t>0$ we introduce the set
	\begin{linenomath}
	\begin{equation*}
	[A]_{t}:=\left\{\mb{x}\in A\colon \dist(\mb{x},\partial A)\geq t\right\}
	\end{equation*}
	\end{linenomath}
	of all points in the interior of $A$, whose distance to the boundary of $A$ is at least $t$.
	Using \eqref{eq:translation2} and the lower bilipschitz bound on $f_{2}$ we deduce that
	\begin{linenomath}
	\begin{equation*}
	f_{1}([S]_{t})\subseteq\overline{B}(f_{2}([S]_{t}),\varepsilon\lambda)\subseteq \overline{B}([f_{2}(S)]_{t/L},\varepsilon\lambda)
	\end{equation*}
	\end{linenomath}
	for all $t>0$. For the second inclusion, we use Brouwer's Invariance of Domain~\cite[Thm.\ 2B.3]{Hatcher}\footnote{Brouwer's Invariance of Domain is used here in order to derive that bilipschitz mappings defined on a subset of $\R^{d}$ and taking values in $\R^{d}$ are open and therefore preserve boundaries. In general Hilbert spaces this is not true: The mapping $\ell_{2}\to\ell_{2}$, $(x_{1},x_{2},\ldots)\mapsto(0,x_{1},x_{2},\ldots)$ is an isometry but the image of the whole space $\ell_{2}$ under this mapping has empty interior.} in order to prove $f_{2}([S]_{t})\subseteq [f_{2}(S)]_{t/L}$. It follows that
	\begin{linenomath}
	\begin{equation*}\label{eq:ginc}
	f_{1}([S]_{L\varepsilon\lambda})\subseteq f_{2}(S).
	\end{equation*}
	\end{linenomath}
	Therefore
	\begin{linenomath}
	\begin{align*}
	\leb (f_{1}(S))-\leb (f_{2}(S))\leq \leb (f_{1}(S\setminus [S]_{L\varepsilon\lambda}))\leq L^{d}\leb (S\setminus [S]_{L\varepsilon\lambda}).
	\end{align*}
	\end{linenomath}
	The Lebesgue measure of the set $S\setminus [S]_{L\varepsilon\lambda}$ can be easily computed using the fact that $[S]_{L\varepsilon\lambda}$ is a cube of side length $\lambda(1-2L\varepsilon)$:
	\begin{linenomath}
	\begin{align*}
	\leb (S\setminus [S]_{L\varepsilon\lambda})=\leb (S)-\leb ([S]_{L\varepsilon\lambda})&=\lambda^{d}-\lambda^{d}(1-2L\varepsilon)^{d}\\
	&\leq \lambda^{d}2dL\varepsilon=2dL\varepsilon\leb (S).
	\end{align*}
	\end{linenomath}
	For the inequality we use $2L\varepsilon\in(0,1)$ and apply Bernoulli's inequality. We conclude that
	\begin{linenomath}
	\begin{equation*}
	\leb (f_{1}(S))-\leb (f_{2}(S))\leq 2dL^{d+1}\varepsilon\leb (S).
	\end{equation*}
	\end{linenomath}
	Since the above argument is completely symmetric with respect to $f_{1}$ and $f_{2}$, we also have
	\begin{linenomath}
	\begin{equation*}
	\leb (f_{2}(S))-\leb (f_{1}(S))\leq 2dL^{d+1}\varepsilon\leb (S).
	\end{equation*}
	\end{linenomath}
\end{proof}
\section{Appendix to Section~\ref{section:realizability}: Realisability in spaces of functions.}
\label{app:realisability}

\lemmapeturb*
\begin{proof}
	Let $r\in\N$ and the finite, tiled families $\Sq_{1},\Sq_{2},\ldots,\Sq_{r}$ of cubes in $U$ be given by the conclusion of Lemma~\ref{lemma:geometric} applied to $d$, $k$, $L$, $\zeta$ and $\eta\in(0,1)$, where $\eta$ is a parameter to be determined later in the proof. We will now define a sequence $\psi_{0},\psi_{1},\psi_{2},\ldots,\psi_{r}$ of continuous functions on $U$. The sought after function $\psi$ will then be defined on $U$ by $\psi|_{U}=\psi_{r}$.
	
	We begin by setting $\psi_{0}=0$. If $i\geq 0$ and $\psi_{i}$ is already constructed, we define $\psi_{i+1}$ as any continuous function on $U$ with the following properties:
	\begin{enumerate}[(i)]
		\item\label{psi1} $\psi_{i}=\psi_{i-1}$ outside of $\bigcup \Sq_{i}$.
		\item\label{psi2} $-\varepsilon\leq \psi_{i}\leq\varepsilon$.
		\item\label{psi3} For every cube $S\in \Sq_{i}$, $\dashint_{S}\psi_{i}\in\left\{-8\varepsilon/9,8\varepsilon/9\right\}$.
		\item\label{psi4} For every pair of $\mb{e}_{1}$-adjacent cubes $S,S'\in\Sq_{i}$, $\dashint_{S}\psi_{i}\neq \dashint_{S'}\psi_{i}$.	
	\end{enumerate}
	It is clear that such a continuous function exists; see Figure~\ref{f:nonrealisable} for an example. Note that conditions \eqref{psi3} and \eqref{psi4} just prescribe that the average values of $\psi_{i}$ on the cubes in $\Sq_{i}$ follow a `chessboard' pattern.
	
\begin{figure}[htb]
\begin{center}
\includegraphics{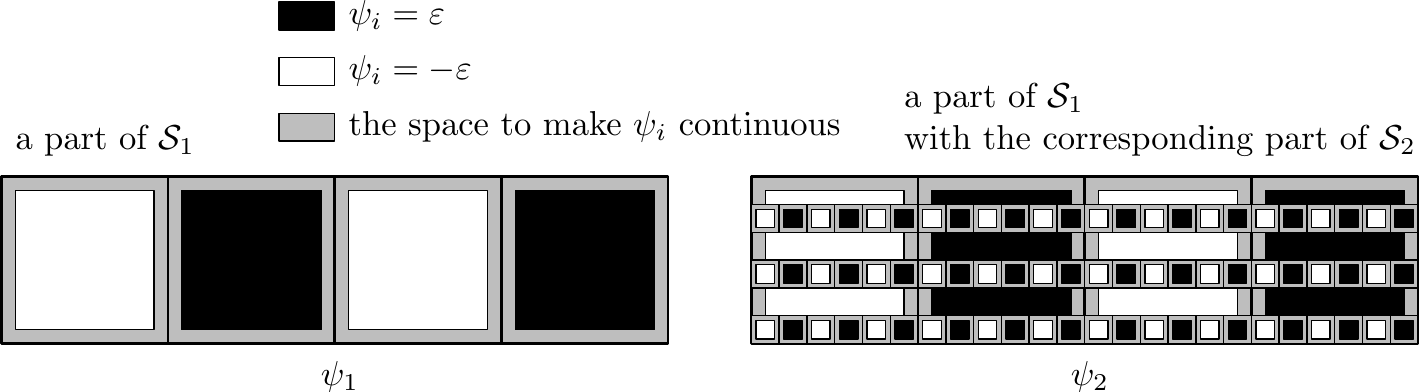}
\caption{An example of the construction of the sequence of functions $\psi_1,\ldots,\psi_r$.}
\label{f:nonrealisable}
\end{center}
\end{figure}

	The final function $\psi_{r}$ clearly satisfies $-\varepsilon\leq \psi_{r}\leq\varepsilon$ and $\psi_{r}=0$ outside of $\bigcup_{i=1}^{r}\bigcup \Sq_{i}$. Moreover, for any $i\in[r]$ and $\mb{e}_{1}$-adjacent cubes $S,S'\in \Sq_{i}$ we will prove that
	\begin{linenomath}
	\begin{equation}\label{eq:oldpropc}
	\left|\dashint_{S}\psi_{r}-\dashint_{S'}\psi_{r}\right|\geq\varepsilon.
	\end{equation}
	\end{linenomath}
	Fix $i\in[r]$, $\mb{e}_{1}$-adjacent cubes $S,S'\in\Sq_{i}$ and combine \eqref{psi3} and \eqref{psi4} to obtain $\left|\dashint_{S}\psi_{i}-\dashint_{S'}\psi_{i}\right|=16\varepsilon/9$. Letting $T:=S\cap\bigcup_{j=i+1}^{r}\bigcup\Sq_{j}$ and $T':=S'\cap\bigcup_{j=i+1}^{r}\bigcup\Sq_{j}$, we have from conclusion~\ref{lemma:geometric1} of Lemma~\ref{lemma:geometric} that $\max\left\{\leb(T),\leb(T')\right\}\leq\eta\leb(S)$. Moreover, condition~\eqref{psi1} in the construction above guarantees that $\psi_{r}=\psi_{i}$ on $(S\setminus T)\cup(S'\setminus T')$. We conclude that
	\begin{linenomath}
	\begin{align*}
	\left|\dashint_{S}\psi_{r}-\dashint_{S'}\psi_{r}\right|&\geq \left|\dashint_{S}\psi_{i}-\dashint_{S'}\psi_{i}\right|-\left|\frac{1}{\leb(S)}\int_{T}(\psi_{i}-\psi_{r})-\frac{1}{\leb(S)}\int_{T'}(\psi_{i}-\psi_{r})\right|\\
	&\geq \frac{16\varepsilon}{9}-\frac{2\left\|\psi_{i}-\psi_{r}\right\|_{\infty}\max\left\{\leb(T),\leb(T')\right\}}{\leb(S)}\geq \frac{16\varepsilon}{9}-4\varepsilon\eta.
	\end{align*}
	\end{linenomath}
	Thus, setting $\eta=1/9$, we verify \eqref{eq:oldpropc}. To complete the proof, it now only remains to extend the function $\psi$ to the whole of $I^{d}$ by setting $\psi=0$ outside of $U$ so that $\supp(\psi)\subseteq U$.
\end{proof}

\Linflemma*
\begin{proof}
	We define the function $\psi$ on $I^{d}$ inductively as follows. Pick any $S_{1}\in\Sq$ such that the first co-ordinate projection map $\pi_{1}\colon \bigcup\Sq\to\R$ attains its minimum on $S_{1}$. This ensures that $S_{1}\neq T+\lambda\mb{e}_{1}$ for any cube $T\in\Sq$. We set $\psi=\rho$ on $S_{1}$. If $n\geq 1$, distinct cubes $S_{1},\ldots,S_{n}\in\Sq$ and the function $\psi|_{\bigcup_{i=1}^{n}S_{i}}$ are defined, we extend $\psi$ as follows: If $\Sq\setminus\left\{S_{1},\ldots,S_{n}\right\}=\emptyset$, we complete the construction of $\psi$ by setting $\psi=\rho$ on $I^{d}\setminus\bigcup \Sq$. Otherwise, we choose, if possible, $S_{n+1}\in\Sq\setminus\left\{S_{1},\ldots,S_{n}\right\}$ such that $S_{n+1}=S_{n}+\lambda\mb{e}_{1}$. If this is not possible we choose $S_{n+1}\in \Sq\setminus\left\{S_{1},\ldots,S_{n}\right\}$ arbitrarily such that the mapping $\pi_{1}\colon \bigcup\Sq\setminus\bigcup_{i=1}^{n}S_{i}\to\R$ attains its minimum on $S_{n+1}$. In the first case, we define $\psi$ on $S_{n+1}$ by
	\begin{linenomath}
	\begin{equation*}
	\psi=\begin{cases}
	\rho & \text{if }\left|\dashint_{S_{n+1}}\rho-\dashint_{S_{n}}\psi\right|\geq \varepsilon\\
	\rho+\varepsilon & \text{if }\dashint_{S_{n+1}}\rho-\dashint_{S_{n}}\psi\in(0,\varepsilon),\\
	\rho-\varepsilon & \text{if }\dashint_{S_{n+1}}\rho-\dashint_{S_{n}}\psi\in (-\varepsilon,0).
	\end{cases}
	\end{equation*}
	\end{linenomath}
	In the latter case, we may simply take $\psi=\rho$ on $S_{n+1}$. It is now readily verified that the function $\psi\in L^{\infty}(I^{d})$ produced by this construction possesses all of the required properties.
\end{proof}
\section{Appendix to Section~\ref{section:feige}: Feige's question.}\label{app:feige}
\begin{obs}\label{obs:Feige_equiv}
Question~\ref{q:Feige} is equivalent to Question~\ref{q:Feige2}.
\end{obs}
\begin{proof}
A negative answer to Question~\ref{q:Feige} trivially provides a negative answer to Question~\ref{q:Feige2}. Thus, we focus on the opposite direction.

Let $r>0$ be such that there is no $L(r)>0$ as in Question~\ref{q:Feige2} in dimension $d\in\N$, $d\geq 2$. Let $n\in\N$ and $S\subset\R^d$ be an $r$-separated set of cardinality $n^d$. We consider a linear mapping $h\colon\R^d\to\R^d$ defined as $h(x):=\frac{d}{r}x$. Let us write $S':=h(S)$ for a copy of $S$ scaled by the factor $\frac{d}{r}$. The set $S'$ is $d$-separated.
	
	Next, we define a mapping $z\colon S'\to \Z^d$ as
	\begin{linenomath}
	$$
	z(x):= \argmin_{x'\in\Z^d} \lnorm{2}{x-x'}.
	$$
	\end{linenomath}
	If there is more than one such point, we choose one of them arbitrarily. Since $S'$ is $d$-separated, for every $x,y\in S'$ the points $z(x)$ and $z(y)$ are distinct whenever $x\neq y$. We form a set $S'':=z(S')\subset\Z^d$. 
	
	Now we assume, for contradiction, that there is $L>0$ as in Question~\ref{q:Feige}. Therefore, there is an $L$-Lipschitz bijection $f\colon S''\to[n]^d$. The mapping $f\circ z$ is clearly a bijection, but it is also Lipschitz:
	\begin{linenomath}
	$$
	\lnorm{2}{f\circ z(x)-f\circ z(y)}\leq L\cdot \lnorm{2}{z(x)-z(y)}\leq L\cdot (\lnorm{2}{x-y}+\sqrt{d})\leq L\br*{1+\sqrt{d}}\lnorm{2}{x-y}
	$$
	\end{linenomath}
	whenever $x,y\in S'$ are two distinct points. Consequently, the mapping $f\circ z\circ h$ defines an $\frac{Ld\br*{1+\sqrt{d}}}{r}$-Lipschitz bijection $S\to[n]^d$. Since the last Lipschitz constant is not dependent on the original choice of $r$-separated set $S$, this is a contradiction.
\end{proof}

\lemmaweakconvcrit*
\begin{proof}
Fix any $\psi\in C(K;\R)$ and $\varepsilon>0$.
For every $Q\in\mc Q_n$ we choose $z_Q\in Q$ arbitrarily. By the uniform continuity of $\psi$, there is $N\in\N$ such that for every $n\geq N$, every $Q\in\mc Q_n$ and every $x\in Q$ we have $\abs{\psi(z_Q)-\psi(x)}\leq\varepsilon$. Moreover, we require that for every $n\geq N$ it holds that $\sup_{Q\in\mc{Q}_{n}}\abs{\nu_{n}(Q)-\nu(Q)}\leq\frac{\varepsilon}{\abs{\mc Q_n}}$. We write
\begin{linenomath}
\begin{align*}
&\niceint{Q}{\psi}{\nu_n}\leq\niceint{Q}{\psi(z_Q)+\varepsilon}{\nu_n}=\br*{\psi(z_Q)+\varepsilon}\nu_n(Q)\leq
\br*{\psi(z_Q)+\varepsilon}\br*{\nu(Q)+\frac{\varepsilon}{\abs{Q_n}}}\\
=&\niceint{Q}{\br*{\psi(z_Q)+\varepsilon}}{\nu} + \br*{\psi(z_Q)+\varepsilon}\frac{\varepsilon}{\abs{Q_n}}\leq\niceint{Q}{\br*{\psi+2\varepsilon}}{\nu}+\br*{\psi(z_Q)+\varepsilon}\frac{\varepsilon}{\abs{Q_n}}\\
\leq&\niceint{Q}{\psi}{\nu}+2\varepsilon\nu(Q)+\br*{\psi(z_Q)+\varepsilon}\frac{\varepsilon}{\abs{Q_n}}.
\end{align*}
\end{linenomath}

Symmetrically, we derive the lower bound
\begin{linenomath}
\begin{align*}
\niceint{Q}{\psi}{\nu_n}\geq\niceint{Q}{\psi}{\nu}-2\varepsilon\nu(Q)-\br*{\psi(z_Q)-\varepsilon}\frac{\varepsilon}{\abs{Q_n}}.
\end{align*}
\end{linenomath}

Summing over all elements of $\mc Q_n$ and using the assumption that elements of $\mc Q_n$ form $\nu$-almost disjoint cover of $K$ we get
\begin{linenomath}
\begin{align*}
\abs{\niceint{K}{\psi}{\nu_n}-\niceint{K}{\psi}{\nu}}\leq2\varepsilon\nu(K)+\varepsilon\br*{\max\abs{\psi}+\varepsilon}.
\end{align*}
\end{linenomath}
Since $\psi$ is continuous on a compact $K$ it is also bounded and the right hand side of the last inequality tends to zero with $\varepsilon$.
\end{proof}

\lemmaweakconvcrittwo*
\begin{proof}
We take any $\psi\in C(X,\R)$ with compact support. We can bound
\begin{linenomath}
\begin{align*}
&\abs{\niceint{h_n(K)}{\psi}{(h_n)_\sharp(\nu_n)}-\niceint{h(K)}{\psi}{}h_\sharp(\nu)}=
\abs{\niceint{K}{\psi\circ h_n}{\nu_n}-\niceint{K}{\psi\circ h}{\nu}}\\
\leq&\niceint{K}{\abs{\psi\circ h_n-\psi\circ h}}{\nu_n}+\abs{\niceint{K}{\psi\circ h}{\nu_n}-\niceint{K}{\psi\circ h}{\nu}}.
\end{align*}
\end{linenomath}
As $n\to\infty$ the first term in the final sum goes to zero since $\psi\circ h_n$ converges uniformly to $\psi\circ h$. Moreover, the second term converges to zero as well, because $\nu_{n}$ converges weakly to $\nu$. 
\end{proof}	
\bibliographystyle{plain}
\bibliography{citations}

\bigskip

\samepage{
\noindent Michael Dymond\\
Institut für Mathematik\\
Universität Innsbruck\\
Technikerstraße 13,
6020 Innsbruck,
Austria\\
\texttt{michael.dymond@uibk.ac.at}\\[3mm]}

\noindent
\begin{tabular}{@{}l c l}
Vojtěch Kaluža & & \\
Katedra Aplikované Matematiky &  & Institut für Mathematik\\
Univerzita Karlova & & Universität Innsbruck\\
Malostranské nám. 25, 118 00 Praha 1, & \quad\textit{\&}\quad\quad & Technikerstraße 13, 6020 Innsbruck,\\
Czech Republic & & Austria\\
\texttt{kaluza@kam.mff.cuni.cz}& & \\
\end{tabular}
\\[3mm]

\samepage{
\noindent Eva Kopeck\' a\\
Institut für Mathematik\\
Universität Innsbruck\\
Technikerstraße 13,
6020 Innsbruck,
Austria\\
\texttt{eva.kopecka@uibk.ac.at}}
\end{document}